\pdfoutput=1

\documentclass[12pt,a4paper,reqno]{amsart}
\usepackage[utf8]{inputenc}
\usepackage[T1]{fontenc}
\usepackage{indentfirst}
\usepackage{textcomp,lscape,rotating}
\usepackage[english]{babel}
\usepackage{enumerate}
\usepackage{amsmath,amsfonts,amssymb,amsopn,amscd,amsthm}
\usepackage{stmaryrd}
\usepackage[bbgreekl]{mathbbol}
\usepackage{mathrsfs}
\usepackage{graphicx}
\usepackage[dvipsnames]{xcolor}
\usepackage[colorlinks=true,citecolor=DarkOrchid,linkcolor=NavyBlue]{hyperref}

\newtheorem{definition}{Definition}[section]
\newtheorem{lemma}[definition]{Lemma}
\newtheorem{theorem}[definition]{Theorem}
\newtheorem{proposition}[definition]{Proposition}
\newtheorem{corollary}[definition]{Corollary}
\theoremstyle{remark}
\newtheorem{example}[definition]{Example}
\newtheorem{remark}[definition]{Remark}
\newtheorem{problem}[definition]{Problem}

\newcommand{\N}{\mathbb{N}}
\newcommand{\Z}{\mathbb{Z}}

\newcommand{\C}{\mathbb{C}}
\newcommand{\Q}{\mathbb{Q}}

\newcommand{\id}{\mathrm{id}}
\newcommand{\GL}{\mathrm{GL}}
\newcommand{\matr}{\mathrm{mat}}
\newcommand{\card}{\mathrm{card}}
\newcommand{\For}{\mathbb{F}}
\newcommand{\ym}{\mathfrak{P}}
\newcommand{\sym}{\mathfrak{S}}
\newcommand{\uni}{\mathbb{U}}
\newcommand{\buni}{\mathbb{U\!F}}
\newcommand{\proba}{\mathbb{P}}
\newcommand{\esper}{\mathbb{E}}
\newcommand{\lle}{\left[\!\left[} 
\newcommand{\Dc}{\mathcal{D}}
\newcommand{\Ec}{\mathcal{E}}
\newcommand{\Fc}{\mathcal{F}}
\newcommand{\Gc}{\mathcal{G}}
\newcommand{\Hc}{\mathcal{H}}
\newcommand{\rre}{\right]\!\right]} 
\newcommand{\opp}{^{\mathrm{opp}}}
\newcommand{\Vect}{\mathrm{Span}}
\newcommand{\jacobi}[2]{(\frac{#1}{#2})}

\newcommand{\comment}[1]{}
\newcommand{\partiso}[4]{(#1 \,\,^{#3}\!\!\rightleftarrows_{#4} #2)}
\newcommand{\partext}[4]{(#1 \,\,^{#3}\!\!\rightleftarrows {\scriptstyle \!#4}\,\, \normalsize  #2)}

\begin{document}
\title{Partial isomorphisms over finite fields}
\author{Pierre-Lo\"ic M\'eliot}
\address{Laboratoire de Math\'ematiques -- B\^atiment 425 -- Facult\'e des Sciences d'Orsay -- Universit\'e Paris-Sud -- F-91405 ORSAY -- France}
\email{pierre-loic.meliot@math.u-psud.fr}
\date{\today}

\begin{abstract}
In this paper, we construct a combinatorial algebra of partial isomorphisms that gives rise to a ``projective limit'' of the centers of the group algebras $\C\GL(n,\For_{q})$. It allows us to prove a $\GL(n,\For_{q})$-analogue of a theorem of Farahat and Higman regarding products of conjugacy classes of permutations.
\end{abstract}

\maketitle

\hrule
\setcounter{tocdepth}{2}
\tableofcontents
\hrule
\bigskip

\section{Introduction}
In this paper, $n$ is a positive integer; $q$ is a prime power; $\For_{q}$ is a finite field with $q$ elements; and $\GL(n,\For_{q})$ is the group of invertible $n\times n$ matrices with coefficients in $\For_{q}$, or, equivalently, the group of linear isomorphisms of the $\For_{q}$-vector space $(\For_{q})^{n}$. It will be convenient to write $gh$ for the composition of isomorphisms $h \circ g$ (beware of the order of composition). The matrix of an isomorphism $u : V \to W$ between two $(\For_{q})$-vector spaces and with respect to two bases $\Ec=(e_{1},\ldots,e_{n})$ of $V$ and $\Fc=(f_{1},\ldots,f_{n})$ of $W$ is
$$\matr_{\Ec,\Fc}(u)=\matr_{\Fc}(u(e_{1}),\ldots,u(e_{n})),$$
the vectors being written in columns; then $\matr_{\Ec,\Gc}(uv)=\matr_{\Fc,\Gc}(v)\,\matr_{\Ec,\Fc}(u)$. \bigskip

\subsection{Generic products in group algebras and the approach of Ivanov and Kerov}
If $G= \GL(n,\For_{q})$, we shall be interested in the group algebra $\C G$, and more precisely in its center $Z(\C G)$. A linear basis of $Z(\C G)$ is the set of conjugacy classes of $G$, viewed as the formal sums of their elements. We want to address the following kind of problem:
\begin{problem}
Let $a \neq b \neq 1$ be two elements of $(\For_{q})^{\times}$, and $C_{a,n}$ and $C_{b,n}$ the conjugacy classes of the diagonal matrices of size $n$
$$D_{a}=\begin{pmatrix} a & & &\\
& 1 &            & \\
 &   & \ddots &\\
 &   &           & 1\end{pmatrix} \qquad\text{and}\qquad D_{b}=\begin{pmatrix} b & & &\\
& 1 &            & \\
 &   & \ddots &\\
 &   &           & 1\end{pmatrix}. $$
What is the expansion $\sum_{\lambda \in \Lambda(a,b)} c_{a,b,n}^{\lambda}\,C_{\lambda}$ in conjugacy classes of the product $C_{a,n}*C_{b,n}$ computed in the center of the group algebra $Z(\C\GL(n,\For_{q}))$?
\end{problem}
\bigskip

At first sight, this problem might seem to be not so difficult: since $D_{a}$ and $D_{b}$ each leave invariant a subspace of dimension $n-1$, the product $P=AB$ of two matrices conjugated to $D_{a}$ and $D_{b}$ has to leave invariant a subspace of dimension at least $n-2$, and as we shall see in a moment, this only leaves a few possibilities for the conjugacy class $\lambda$ of $P$. Moreover, by looking at how many ways the eigenspaces of the matrices can intersect, one can guess that the coefficients $c_{a,b,n}^{\lambda}$ count configurations of subspaces in $(\For_{q})^{n}$, and are therefore rational functions in $q$ and in its powers $q^{k}$, up to $k=n$.\bigskip

However, it is in practice quite hard to determine the set of possible conjugacy classes $\Lambda(a,b)$; and also quite hard to compute the structure coefficients. The reader can already have a look at Theorem \ref{productdegree1}, to realize that the result depends in particular on:\vspace{2mm}
\begin{itemize}
\item whether $q$ is even or odd;\vspace{2mm}
\item in the odd case, whether $ab$ is a square in $(\For_{q})^{\times}$ or not.\vspace{2mm}
\end{itemize}  
Also, it is \emph{a priori} unclear that the coefficients $c_{a,b,n}^{\lambda}$ are in fact all polynomials in $q^{n}$, with rational coefficients independent of $n$. One might understand intuitively why this is true from the previous informal discussion: every number of subspaces of fixed dimension $k$ in $(\For_{q})^{n}$ is a polynomial in $q^{n}$ (of degree $k$), and the same holds when counting ``finite-dimensional configurations''. However, a formal and rigorous proof demands a lot of combinatorial preparations, especially if one wants to \emph{compute} the actual polynomials. The author conjectured this polynomiality a few years ago, and this will be one of the major result of the paper. We shall prove it in the general setting where $D_{a}$ and $D_{b}$ are replaced by arbitrary matrices $A$ and $B$ that are completed by $1$'s on the diagonal to obtain matrices of size $n$. We shall also provide a general framework which reduces all the computations to the case when $n=k+l$, $k$ and $l$ being the sizes of $A$ and $B$ ($k=l=1$ for $A=D_{a}$ and $B=D_{b}$). In other words, knowing the structure coefficients $c_{A,B,k+l}^{\lambda}$, we shall have at the end of the paper an easy rule to compute all the coefficients $c_{A,B,n}^{\lambda}$. \bigskip

A similar and simpler problem has been solved for permutations by Farahat and Higman in \cite{FH59}, and forty years later, it was given a beautiful explanation by Ivanov and Kerov, see \cite{IK99}. Fix two permutations $\sigma$ and $\rho$ of size $k$ and $l$; we also assume $\sigma$ and $\rho$ without fixed points in $\lle 1,k\rre$ and in $\lle 1,l\rre$. For any $n \geq \max(k,l)$, one can view $\sigma$ and $\rho$ as elements of the symmetric group of order $n$, denoted $\sym(n)$; $\sigma$ fixes the integers after $k$ and $\rho$ fixes the integer after $l$.  Denote $C_{\sigma,n}$ and $C_{\rho,n}$ the conjugacy classes of $\sigma$ and $\rho$ in $\C\sym(n)$. Then, there exists a unique finite set of permutations $S(\sigma,\rho) \subset \sym(k+l)$, and polynomials $p_{\sigma,\rho}^{\nu}(n)$ with rational coefficients and integer values at integers, such that
$$C_{\sigma,n}*C_{\rho,n}=\sum_{\nu \in S(\sigma,\rho)} p^{\nu}_{\sigma,\rho}(n)\,C_{\nu,n}.$$
Using the well-known labeling of conjugacy classes of $\sym(n)$ by integer partitions of size $n$, one can of course restate this result with two partitions $\lambda$ and $\mu$ of size $k$ and $l$ and without parts of size $1$, and a finite set of partitions of size smaller than $k+l$ and again without parts of size $1$:\label{fhig}
$$C_{\lambda\sqcup 1^{n-k}}*C_{\mu \sqcup 1^{n-l}}=\sum_{|\rho| \leq k+l} p^{\rho}_{\lambda,\mu}(n)\,C_{\rho \sqcup 1^{n-|\rho|}}.$$

\bigskip

The proof of Ivanov and Kerov uses the following idea: by manipulating \emph{partial permutations} that are permutations with a distinguished support $A \subset \lle 1,n\rre$, one can construct an inverse system of graded algebras 
$$\begin{CD} \cdots @>>>\mathscr{Z}(n+2)  @>>>\mathscr{Z}(n+1)  @>>>\mathscr{Z}(n) @>>> \cdots \end{CD}\,\,,$$
such that:\vspace{2mm}
\begin{enumerate}
\item each algebra $\mathscr{Z}(n)$ projects onto $Z(\C\sym(n))$ via a morphism $\pi_{n}$;\vspace{2mm} 
\item each algebra has a basis $(A_{\lambda,n})_{\lambda}$ labelled by partitions $\lambda$ of size smaller than $n$, with $\pi_{n}(A_{\lambda,n})=p_{\lambda}(n)\,C_{\lambda,n}$; $p_{\lambda}(n)$ is an explicit polynomial in $n$;\vspace{2mm}
\item the bases $A_{\lambda,n}$ are compatible, \emph{i.e.}, if $|\lambda|\leq n$, then the map $\mathscr{Z}(n+1)  \to \mathscr{Z}(n)$ sends $A_{\lambda,n+1}$ to $A_{\lambda,n}$.\vspace{2mm}
\end{enumerate}
From this construction, an identity
$$A_{\lambda}*A_{\mu}=\sum_{|\nu| \leq |\lambda|+|\mu|} c_{\lambda,\mu}^{\nu}\,A_{\nu}.$$
in the projective limit $\mathscr{Z}(\infty)=\varprojlim_{n \to \infty}\mathscr{Z}(n)$ gives immediately the identity
$$C_{\lambda,n}*C_{\mu,n}=\sum_{|\nu| \leq |\lambda|+|\mu|} c_{\lambda,\mu}^{\nu}\,\frac{p_{\nu}(n)}{p_{\lambda}(n)\,p_{\mu}(n)}\,C_{\nu,n}.$$
in all the centers of the group algebras $Z(\C\sym(n))$; this provides an easy proof of the theorem of Farahat and Higman. The same techniques and kind of results have been studied more recently by the author for product of Geck-Rouquier elements in the centers of the Hecke algebras $\mathscr{H}_{q}(\sym_{n})$ of the symmetric groups (\emph{cf.} \cite{Mel10}); and by O. Tout in \cite{Tout12} for product of classes in the algebras $\C[\mathfrak{H}(n)\backslash\sym(2n)/\mathfrak{H}(n)]$ associated to the Gelfand pairs $(\sym(2n),\mathfrak{H}(n))$. \bigskip

The objective of this paper is to follow the same program for general linear groups over finite fields; unfortunately, in this setting, there are a lot of complications in comparison to the framework previously described. The naive idea is of course to replace partial permutations by \emph{partial isomorphisms}; but there are two major obstacles to this idea.\vspace{2mm}
\begin{enumerate}
\item For permutations, if $\sigma \in \sym(A)$ with $A \subset \lle 1,n\rre$ of size $k$, then there is a unique canonical way to see $\sigma$ as an element of $\sym(n)$; as explained before, one leaves the integers outside $A$ fixed by $\sigma$. The same cannot be done in a canonical way for isomorphisms: if $g$ is an $\For_{q}$-isomorphism of a subspace $V \subset (\For_{q})^{n}$ of dimension $k$, then $g$ admits a lot of extensions to isomorphisms of $(\For_{q})^{n}$, related to the choice of complement subspaces of $V$ inside $(\For_{q})^{n}$. And there is no reason or possibility to distinguish a particular extension among those.\vspace{2mm}
\item Knowing that, a natural thing to do is to take the \emph{mean of all possible extensions}; it seems then possible to define correctly the product of partial isomorphisms. But this still does not work: as we shall explain later (see Remark \ref{naive}), the algebra that one obtains by this construction is not associative. \vspace{2mm}
\end{enumerate}
The solution to these problems, and the main idea of this paper, is to define partial isomorphisms as \emph{pairs} of isomorphisms between two subspaces.\bigskip

\begin{remark}
An alternative construction of partial bijections and partial isomorphisms has been proposed by G. Olshanski in \cite{Ols90,Ols91}. Let $G=\varinjlim_{n \to \infty} G(n)$ an inductive limit of finite groups, and $G\supset K(0) \supset K(1) \supset \cdots \supset K(n) \supset \cdots $ a chain of subgroups of $G$, such that $K(n)$ commutes with $G(n)$ for all $n$. For instance one can take 
\begin{align*}
G(n)&=\sym(n) \qquad;\qquad K(n)=\{\sigma \in \sym(\infty)\,\,|\,\,\sigma_{\lle 1,n\rre}=\id_{\lle 1, n \rre}\} ;\\
\text{or}\,\,G(n)&=\GL(n,\For_q)\qquad;\qquad K(n)=\{g \in \GL(\infty,\For_q)\,\,|\,\,g_{|(\For_q)^n}=\id_{(\For_q)^n}\}.
\end{align*}
 In these situations, there is a structure of semigroup on the set of double cosets $\Gamma(n)=K(n)\backslash G/K(n)$. For instance, when $G=\sym(\infty)$, $\Gamma(n)$ can be identified with the set of bijections $\sigma : A \to B$ between two parts $A$ and $B$ of $\lle 1,n\rre$, and the product is
 $$(\sigma : A \to B) * (\tau : C \to D) = (\sigma\tau : \sigma^{-1}(B\cap C) \to \tau(B \cap C)).$$
Notice that the product $*$ written above is quite different from the product of partial permutations defined in \cite{IK99}, or from the product of partial bijections defined in \cite{Tout12}. A similar construction of semigroup can be performed when $G=\GL(\infty,\For_q)$. However, these semigroups (or the sets of double cosets $K(n)\backslash G\times G/K(n)$, studied in Okounkov's thesis \cite{Oko95}), do not really help for the specific problem that we are looking at, namely, the understanding of the structure coefficients of the group algebra centers $\C G(n)$: indeed, there is no natural morphism of semigroups $\Gamma(n) \to G(n)$.\bigskip

Unfortunately, we do not see how to use \emph{semigroup algebras} $\mathscr{Z}(n,\For_q)$ in order to construct a projective limit of the centers $\C\GL(n,\For_q)$; we shall however be satisfied with \emph{combinatorial algebras} labelled by partial isomorphisms, without an underlying structure of semigroup on the partial isomorphisms.
 \end{remark}

\subsection{Conjugacy classes of matrices and polypartitions}
For finite general linear groups, the conjugacy classes are given by Jordan's reduction of matrices. If $P(X)=X^{n}+a_{n-1}X^{n-1}+\cdots+a_{0}$ is a monic polynomial, its Jordan matrix is the $n \times n$ matrix
$$J(P)=\begin{pmatrix}
 0 &   &           &  & -a_{0}\\
 1 &0 &            &  & -a_{1}\\
    &1 & \ddots  &  & \vdots \\
    &   & \ddots & 0 & -a_{n-2}\\
    &   &            & 1 & -a_{n-1}
\end{pmatrix}.$$
Call partition of size $k$ a non-increasing sequence of positive integers $(\mu_{1},\ldots,\mu_{\ell})$ with $|\mu|=\sum_{i=1}^{\ell}\mu_{i}=k$. 
\begin{definition}
A polypartition of size $n$ over the finite field $\For_{q}$ is a family of partitions $$\bbmu=\{\mu(P_{1}),\ldots,\mu(P_{r})\}$$ labelled by monic irreducible polynomials over $\For_{q}$, all different from $X$, and such that
$$|\bbmu|=\sum_{i=1}^{r} (\deg P_{i})\,|\mu(P_{i})|=n.$$
\end{definition}
To such a polypartition, we associate the block-diagonal matrix $J(\bbmu)$ whose blocks are the Jordan matrices $J((P_{i})^{(\mu(P_{i}))_{j}})$ with $i \in \lle 1,r\rre$ and $j \in \lle 1,\ell(\mu(P_{i}))\rre$. For instance, if $\bbmu=((2,1)_{P_{1}},(1,1)_{P_{2}})$, then
$$J(\bbmu)=\begin{pmatrix}
J((P_{1})^{2}) &  & & \\
 & J(P_{1})  & & \\
 & & J(P_{2})   & \\
& &   & J(P_{2})   \\
\end{pmatrix}.
$$
Jordan's reduction ensures that each conjugacy class of $\GL(n,\For_{q})$ contains a unique matrix $J(\bbmu)$ with $\bbmu$ polypartition of size $n$ over $\For_{q}$; this result is a classical consequence of the classification of finitely generated modules over principal rings (here, $\For_{q}[X]$). Thus, a linear basis of $Z(\C\GL(n,\For_{q}))$ consists in the classes $C_{\bbmu}$ labelled by the set $\ym(n,\For_{q})$ of these polypartitions. All the elements in $C_{\bbmu}$ have for characteristic and minimal polynomials
$$\chi_{\bbmu}(X)=\prod_{i=1}^{r}(P_{i}(X))^{|\mu(P_{i})|}\qquad;\qquad m_{\bbmu}(X)=\prod_{i=1}^{r}(P_{i}(X))^{(\mu(P_{i}))_{1}}.$$
On the other hand, it can be shown that \label{cardinalconjugacy}
$$\card \,C_{\bbmu}=\frac{(q^{n}-1)(q^{n}-q)\cdots (q^{n}-q^{n-1})}{q^{|\bbmu|+2b(\bbmu)} \, \prod_{i=1}^{r} \prod_{k \geq 1}(q^{-\deg P_{i}})_{m_{k}(\mu(P_{i}))}}$$
where $b(\bbmu)=\sum_{i=1}^{r}(\deg P_{i})\,b(\mu(P_{i}))=\sum_{i=1}^{r}\sum_{j=1}^{\ell(\mu(P_{i}))}(\deg P_{i})\,(j-1)\,(\mu(P_{i}))_{j}$; $(x)_{m}=(x;x)_{m}$ is the Pochhammer symbol $(1-x)(1-x^{2})\cdots(1-x^{m})$; and $m_{k}(\mu)$ is the number of parts $k$ in a partition $\mu$. This formula is proven by using Hall's theory of modules over discrete valuation rings, see \cite[Chapters 2, 4]{Mac95}.
\medskip

\begin{example}
Consider the $\For_{5}$-polypartition 
$\bbmu=\{x^2 + x + 1 : (2),\,\, x + 3 : (1, 1)\}.$
A representative of this conjugacy class in $\GL(6,\For_{5})$ is the Jordan matrix
$$\begin{pmatrix}
2 & 0 & 0 & 0 & 0 & 0 \\
0 & 2 & 0 & 0 & 0 & 0 \\
0 & 0 & 0 & 0 & 0 & 4 \\
0 & 0 & 1 & 0 & 0 & 3 \\
0 & 0 & 0 & 1 & 0 & 2 \\
0 & 0 & 0 & 0 & 1 & 3
\end{pmatrix},$$
and the cardinality of the class is
$$\frac{(5^{6}-1)(5^{6}-5)(5^{6}-5^{2})(5^{6}-5^{3})(5^{6}-5^{4})(5^{6}-5^{5})}{5^{6+2}\,(1-5^{-2})\,(1-5^{-1})(1-5^{-2})}=38\,418\,317\,437\,500\,000\,000.$$
\end{example}\bigskip

For a polypartition $\bbmu$ of size $k \leq n$, denote $\bbmu\!\! \uparrow^{n}$  the polypartition of size $n$ obtained by adding parts $1$ to the partition $\mu(X-1)$. This amounts to complete matrices with $1$'s on the diagonal in the bottom right corner. Therefore, our initial problem can be reformulated and generalized as follows:
\begin{problem}
Fix two polypartitions $\bblambda$ and $\bbmu$ of size $k$ and $l$. What is the expansion of $C_{\bblambda\uparrow^{n}}*C_{\bbmu\uparrow^{n}}$ in completed conjugacy classes
$$\sum_{\bbnu}c^{\bbnu}_{\bblambda\bbmu}(n)\,C_{\bbnu\uparrow^{n}} \,\,\,?$$
In particular, what is the dependence in $n$ of the structure coefficients $c^{\bbnu}_{\bblambda\bbmu}(n)$?
\end{problem}
\bigskip

\subsection{Centers as Hecke algebras and outline of the paper}
Denote $G\opp$ the set $G$ endowed with the opposite of the product of the group $G$; $g \mapsto g^{-1}$ is then an isomorphism of groups between $G$ and $G\opp$. It will be extremely important in our discussion to see $\Z(\C G)$ as the Hecke algebra $\C[G \backslash (G \times G\opp )/ G\opp]$ --- this fact is true for every finite group $G$. Consider indeed the double action of $G$ and $G\opp$ on $G \times G\opp$ given by $g\cdot (g_{1},g_{2}) \cdot h = ((gg_{1}h^{-1}),(hg_{2}g^{-1})).$ These actions correspond to the injective maps
\begin{align*}
G &\hookrightarrow G \times G\opp \qquad;\qquad G\opp \hookrightarrow G \times G\opp\\
g &\mapsto (g,g^{-1}) \qquad\qquad\qquad\quad h \mapsto (h^{-1},h).
\end{align*}
Let us detemine the orbit of an element $(g_{1},g_{2})$ under the action on the left by $G$ and on the right by $G\opp$.
Since one can multiply on the right by $(g_{2},g_{2}^{-1})$ to obtain $(g_{1}g_{2},e_{G})$, the orbit depends only on $g_{1}g_{2}$. Then, if $(g_{1}g_{2},e_{G})$ and $(h_{1}h_{2},e_{G})$ are in the same orbit, there exists $k,l \in G$ such that
$$g_{1}g_{2} = k\,(h_{1}h_{2} )\,l^{-1}\qquad;\qquad lk^{-1}=e_{G}.$$
Thus, $g_{1}g_{2} = k\,(h_{1}h_{2} )\,k^{-1}$, and we have proved that the orbit of $(g_{1},g_{2})$ consists in pairs $(h_{1},h_{2})$ such that $g_{1}g_{2}$ and $h_{1}h_{2} $ are conjugated. For a conjugacy class $C_{\lambda}$ of $G$, denote 
$$C_{\lambda}'=\frac{1}{\card \,G} \sum_{g_{1}g_{2} \,\in\, C_{\lambda}} (g_{1},g_{2}) \quad \in \C[ G\backslash (G \times G\opp)/G\opp].$$
Suppose that $C_{\lambda}\,C_{\mu}=\sum_{\nu}\,a_{\lambda\mu}^{\nu}\,C_{\nu}$. Then, a simple computation allows one to check that 
$C_{\lambda}'\,C_{\mu}'=\sum_{\nu}\,a_{\lambda\mu}^{\nu}\,C_{\nu}'$. Therefore, the map
\begin{align*}
\C[G \backslash (G \times G\opp)/G\opp] & \to \Z(\C G) \\
C_{\lambda}' &\mapsto C_{\lambda}
\end{align*}
realizes an isomorphism of commutative complex algebras. 
\bigskip

We are looking for a family of algebras $\mathscr{Z}(n,\For_{q})$ with the following properties. We want them to form an inverse system of graded algebras
$$\begin{CD} \cdots @>>>\mathscr{Z}(n+2,\For_{q})  @>>>\mathscr{Z}(n+1,\For_{q})  @>>>\mathscr{Z}(n,\For_{q}) @>>> \cdots \end{CD}\,\,,$$
so that in the end we will  be able to look at the projective limit $\mathscr{Z}(\infty,\For_{q})$, which will allow us to make generic computations. Then, we want these algebras to project onto the centers $Z(\C\GL(n,\For_{q}))$ of the group algebras of the general linear groups over $\For_{q}$, so that these generic computations will turn into generic identities between conjugacy classes of isomorphisms. Finally, we wish to construct the $\mathscr{Z}(n,\For_{q})$'s in a fashion similar to \cite{IK99,Tout12}, which is very combinatorial and natural. This means that we want to define the elements of $\mathscr{Z}(n,\For_{q})$ as linear combinations of partial isomorphisms of $(\For_{q})^{n}$, this notion staying for the moment vague and undefined. \bigskip

Since $\mathrm{Z}(n,\For_{q})=Z(\C\GL(n,\For_{q}))$ can be seen as an Hecke subalgebra of the algebra $\C[\GL(n,\For_{q}) \times (\GL(n,\For_{q}))\opp ]$, the way to do it will be to do the same construction with algebras $\mathscr{A}(n,\For_{q})$ that project onto the group algebras $\C[\GL(n,\For_{q}) \times (\GL(n,\For_{q}))\opp ]$. Then, to obtain the $\mathscr{Z}(n,\For_{q})$'s, we shall just look at invariant subalgebras of these $\mathscr{A}(n,\For_{q})$'s. Our paper is therefore organized as follows:\vspace{2mm}
\begin{itemize}
\item In Section \ref{algpartiso}, we define partial isomorphisms of $(\For_{q})^{n}$ and their product, and we prove that we obtain indeed an algebra $\mathscr{A}(n,\For_{q})$ (Theorem \ref{mainan}). The main difficulty is to see that the product is associative; this is related to the properties of certain linear operators on the algebra, and  the associativity in $\mathscr{A}(n,\For_{q})$ is essentially equivalent to the commutativity of the algebra formed by these operators.\vspace{2mm}
\item In Section \ref{constructions}, we look at some invariant subalgebras $\mathscr{Z}(n,\For_{q})\subset \mathscr{A}(n,\For_{q})$, and we prove that they form an inverse system of graded commutative algebras, that project onto the centers $\mathrm{Z}(n,\For_{q})$. We deduce from it the result of polynomiality of the structure coefficients that we have evoked in this introduction; see our main Theorem \ref{generalFH}.\vspace{2mm}
\item Finally, in Section \ref{secdegree1}, we give the table of multiplication of all elements of degree $1$ in the projective limit $\mathscr{Z}(\infty,\For_{q})$. This gives in particular a concrete answer to the first problem stated in this paper.\vspace{2mm}
\end{itemize}
We conclude this introduction by two important remarks.\bigskip

\begin{remark}
In essence, this paper relies only on linear algebra and combinatorics over finite fields. However, some of the results and proofs will be written using the language of probability: so, we shall speak freely of probability laws, conditional laws, Markov chains, \emph{etc}., and we refer to \cite{Bil95} for any detail on these notions. The reason is that the product of the algebra $\mathscr{A}(n,\For_{q})$ of partial isomorphisms that we shall study will be defined by taking averages of certain extensions of the partial isomorphisms. Thus, discrete probability will provide in many situations a convenient way to describe and prove identities in $\mathscr{A}(n,\For_{q})$. Since we are only dealing with finite sets, for pure algebraists, every ``probabilistic'' statement can be replaced by a statement with a function on a finite set taking non-negative values that sum to $1$. So probability is rather a commodity of language and a way of thinking than an absolute necessity in our reasonings.
\end{remark}
\bigskip

\begin{remark}
The reader may naturally ask what could be the uses of the theory of partial isomorphisms developed in this paper. We think of it as a first step in a much larger program, namely, the asymptotic representation theory of the finite general linear groups $\GL(n,\For_{q})$ --- other important approaches appear in the papers \cite{GO10,KV07,GKV12}. For symmetric groups $\sym(n)$, an important problem is to understand the asymptotic behavior of the random character value $\chi^\lambda(\sigma)$, where $\sigma$ is a fixed permutation and $\chi^{\lambda}$ is a random irreducible character of $\sym(n)$ taken under some spectral measure, for instance the Plancherel measure of the group. These random character values are important examples of random variables ``with a non-commutative flavour'', and their asymptotic behavior is related to random matrix theory, exclusion processes of particules, free probability, and many other topics; see for instance \cite{BDJ99,BOO00,Oko00,Bor09}. The Ivanov-Kerov algebra of partial permutations has proven extremely useful in order to compute the moments of these random character values $\chi^{\lambda}(\sigma)$; thus, they allowed to determine their asymptotic distribution. Indeed, the computation of the moments $\esper[(\chi^\lambda(\sigma))^k]$ is directly related to the computation of products and powers of conjugacy classes in symmetric group algebras. We refer in particular to \cite{Bia01,IO02,Sni06,FM12,Mel11,Mel12}, where such techniques and results are detailed. \bigskip

We hope to be able to follow the same program for finite linear groups $\GL(n,\For_q)$. In particular, among the problems that we wish to solve after this paper, let us ask the following question. Take the normalized trace $\mathbb{1}_{I_{n}}$ of $\C\GL(n,\For_{q})$, which expands as linear combination of all irreducible normalized characters of the group:
$$\mathbb{1}_{(g=I_{n})}=\sum_{\bblambda \in \ym(n,\For_{q})} \frac{(\dim V^{\bblambda})^{2}}{\card \,\GL(n,\For_{q})}\,\chi^{\bblambda}(g).$$
For $g$ fixed (say, $g$ is a diagonal matrix $D_{a}$) and completed by $1$'s on the diagonal, we consider $X_{g}=\chi^{\bblambda}(g)$ as a random variable under the probability measure $$\proba_{n,q}[\bblambda]=\frac{(\dim V^{\bblambda})^{2}}{\card\, \GL(n,\For_{q})},$$ which is the Plancherel measure of the group. What are the asymptotics of the law of $X_{g}$? To compute the moments of $X_g$ amounts to compute the powers of $C_\bbmu$ in the group algebras $\C\GL(n,\For_q)$, where $\bbmu$ is the conjugacy class of $g$; whence the interest of our results for the asymptotic analysis of the random character values $X_g$. The idea would then be to use these observables $X_g$ to get a better understanding of the properties of asymptotic concentration of the measures $\proba_{n,q}$; in particular, one should be able to recover this way the results of \cite{Ful08,Dud08}.
\end{remark}

\subsection{List of common notations}
To help the reader keep track of the notations, and of the various characters appearing as exponents or indices throughout the paper, we have listed hereafter the conventions that we shall use, see the following table. On the other hand, we recall the Pochhammer symbol $(q^{-1})_{n}=(1-q^{-1})(1-q^{-2})\cdots(1-q^{-n})$. In the following we shall try to write every enumeration as a polynomial in $q$ and in these Pochhammer symbols in $q^{-1}$.

\begin{center}
$$\hspace*{-1cm}\begin{tabular}{cp{12cm}} 
$q$ & cardinality of the base field $\For_q$\\
$n$ & dimension of the ambiant vector space $(\For_q)^n$\\
$\mathrm{M}(k\times l,\For_q)$ & space of matrices with $k$ rows and $l$ columns, and coefficients in $\For_q$\\
$\GL(n,\For_q)$& group of linear isomorphisms of $(\For_q)^n$\\
$\mathscr{I}(n,\For_q)$ & set of partial isomorphisms of $(\For_q)^n$\\
$\mathscr{A}(n,\For_q)$ & algebra of partial isomorphisms over $(\For_q)^n$\\
$\mathscr{Z}(n,\For_q)$ & subalgebra of invariants for the double action of $\GL(n,\For_q)$\vspace{2mm}\\
\hline \vspace{-2mm}&\\
$U,V,W,\ldots$ & vector subspaces of the ambiant vector space\\
$U^+,V^+,\ldots$ & larger vector subspaces containing $U,V,\ldots$ \\
$k,l,m,\ldots$ & dimensions of vector subspaces of $(\For_q)^n$\\
$\Ec=(e_1,\ldots,e_k)$ & basis of a vector subspace of dimension $k$\\
$\Ec^+=(e_1,\ldots,e_{k^+})$ & completion of the basis $\Ec$ in a basis of a larger subspace of dimension $k^+$\\
$\mathrm{Span}(e_1,\ldots,e_k)$& vector space spanned by the independent vectors $e_1,\ldots,e_k$\vspace{2mm}\\
\hline \vspace{-2mm}&\\
$\bbmu$ & polypartition (family of partitions labelled by polynomials)\\
$t(g)$& type of an automorphism $g : V \to V$, given by a polypartition\\
$\mu(P)$ & partition in a polypartition $\bbmu$ corresponding to the polynomial $P$\vspace{2mm} \\
\hline \vspace{-2mm}&\\
$g_1,g_2$ & isomorphisms between vector subspaces of $(\For_q)^n$\\
$g^+$ & extension of an isomorphism $g : V \to W$ to larger subspaces $V^+$ and $W^+$\\
$ \matr_{\Ec,\Fc}(g)$ & matrix of $g : V \to W$ written with respect to two bases $\Ec$ of $V$ and $\Fc$ of $W$\\
$\partiso{V}{W}{g_1}{g_2}$ & partial isomorphism given by the two arrows $g_1 : V \to W$ and $g_2 : W \to V$ \\
$ k_1$ & dimension of $\mathrm{Fix}(g_1g_2)$, which is also $\ell(\mu(X-1))$ if $t(g_1g_2)=\bbmu$ \\
$ k_{11}$ & $m_1(\mu(X-1))$ if $t(g_1g_2)=\bbmu$\vspace{2mm} \\
\hline \vspace{-2mm}&\\
$\mathscr{E}\partiso{V}{W\uparrow W^{+}}{g_{1}}{g_{2}}$& set of trivial extensions of $\partiso{V}{W}{g_1}{g_2}$ with fixed right space $W^+$\\
$\mathscr{E}\partiso{V^+ \uparrow V}{W}{g_{1}}{g_{2}}$& set of trivial extensions of $\partiso{V}{W}{g_1}{g_2}$ with fixed left space $V^+$\\
$E_q(n,k^+,k,k_1)$ & cardinality of the previous sets if $k^+=\dim V^+$, $k=\dim V$ and $k_1=\dim \mathrm{Fix}(g_1g_2)$\vspace{1mm}\\
$\mathscr{E}\partiso{V^+ \uparrow V}{W \uparrow W^+}{g_{1}}{g_{2}}$& set of trivial extensions of $\partiso{V}{W}{g_1}{g_2}$ with fixed left and rightsubspaces $V^+$ and $W^+$\\
$F_q(k^+,k,k_1)$ & cardinality of the previous sets if $k^+=\dim V^+$, $k=\dim V$ and $k_1=\dim \mathrm{Fix}(g_1g_2)$
\end{tabular}$$

$$\begin{tabular}{cp{12cm}} 

$\mathrm{L}_V^{V^+}, \mathrm{R}_W^{W^+}, \mathrm{LR}_{(V,W)}^{(V^+,W^+)}$& extension operators with fixed subspaces\\
$\mathrm{L}^{V}, \mathrm{R}^{W}$& extension operators\vspace{2mm}\\
\hline \vspace{-2mm}&\\
$X_n$& $\mathrm{rank}(v_1,\ldots,v_n)$, with the $v_i$'s independent uniform random vectors\\
$\proba_{d,q}$& probabilities related to $(X_n)_{n \in \N}$ in a $\For_q$-vector space of dimension $d$\\
$\buni_m, \buni_V$ & uniform law on free families of size $m$ in $(\For_q)^n$, or on bases of $V$\\
$\mathbb{U}_{m},\mathbb{U}_{m,V}$& uniform law on vector subspaces of dimension $m$, or on vector subspaces of dimension $m$ and containing $V$\\
$\mathbb{C}_{l,U,W,Y}$& uniform law on vector subspaces of dimension $l$, containing $U$, and whose sum with $W$ is $Y$\vspace{2mm}\\
\hline \vspace{-2mm}&\\
$\pi_n$ & projection from $\mathscr{A}(n,\For_q)$ to $\C[\GL(n,\For_q)\times (\GL(n,\For_q))^\mathrm{opp}]$\\
$\phi_n^{n'}$ & projection from $\mathscr{Z}(n',\For_q)$ to $\mathscr{Z}(n,\For_q)$ \\
$\Pi_n$ & projection from $\mathscr{Z}(\infty,\For_q)$ to $Z(\C\GL(n,\For_q))$\\
$A_{\bbmu,n}, \widehat{A}_{\bbmu,n}, \widetilde{A}_{\bbmu,n}$ & various renormalizations of the class of label $\bbmu$ in $\mathscr{Z}(n,\For_q)$ \\
$\widehat{A}_\bbmu$ & generic conjugacy class of label $\bbmu$ in $\mathscr{Z}(\infty,\For_q)$\\
$C_{\bbmu\uparrow^n}$ & completed conjugacy class of label $\bbmu$ in $Z(\C\GL(n,\For_q))$\vspace{0.5mm}\\
$\widetilde{C}_{\bbmu\uparrow^n}$ & normalized class $C_{\bbmu\uparrow^n}/(\card\, C_{\bbmu\uparrow^n})$
\end{tabular}$$
\end{center}
\bigskip

\section{Partial isomorphisms and their algebra}\label{algpartiso}
Fix $n\geq 1$ and a finite field $\For_{q}$.
\subsection{Partial isomorphisms and trivial extensions}
\begin{definition}\label{defpair}
A partial isomorphism of $(\For_{q})^{n}$ is a pair $(g_{1}:V \to W, g_{2}: W \to V)$, where $V$ and $W$ are two vector subspaces of same dimension $k$ in $(\For_{q})^{n}$, and $g_{1}$ and $g_{2}$ are isomorphisms between these spaces. 
\end{definition}
\noindent We shall use the notation $\partiso{V}{W}{g_{1}}{g_{2}}$ for a partial isomorphism, and the set of all partial isomorphisms of $(\For_{q})^{n}$ will be denoted $\mathscr{I}(n,\For_{q})$. Recall that the number of isomorphisms of a $\For_{q}$-vector space of dimension $k$ is
$$(q^{k}-1)(q^{k}-q)\cdots(q^{k}-q^{k-1})=q^{k^{2}}\,(q^{-1})_{k},$$
and the number of vector subspaces of dimension $k$ inside $(\For_{q})^{n}$ is
$$\frac{(q^{n}-1)(q^{n}-q)\cdots(q^{n}-q^{k-1})}{(q^{k}-1)(q^{k}-q)\cdots(q^{k}-q^{k-1})}=q^{(n-k)k} \,\frac{(q^{-1})_{n}}{ (q^{-1})_{n-k}\,(q^{-1})_{k}}.$$
Therefore, the cardinality of $\mathscr{I}(n,\For_{q})$ is
$$\sum_{k=0}^{n} q^{2nk} \left(\frac{(q^{-1})_{n}}{(q^{-1})_{n-k}}\right)^{\!2}.$$\bigskip

There is a natural action of $\GL(n,\For_{q})$ (respectively, of $(\GL(n,\For_{q}))\opp$) on the left (resp., on the right) of $\mathscr{I}(n,\For_{q})$, namely,
$$k \cdot \partiso{V}{W}{g_{1}}{g_{2}} \cdot l = \partiso{k^{-1}(V)}{l^{-1}(W)}{kg_{1}l^{-1}\,}{\,lg_{2}k^{-1}}.$$
The orbits of this double action are labelled by polypartitions of size $k \in \lle 0,n\rre$. Indeed, it is easy to see that the orbit of a partial isomorphism $\partiso{V}{W}{g_{1}}{g_{2}}$ is entirely determined by the type (polypartition) of the composed isomorphism $g_{1}g_{2} : V \to V$ (or, of the composed isomorphism $g_{2}g_{1} : W \to W$, since they have same type). Thus, if we are able to construct an algebra $\mathscr{A}(n,\For_{q})$ out of $\mathscr{I}(n,\For_{q})$, we will readily obtain a candidate for $\mathscr{Z}(n,\For_{q})$ by looking at these orbits. Therefore, the main problem that we shall adress in this Section is the definition of the product of two partial isomorphisms.
\bigskip

\begin{example}
Over the finite field $\For_{5}$, consider the subspaces $V=\Vect((1,0,0),(0,1,0))$ and $W=\Vect((0,1,0),(0,0,1))$ inside $(\For_{5})^{3}$.  With respect to the previously given bases, the matrices $\left(\begin{smallmatrix} 1& 2\\ 0&1 \end{smallmatrix}\right)$ and  $\left(\begin{smallmatrix} 3& 3\\1 & 2\end{smallmatrix}\right)$ define two isomorphisms $g_{1} : V \to W$ and $g_{2} : W \to V$. The type of the partial isomorphism $\partiso{V}{W}{g_{1}}{g_{2}}$ is the polypartition associated to the conjugacy class of the matrix 
$$ \begin{pmatrix} 1& 2\\ 0&1 \end{pmatrix}\,\begin{pmatrix} 3& 3\\ 1&2 \end{pmatrix}=\begin{pmatrix}0 & 2 \\ 1 & 2\end{pmatrix}.$$
This is the Jordan matrix of the irreducible polynomial $X^{2}+3X+3$, so $t\partiso{V}{W}{g_{1}}{g_{2}}=\{X^{2}+3X+3:(1)\}$.
\end{example}\bigskip

A preliminary step for our program is to define properly what is an extension of a partial isomorphism.
\begin{definition}\label{partisodef}
An extension of a partial isomorphism $\partiso{V}{W}{g_{1}}{g_{2}}$ is a partial isomorphism $\partext{V^{+}}{W^{+}}{g_{1}^{+}}{g_{2}^{+}}$ with $V \subset V^{+}$, $W \subset W^{+}$, $(g_{1}^{+})_{|V}=g_{1}$ and $(g_{2}^{+})_{|W}=g_{2}$.\bigskip

\noindent Denote $k=\dim V=\dim W$ and $k^{+}=\dim V^{+}=\dim W^{+}$. The extension is called trivial if one of the following equivalent assertions is satisfied:\vspace{2mm}
\begin{enumerate}
\item Its type is obtained from the type $\bbmu=(\mu(P_{1}),\ldots,\mu(P_{r}))$ of the partial isomorphism $\partiso{V}{W}{g_{1}}{g_{2}}$ by adding parts $1$ to the partition $\mu(X-1)$ (which might have been empty). \vspace{2mm}
\item There are decompositions $V^{+}=V\oplus A$ and $W^{+}=W \oplus B$, such that
$$g_{1}^{+}=g_{1}\oplus \psi \quad;\quad g_{2}^{+}=g_{2}\oplus \psi^{-1}$$
with $\psi$ isomorphism between $A$ and $B$.\vspace{2mm}
\item The induced quotient isomorphisms 
$$\widetilde{g_{1}} : V^{+}/V \to W^{+}/W \quad\text{and}\quad \widetilde{g_{2}} : W^{+}/W \to V^{+}/V$$
are inverse of one another: $\widetilde{g_{1}}\widetilde{g_{2}}=\id_{V^{+}/V}$.\vspace{2mm}
\item Fix any basis $\Ec^{+}=(e_{1},\ldots,e_{k}^{+})$ of $V^{+}$ such that $\Ec=(e_{1},\ldots,e_{k})$ is a basis of $V$. Denote then $\Fc^{+}=(f_{1},\ldots,f_{k}^{+})$ the basis of $W^{+}$ given by $f_{i}=g_{1}^{+}(e_{i})$; the matrix of $g_{1}^{+}$ with respect to these two bases is then $I_{k^{+}}$. One has:
$$\matr_{\Fc^{+},\Ec^{+}}(g_{2}^{+})=\begin{pmatrix} G & P \\ 0 & I_{k^{+}-k}\end{pmatrix},$$
where $G=\matr_{\Ec}(g_{1}g_{2})$, and $P=(G-I_{k})\,R$ with $R$ arbitrary rectangular matrix of size $k \times (k^{+}-k)$.\vspace{2mm}
\end{enumerate}
\end{definition}
\begin{proof} We denote $\partext{V^{+}}{W^{+}}{g_{1}^{+}}{g_{2}^{+}}$ an extension of $\partiso{V}{W}{g_{1}}{g_{2}}$.\vspace{2mm}
\begin{itemize}
\item[$(1) \Rightarrow (2)$] Suppose that $t(g_{1}^{+}g_{2}^{+})=\bbmu \sqcup (X-1 : 1^{k^{+}-k})$. There is a basis $\Ec^{+}=(e_{1},\ldots,e_{k^{+}})$ of $V^{+}$ such that $\Ec=(e_{1},\ldots,e_{k})$ is a basis of $V$, and
\begin{equation}\matr_{\Ec^{+}}(g_{1}^{+}g_{2}^{+})=\begin{pmatrix} \matr_{\Ec} (g_{1}g_{2}) & 0 \\ 0 & I_{k^{+}-k}\end{pmatrix}.\label{trivialmatrix1}\end{equation}
We take for basis $\Fc^{+}$ of $W^{+}$ the images $f_{1},\ldots,f_{k^{+}}$ of the vectors $e_{1},\ldots,e_{k^{+}}$ by $g^{+}_{1}$. Since $(g_{1}^{+})_{|V}=g_{1}$, the $k$ first vectors $f_{1},\ldots,f_{k}$ form a basis $\Fc$ of $W$. Set then $A=\Vect(e_{k+1},\ldots,e_{k^{+}})$ and $B=\Vect(f_{k+1},\ldots,f_{k^{+}})$. By choice of $\Fc^{+}$, one has
\begin{equation}\matr_{\Ec^{+},\Fc^{+}}(g_{1}^{+})=\begin{pmatrix} I_{k} & 0 \\ 0 & I_{k^{+}-k} \end{pmatrix},\label{trivialmatrix2}\end{equation}
whereas
\begin{equation}\matr_{\Fc^{+},\Ec^{+}}(g_{2}^{+})=\begin{pmatrix} \matr_{\Fc,\Ec}(g_{2}) & P \\ 0 & M \end{pmatrix},\label{trivialmatrix3}\end{equation}
where $M$ is invertible and $P$ is \emph{a priori} an arbitrary matrix. However, the matrix in Equation \eqref{trivialmatrix1} is the product of the matrices in \eqref{trivialmatrix2} and \eqref{trivialmatrix3}, so, by identification, $\matr_{\Fc,\Ec}(g_{2})=\matr_{\Ec}(g_{1}g_{2})$, $M=I_{k^{+}-k}$ and $P=0$. The isomorphism $\psi$ is then the one sending $e_{j}$ to $f_{j}$ for $j \in \lle k+1,k^{+}\rre$, and the proof of the implication is done.\vspace{2mm}
\item[$(2) \Rightarrow (1)$] Conversely, with respect to the decomposition $V^{+}=V \oplus A$, one has
$$g_{1}^{+}g_{2}^{+}=(g_{1}\oplus \psi)(g_{2}\oplus \psi^{-1})=g_{1}g_{2}\oplus \id_{A},$$
so $t(g_{1}^{+}g_{2}^{+})=t(g_{1}g_{2})\sqcup (X-1 : 1^{k^{+}-k})$.\vspace{2mm}
\item[$(2) \Leftrightarrow (3)$] Obvious since there are natural isomorphisms $A \simeq V^{+}/V$ and $B \simeq W^{+}/W$.\vspace{2mm}
\item[$(2) \Rightarrow (4)$] For any basis $\Dc^{+}$ adapted to the decomposition $V^{+}=V\oplus A$, 
$$\matr_{\Dc^{+}}(g_{1}^{+}g_{2}^{+})=\begin{pmatrix}G & 0 \\ 
0 & I_{k^{+}-k}\end{pmatrix}.$$
Fix a basis $\Ec^{+}=(e_{1},\ldots,e_{k},e_{k+1},\ldots,e_{k^{+}})$ of $V^{+}$, and another basis adapted to the decomposition $V^{+}=V \oplus A$: $\mathcal{C}^{+}=(e_{1},\ldots,e_{k},c_{k+1},\ldots,c_{k^{+}})$.  One has
$$\matr_{\mathcal{C}^{+}}(\Ec^{+})=\begin{pmatrix} I_{k} & R\\  0 & Q\end{pmatrix},$$
and since this is an invertible matrix, $Q$ has to be invertible. Consider now the basis $\Dc^{+}$ with
$$\matr_{\Dc^{+}}(\mathcal{C}^{+})=\begin{pmatrix} I_{k} & 0\\  0 & Q^{-1}\end{pmatrix};$$
it is again adapted to the decomposition $V^{+}=V \oplus A$. If $R'=Q^{-1}R$, then one has
$$\matr_{\Dc^{+}}(\Ec^{+})=\matr_{\Dc^{+}}(\mathcal{C}^{+})\,\matr_{\mathcal{C}^{+}}(\Ec^{+})=\begin{pmatrix}
I_{k} & R' \\
0 & I_{k^{+}-k}
\end{pmatrix},$$
and therefore,
\begin{align*}\matr_{\Ec^{+}}(g_{1}^{+}g_{2}^{+})&=\matr_{\Ec^{+}}(\Dc^{+})\,\matr_{\Dc^{+}}(g_{1}^{+}g_{2}^{+}) \,\matr_{\Dc^{+}}(\Ec^{+})\\
&=\begin{pmatrix}
I_{k} & -R' \\
0 & I_{k^{+}-k}
\end{pmatrix}\,\begin{pmatrix} G & 0 \\ 0 & I_{k^{+}-k} \end{pmatrix}\,\begin{pmatrix}
I_{k} & R' \\
0 & I_{k^{+}-k}
\end{pmatrix}\\
&=\begin{pmatrix}
G & (G-I_{k}) \,R'\\
0 & I_{k^{+}-k}
\end{pmatrix},\end{align*}
where $G=\matr_{\Ec}(g_{1}g_{2})$. With $f_{i}=g_{1}^{+}(e_{i})$, this matrix is also $\matr_{\Fc^{+},\Ec^{+}}(g_{2}^{+})$, so $(2)\Rightarrow (4)$ is proven.\vspace{2mm}
\item[$(4) \Rightarrow (2)$] From a writing $\matr_{\Fc^{+},\Ec^{+}}(g_{2}^{+})=\left(\begin{smallmatrix}G & (G-I_{k})\,R\\ 0 & I_{k^{+}-k}\end{smallmatrix}\right)$, one can go backwards and look at the basis $\Dc^{+}$ of $V^{+}$ given by 
$$\matr_{\Dc^{+}}(\Ec^{+})=\begin{pmatrix} I_{k} & R \\
0 & I_{k^{+}-k}
\end{pmatrix};$$
the matrix of $g_{1}^{+}g_{2}^{+}$ in this new basis is $\left(\begin{smallmatrix} G & 0 \\ 0 & I_{k^{+}-k}\end{smallmatrix}\right)$. Taking $A=\Vect(d_{k+1},\ldots,d_{k^{+}})$ and $B=g_{1}^{+}(A)$, one gets back the hypothesis $(2)$.\vspace{-6mm}
\end{itemize}
\end{proof}
\bigskip

In the following, we shall need to know the number of trivial extensions of a given partial isomorphism $\partiso{V}{W}{g_{1}}{g_{2}}$ of degree $k$ (the dimension of $V$ and $W$) to spaces $V^{+}$ and $W^{+}$ of dimension $k^{+}$, with $W^{+}$ fixed among the subspaces of $(\For_{q})^{n}$ containing $W$, but $V^{+}$ free. The fourth characterization of trivial extensions will ease a lot this enumeration.
\begin{lemma}\label{importantbijection}
There is a bijection between trivial extensions $\partext{V^{+}}{W^{+}}{g_{1}^{+}}{g_{2}^{+}}$ of a partial isomorphism $\partiso{V}{W}{g_{1}}{g_{2}}$ with $W^{+}$ fixed, and pairs $((e_{k+1},\ldots,e_{k^{+}}),P)$, where:\vspace{2mm}
\begin{enumerate}
\item $(e_{1},\ldots,e_{k^{+}})$ is a completion of a basis $(e_{1},\ldots,e_{k})$ of $V$ into a family of $k^{+}$ linearly independent vectors;\vspace{2mm}
\item $P=(G-I_{k})\,R$ is a rectangular matrix of size $k \times (k^{+}-k)$, where $R$ is arbitrary and $G=\matr_{(e_{1},\ldots,e_{k})}(g_{1}g_{2})$.\vspace{2mm}
\end{enumerate}
\end{lemma}
\begin{proof}
In the following, we fix a basis $\Fc^{+}=(f_{1},\ldots,f_{k}^{+})$ of $W^{+}$ such that $\Fc=(f_{1},\ldots,f_{k})$ is a basis of $W$. We then denote $e_{i}=(g_{1}^{+})^{-1}(f_{i})$ for $i \in \lle 1,k \rre$; this basis $\Ec$ of $V$ is also fixed.\bigskip

\noindent Fix a pair $((e_{k+1},\ldots,e_{k^{+}}),P)$ such as in the statement of the lemma.  One defines a trivial extension of the partial isomorphism by setting
\begin{align*}
\Vect(e_{1},\ldots,e_{k^{+}})&=V^{+}\\
\matr_{\Ec^{+},\Fc^{+}}(g_{1}^{+})&=I_{k^{+}}\\
\matr_{\Fc^{+},\Ec^{+}}(g_{2}^{+})&=\begin{pmatrix} G & P \\
0 & I_{k^{+}-k}\end{pmatrix}.
\end{align*}
The fourth characterization of trivial extensions ensures that this extension is indeed trivial. Conversely, given a trivial extension $\partext{V^{+}}{W^{+}}{g_{1}^{+}}{g_{2}^{+}}$, one gets back $e_{i}$ for $i \in \lle k+1,k^{+} \rre$ by setting $e_{i}=(g_{1}^{+})^{-1}(f_{i})$; and then $P$ by looking at the upper-right block of the matrix $\matr_{\Fc^{+},\Ec^{+}}(g_{2}^{+})$.
\end{proof}
\bigskip

\begin{corollary}\label{cleverenumeration}
Let $\partiso{V}{W}{g_{1}}{g_{2}}$ be a partial isomorphism of $(\For_{q})^{n}$. Denote $k=\dim V=\dim W$, and $k_{1}$ the dimension of the set of fixed points of $g_{1}g_{2}$ --- this is also the length of $\mu(X-1)$ if $t(g_{1}g_{2})=\bbmu$. If $W^{+} \supset W$ is a fixed vector subspace of dimension $k^{+}\geq k$, then the number of trivial extensions of $\partiso{V}{W}{g_{1}}{g_{2}}$ with right vector subspace $W^{+}$ and free left vector subspace $V^{+}$ is
\begin{align*}E_{q}(n,k^{+},k,k_{1})&=q^{(k-k_{1})(k^{+}-k)}\,(q^{n}-q^{k})\cdots (q^{n}-q^{k^{+}-1})\\
&=q^{(n+k-k_{1})(k^{+}-k)}\,\frac{(q^{-1})_{n-k}}{(q^{-1})_{n-k^{+}}}.\end{align*}
By symmetry, one obtains of course the same number if $V^{+}$ is fixed but $W^{+}$ is left free.
\end{corollary}
\medskip

\begin{proof}
The first factor $q^{(k-k_{1})(k^{+}-k)}$ is the cardinality of the image of the map
\begin{align*}
\mathrm{M}(k \times (k^{+}-k),\For_{q}) & \to \mathrm{M}(k \times (k^{+}-k),\For_{q})\\
R &\mapsto (G-I_{k})\,R,
\end{align*}
and the second factor is the number of possible completions of the basis $\Ec=(e_{1},\ldots,e_{k})$ into a family of linearly independent vectors of size $k^{+}$.
\end{proof}
\bigskip

To be complete, let us state the equivalent of Corollary \ref{cleverenumeration} with fixed left and right subspaces.

\begin{corollary}
The number of trivial extensions of a partial isomorphism $\partiso{V}{W}{g_{1}}{g_{2}}$ with fixed left and right subspaces $V^{+}$ and $W^{+}$ is 
\begin{align*}F_{q}(k^{+},k,k_{1})&=q^{(k-k_{1})(k^{+}-k)}\,(q^{k^{+}}-q^{k})(q^{k^{+}}-q^{k+1})\cdots(q^{k^{+}}-q^{k^{+}-1})\\
&=q^{(k^{+}+k-k_{1})(k^{+}-k)}\,(q^{-1})_{k^{+}-k}.\end{align*}
\end{corollary}
\begin{proof}
In the set of all trivial extensions of $\partiso{V}{W}{g_{1}}{g_{2}}$ with fixed right subspace $W^{+}$, each left subspace $V^{+}$ is obtained the same number of times, because if $V^{+}_{1}$ and $V^{+}_{2}$ are two such subspaces, then one can use an isomorphism between them that fixes $V$ to get a bijection between trivial extensions. Therefore, 
$$F_{q}(k^{+},k,k_{1})=\frac{E_{q}(n,k^{+},k,k_{1})}{\text{number of subspaces of $(\For_{q})^{n}$ containing $V$ and of dimension }k^{+}}.$$
The denominator of this fraction is equal to
$$\frac{(q^{n}-q^{k})(q^{n}-q^{k+1})\cdots(q^{n}-q^{k^{+}-1})}{(q^{k^{+}}-q^{k})(q^{k^{+}}-q^{k+1})\cdots (q^{k^{+}}-q^{k^{+}-1})},$$
as can by seen by first enumerating families $(e_{1},\ldots,e_{k^{+}})$ of linearly independent vectors completing a basis $(e_{1},\ldots,e_{k})$ of $V$, and then counting how many families give the same subspace $V^{+}=\Vect(e_{1},\ldots,e_{k^{+}})$.
\end{proof}\bigskip

\subsection{Combinatorics of the extension operators} The notion of extension of a partial isomorphism being clarified, one can define correctly the product of two partial isomorphisms. Let $\mathscr{A}(n,\For_{q})$ be the $\C$-vector space with basis the set $\mathscr{I}(n,\For_{q})$ of partial isomorphisms over $(\For_{q})^{n}$. Given a partial isomorphism $\partiso{V}{W}{g_{1}}{g_{2}}$, we denote $\mathscr{E}\partiso{V^{+}\uparrow V}{W}{g_{1}}{g_{2}}$ the set of trivial extensions with fixed left vector subspace $V^{+}$, and $\mathscr{E}\partiso{V}{W\uparrow W^{+}}{g_{1}}{g_{2}}$ the set of trivial extensions with fixed right vector subspace $W^{+}$. 
\begin{definition}
The product of two elements of $\mathscr{A}(n,\For_{q})$ is defined by linear extension of the rule
$$\partiso{U}{V}{g_{1}}{g_{2}} \mathrel{*} \partiso{W}{X}{h_{1}}{h_{2}}  = \frac{1}{E_{q}(n,m,k,k_{1})\,\,E_{q}(n,m,l,l_{1})}\,\, \sum \,\partext{U^{+}}{X^{+}}{ g_{1}^{+}h_{1}^{+}}{ h_{2}^{+}g_{2}^{+}},$$
where the sum runs over trivial extensions 
\begin{align*}
&\partext{U^{+}}{V+W}{g_{1}^{+}}{g_{2}^{+}} \,\in\, \mathscr{E}\partiso{U}{V\uparrow(V+W)}{g_{1}}{g_{2}}\\
\text{and }&\partext{V+W}{X^{+}}{h_{1}^{+}}{h_{2}^{+} } \,\in\, \mathscr{E}\partiso{(V+W)\uparrow W}{X}{h_{1}}{h_{2}}.
\end{align*}
The parameters $m$, $k$, $l$ are the dimensions of $V+W$, $V$ and $W$, and $k_{1}$ and $l_{1}$ are the dimensions of the spaces of fixed points of $g_{1}g_{2}$ and $h_{1}h_{2}$.
\end{definition}
\begin{theorem}\label{mainan}
Endowed with the product $*$, $\mathscr{A}(n,\For_{q})$ is an associative complex algebra.
\end{theorem}\bigskip

The unity of the algebra is trivially the ``empty'' partial isomorphism $\partiso{\{0\}}{\{0\}}{\id}{\id}$, so the only thing to check in Theorem \ref{mainan} is the associativity of the product, and this is surprisingly difficult. Let us introduce a few more notations. If $V \subset V^{+}$ and $W \subset W^{+}$, denote
\begin{align*}
\mathrm{L}_{V}^{V^{+}}\partiso{V}{W}{g_{1}}{g_{2}} &=\frac{1}{E_{q}(n,k^{+},k,k_{1})}\,\sum_{\mathscr{E}\partiso{V^{+}\uparrow V}{W}{g_{1}\,\,}{g_{2}\,} }\partiso{V^{+}}{W^{+}}{g_{1}^{+}}{g_{2}^{+}} \\
\mathrm{R}_{W}^{W^{+}}\partiso{V}{W}{g_{1}}{g_{2}} &=\frac{1}{E_{q}(n,k^{+},k,k_{1})}\,\sum_{\mathscr{E}\partiso{V}{W\uparrow W^{+}}{g_{1}\,\,}{g_{2}\,} }\partiso{V^{+}}{W^{+}}{g_{1}^{+}}{g_{2}^{+}} 
\end{align*}
the means of the trivial extensions of a given partial isomorphism, with fixed left or right vector subspace. The multiplication rule is then
$$\partiso{U}{V}{g_{1}}{g_{2}} \mathrel{*} \partiso{W}{X}{h_{1}}{h_{2}}  =  \mathrm{R}_{V}^{V+W}\partiso{U}{V}{g_{1}}{g_{2}} \cdot \mathrm{L}_{W}^{V+W}\partiso{W}{X}{h_{1}}{h_{2}},$$
where on the right-hand side the product is the usual composition of arrows $U^{+} \to V+W$ and $V+W \to X^{+}$, or of arrows in the reverse directions. Notice that the operators $\mathrm{L}$ and $\mathrm{R}$ are particular cases of multiplications: 
\begin{align*}
\mathrm{L}_{V}^{V^{+}}\partiso{V}{W}{g_{1}}{g_{2}} &= \partiso{V^{+}}{V^{+}}{\id}{\id}*\partiso{V}{W}{g_{1}}{g_{2}};\\
\mathrm{R}_{W}^{W^{+}}\partiso{V}{W}{g_{1}}{g_{2}} &= \partiso{V}{W}{g_{1}}{g_{2}}*\partiso{W^{+}}{W^{+}}{\id}{\id}.
\end{align*}
\begin{proposition}\label{firststepan}
Consider nested subspaces $W \subset W^{+}\subset W^{++}$. One has $$\mathrm{R}_{W^{+}}^{W^{++}}\circ \mathrm{R}_{W}^{W^{+}} = \mathrm{R}_{W}^{W^{++}},$$ and similarly for the operators $\mathrm{L}$.
\end{proposition}
\begin{proof}
Even if one already knew the associativity of the product $*$, this would not be obvious, since the product $\partiso{W^{+}}{W^{+}}{\id}{\id}*\partiso{W^{++}}{W^{++}}{\id}{\id}$ is an average of partial isomorphisms between spaces $(W^{+})^{+}$ containing $W^{+}$, and $W^{++}$. Hence,  it is different from $\partiso{W^{++}}{W^{++}}{\id}{\id}$. However, these intermediary spaces $(W^{+})^{+}$ will disappear when one multiplies on the left by a partial isomorphism $\partiso{V}{W}{g_{1}}{g_{2}}$. A similar idea will be used in the proof of Lemma \ref{lastlemmaan}\bigskip

\noindent Fix a basis $(f_{1},\ldots,f_{k^{++}})$ of $W^{++}$, such that $(f_{1},\ldots,f_{k})$ is a basis of $W$ and $(f_{1},\ldots,f_{k^{+}})$ is a basis of $W^{+}$. We then denote $e_{1},\ldots,e_{k}$ the reciprocal images of the vectors $f_{1},\ldots,f_{k}$ by $g_{1}$; and $G=\matr_{\Fc,\Ec}(g_{1}g_{2})$. A convenient way to write $\mathrm{R}_{W}^{W^{++}}\partiso{V}{W}{g_{1}}{g_{2}}$ is as follows:
\begin{align}
\mathrm{R}_{W}^{W^{++}}\partiso{V}{W}{g_{1}}{g_{2}}&= \frac{1}{q^{(k-k_{1})(k^{++}-k)}\,(q^{n}-q^{k})(q^{n}-q^{k+1})\cdots(q^{n}-q^{k^{++}-1})} \nonumber \\
&\quad\times\!\!\!\!\!\!\sum_{\substack{\Ec^{++}\setminus\Ec=(e_{k+1},\ldots,e_{k^{++}})\\ P=(G-I_{k})\,R}} \left(\Vect(\Ec^{++})\,\,\big|\,\, I_{k^{++}}\rightleftarrows \left(\begin{smallmatrix} G & P\\ 0 & I_{k^{++}-k} \end{smallmatrix}\right) \,\,\big|\,\,W^{++} \right).\label{RR}
\end{align}
This equation corresponds to the content of Lemma \ref{importantbijection}. Using the same notations, we get:
\begin{align}
\mathrm{R}_{W}^{W^{+}}\partiso{V}{W}{g_{1}}{g_{2}}&= \frac{1}{q^{(k-k_{1})(k^{+}-k)}\,(q^{n}-q^{k})(q^{n}-q^{k+1})\cdots(q^{n}-q^{k^{+}-1})} \nonumber \\
&\quad\times\!\!\!\!\!\!\sum_{\substack{(e_{k+1},\ldots,e_{k^{+}})\\ M=(G-I_{k})\,S}} \left(\Vect(\Ec^{+})\,\,\big|\,\, I_{k^{+}}\rightleftarrows \left(\begin{smallmatrix} G & M\\ 0 & I_{k^{+}-k} \end{smallmatrix}\right) \,\,\big|\,\,W^{+} \right);\nonumber \\
\mathrm{R}_{W}^{W^{+}}\partiso{V}{W}{g_{1}}{g_{2}}&= \frac{1}{q^{(k-k_{1})(k^{+}-k)}\,(q^{n}-q^{k})(q^{n}-q^{k+1})\cdots(q^{n}-q^{k^{+}-1})}\nonumber  \\
&\quad\times \frac{1}{q^{(k-k_{1})(k^{++}-k^{+})}\,(q^{n}-q^{k^{+}})(q^{n}-q^{k^{+}+1})\cdots(q^{n}-q^{k^{++}-1})} \nonumber \\
&\quad\times\!\!\!\!\!\!\sum_{\substack{(e_{k+1},\ldots,e_{k^{++}})\\M=(G-I_{k})\,S \\ N=(G'-I_{k^{+}})\,T}} \left(\Vect(\Ec^{++})\,\,\big|\,\, I_{k^{+}}\rightleftarrows \begin{pmatrix}\left(\begin{smallmatrix} G & M\\ 0 & I_{k^{+}-k} \end{smallmatrix}\right)& N\\ 0 & I_{k^{++}-k^{+}}\end{pmatrix} \,\,\big|\,\,W^{++} \right);\label{RRR}
\end{align}
where in the last term $G'$ denotes the intermediary matrix. Notice that the number of terms in Equations \eqref{RR} and \eqref{RRR} is the same. Therefore, it suffices to show that each term of \eqref{RRR} appears in \eqref{RR}. This is immediate from the following computation:
\begin{align*}
\begin{pmatrix}\left(\begin{smallmatrix} G & M\\ 0 & I_{k^{+}-k} \end{smallmatrix}\right)& N\\ 0 & I_{k^{++}-k^{+}}\end{pmatrix} &= \begin{pmatrix}\left(\begin{smallmatrix} G & (G-I_{k})\,S\\ 0 & I_{k^{+}-k} \end{smallmatrix}\right)&\left(\left(\begin{smallmatrix} G & (G-I_{k})\,S\\ 0 & I_{k^{+}-k} \end{smallmatrix}\right)-I_{k^{+}}\right)\left(\begin{smallmatrix}T_{1}\\ T_{2}\end{smallmatrix}\right) \\ 0 & I_{k^{++}-k^{+}}\end{pmatrix}\\
&= \begin{pmatrix}\left(\begin{smallmatrix} G & (G-I_{k})\,S\\ 0 & I_{k^{+}-k} \end{smallmatrix}\right)&\left(\begin{smallmatrix} G-I_{k} & (G-I_{k})\,S\\ 0 & 0 \end{smallmatrix}\right)\left(\begin{smallmatrix}T_{1}\\ T_{2}\end{smallmatrix}\right) \\ 0 & I_{k^{++}-k^{+}}\end{pmatrix}\\
&= \begin{pmatrix}G & (G-I_{k})\,S & (G-I_{k})\,T_{3} \\ 0 & I_{k^{+}-k} & 0 \\ 0 & 0 & I_{k^{++}-k^{+}}\end{pmatrix}
\end{align*}
with $T_{3}=T_{1}+ST_{2}$, which is again an arbitrary rectangular matrix.
\end{proof}
\bigskip

In the following, we shall work with slightly more general operators $\mathrm{L}$ and $\mathrm{R}$.
\begin{definition}
For $W$ arbitrary subspace of $(\For_{q})^{n}$ and $\partiso{U}{V}{g_{1}}{g_{2}}$ arbitrary element of $\mathscr{I}(n,\For_{q})$, we set
$$\mathrm{R}^{W}\partiso{U}{V}{g_{1}}{g_{2}}=\mathrm{R}^{V+W}_{V}\partiso{U}{V}{g_{1}}{g_{2}},$$
and similarly for the operators $\mathrm{L}^{W}$. We call these generalized operators $\mathrm{L}^{W}$ and $\mathrm{R}^{W}$ the extension operators; they give rise to well-defined linear endomorphisms of $\mathscr{A}(n,\For_{q})$.
\end{definition} 
\begin{proposition}\label{secondstepan}
Consider two arbitrary subspaces $W$ and $X$ of $(\For_{q})^{n}$. One has 
$$\mathrm{R}^{X}\circ \mathrm{R}^{W}=\mathrm{R}^{W+X},$$
and similarly for the operators $\mathrm{L}$.
\end{proposition}
\begin{proof}
For any partial isomorphism $\partiso{U}{V}{g_{1}}{g_{2}}$, if $W\subset V$, then 
$$\mathrm{R}^{W}\partiso{U}{V}{g_{1}}{g_{2}}=\mathrm{R}^{V+W}_{V}\partiso{U}{V}{g_{1}}{g_{2}}=\mathrm{R}^{V}_{V}\partiso{U}{V}{g_{1}}{g_{2}}=\partiso{U}{V}{g_{1}}{g_{2}}.$$
Consider a partial isomorphism appearing in $\mathrm{R}^{X}\circ \mathrm{R}^{W}\partiso{U}{V}{g_{1}}{g_{2}}$. Its right subspace contains both $W$ and $X$ (and $V$), so it must contain $W+X$. Therefore, one can apply again $\mathrm{R}^{W+X}$ without changing the result, and
$$\mathrm{R}^{X}\circ \mathrm{R}^{W} = \mathrm{R}^{W+X}\circ\mathrm{R}^{X}\circ \mathrm{R}^{W}.$$
As a consequence, the proposition will be shown if one proves that for nested subspaces $W \subset W'$, $\mathrm{R}^{W'}\circ \mathrm{R}^{W}=\mathrm{R}^{W'}.$ Indeed, assuming this is true, one has then
\begin{align*}\mathrm{R}^{X}\circ \mathrm{R}^{W} &= \mathrm{R}^{W+X}\circ (\mathrm{R}^{X}\circ \mathrm{R}^{W})\\
&=(\mathrm{R}^{W+X}\circ \mathrm{R}^{X})\circ \mathrm{R}^{W}\\
&=\mathrm{R}^{W+X}\circ \mathrm{R}^{W} \quad\text{since }X \subset W+X,\\
&=\mathrm{R}^{W+X} \qquad\quad\,\,\,\text{since }W \subset W+X.
\end{align*}
Fix two nested subspaces $W \subset W'$, and a partial isomorphism $\partiso{U}{V}{g_{1}}{g_{2}}$. One has 
\begin{align*}
\mathrm{R}^{W'}(\mathrm{R}^{W}\partiso{U}{V}{g_{1}}{g_{2}})&=\mathrm{R}^{W'}(\mathrm{R}^{U+W}_{U}\partiso{U}{V}{g_{1}}{g_{2}}) \\
&=\mathrm{R}^{U+W'}_{U+W}(\mathrm{R}^{U+W}_{U}\partiso{U}{V}{g_{1}}{g_{2}}) \\
&=\mathrm{R}^{U+W'}_{U}\partiso{U}{V}{g_{1}}{g_{2}}\\
&=\mathrm{R}^{W'}\partiso{U}{V}{g_{1}}{g_{2}}
\end{align*}
by using Proposition \ref{firststepan} on the third line.
\end{proof}
\bigskip

This leads us to the main result regarding extension operators:
\begin{theorem}\label{mainstepan}
The algebra of extension operators $\langle \mathrm{L}^{W},\mathrm{R}^{W} \rangle_{W\,\text{subspace of }(\For_{q})^{n}}$ is a commutative subalgebra of $\mathrm{End}(\mathscr{A}(n,\For_{q}))$.
\end{theorem}\bigskip

Proposition \ref{secondstepan} obviously implies that $\mathrm{R}^{W}\circ \mathrm{R}^{X} =\mathrm{R}^{X}\circ \mathrm{R}^{W} $ and  $\mathrm{L}^{W}\circ \mathrm{L}^{X} =\mathrm{L}^{X}\circ \mathrm{L}^{W} $, so it remains to prove that
\begin{equation} \mathrm{L}^{W}\circ \mathrm{R}^{X}= \mathrm{R}^{X}\circ \mathrm{L}^{W}  \label{hardcommutation}.\end{equation}
The difficulty here is that when applied to an element $\partiso{U}{V}{g_{1}}{g_{2}}$, both sides of \eqref{hardcommutation} yield averages of partial isomorphisms whose left and right subspaces are not fixed. Indeed, these subspaces must contain $U+W$ and $V+X$, but they can be larger. Consequently, the first thing to check will be that both sides induce the same probability measure on pairs $(W',X')$ of subspaces of same dimension. Then, we shall prove the following decomposition:
$$\mathrm{L}^{W}\circ \mathrm{R}^{X} \partiso{U}{V}{g_{1}}{g_{2}}= \sum_{\substack{ U+W \subset Y\\ V+X \subset Z}} \proba[(Y,Z)]\,\,\Psi^{(Y;Z)}_{(U,W;V,X)}\partiso{U}{V}{g_{1}}{g_{2}} ,$$
where $\proba[(Y,Z)]$ is a probability measure, and the $\Psi^{(Y;Z)}_{(U,W;V,X)}$ are operators that yield averages of certain trivial extensions with fixed left and right subspaces. It will be clear from the final formula that the same decomposition holds for the right-hand side of Equation \ref{hardcommutation}, and this will end the proof of Theorem \ref{mainstepan}.
\bigskip

\subsection{Probabilities of  conditioned random trivial extensions}
Let us define  new operators
$$\mathrm{LR}_{(V,W)}^{(V^{+},W^{+})}\partiso{V}{W}{g_{1}}{g_{2}}=\frac{1}{F_{q}(k^{+},k,k_{1})}\sum_{\mathscr{E}\partiso{V^{+}\uparrow V}{W\uparrow W^{+}}{g_{1}\,\,}{g_{2}}} \partext{V^{+}}{W^{+}}{g_{1}^{+}}{g_{2}^{+}},$$
where the sum runs over trivial extensions of $\partiso{V}{W}{g_{1}}{g_{2}}$ with fixed left subspace $V^{+}$ and right subspace $W^{+}$. One has
$$\mathrm{R}_{W}^{W^{+}}\partiso{V}{W}{g_{1}}{g_{2}}=\frac{(q^{k^{+}}-q^{k})\cdots (q^{k^{+}}-q^{k^{+}-1})}{(q^{n}-q^{k})\cdots(q^{n}-q^{k^{+}-1})} \sum_{\substack{V \subset V^{+}\\ \dim V^{+}=\dim W^{+}}} \mathrm{LR}_{(V,W)}^{(V^{+},W^{+})}\partiso{V}{W}{g_{1}}{g_{2}},$$
and similarly for the operators $\mathrm{L}$.\bigskip

\begin{lemma}\label{complicateproba}
Fix two subspaces $U$ and $W$ of $(\For_{q})^{n}$, with $j=\dim U$ and $k=\dim (U+W)$. We choose a random subspace $U^{+}$ of fixed dimension $l \geq j$, uniformly among all such subspaces containing $U$. The law of $m=\dim (U^{+}+W)$ is
$$\proba[m]=q^{(k+l-j-m)(m-n)}\,\frac{(q^{-1})_{n-k}\,(q^{-1})_{n-l}\,(q^{-1})_{k-j}\,(q^{-1})_{l-j}}{(q^{-1})_{k+l-j-m}\,(q^{-1})_{n-m}\,(q^{-1})_{n-j}\,(q^{-1})_{m-k}\,\,(q^{-1})_{m-l}}.$$
\end{lemma}
\begin{proof}
Fix a basis $\Ec=(e_{1},\ldots,e_{n})$ of $(\For_{q})^{n}$ such that $(e_{1},\ldots,e_{j})$ is a basis of $U$, and $(e_{1},\ldots,e_{k})$ is a basis of $(U+W)$. To choose a subspace $U^{+}$ of dimension $l$ uniformly among those that contains $U$, it suffices to choose random vectors $(f_{j+1},\ldots,f_{l})$ among the $(q^{n}-q^{j})\cdots(q^{n}-q^{l-1})$ families of vectors such that $(e_{1},\ldots,e_{j},f_{j+1},\ldots,f_{l})$ is of rank $l$, and to take $U^{+}=\Vect(e_{1},\ldots,e_{j},f_{j+1},\ldots,f_{l})$. Denote 
$$A=\matr_{\Ec}(f_{j+1},\ldots,f_{l})=\begin{pmatrix} a_{1,1} &\cdots & a_{1,l-j} \\ 
\vdots & & \vdots \\
a_{j,1} & \cdots & a_{j,l-j}\\
\vdots & & \vdots \\
a_{k,1} & \cdots & a_{k,l-j}\\
\vdots & & \vdots \\
a_{n,1} & \cdots & a_{n,l-j}\\
 \end{pmatrix}.$$
The condition $\dim U^{+}=l$ is equivalent to $\mathrm{rank}(A')=l-j$, where $A'$ is the submatrix of $A$ that consists in the $(n-j)$ last rows. Then, $m=k+p$, where $p=\mathrm{rank}(A'')$ and $A''$ is the submatrix of $A$ that consists in the $(n-k)$ last rows. So, the law of $(m-k)$ is the law of the rank of the $(n-k)$ last rows of a random matrix of size $(l-j)\times (n-j)$, uniformly chosen among those that are of rank $(l-j)$.\bigskip

Thus, our problem can be reformulated as follows. Fix a dimension $d$, and consider the random process $X_{k}=\mathrm{rank}(v_{1},\ldots,v_{k})$, where the $v_{i}$'s are uniform random vectors of $(\For_{q})^{d}$ chosen independently. The process $(X_{k})_{k \geq 0}$ is a Markov chain with transition matrix
$$p(i,i)=\frac{1}{q^{d-i}}\qquad;\qquad p(i,i+1)=1-\frac{1}{q^{d-i}}.$$
\begin{figure}[ht]
\begin{center}
\includegraphics{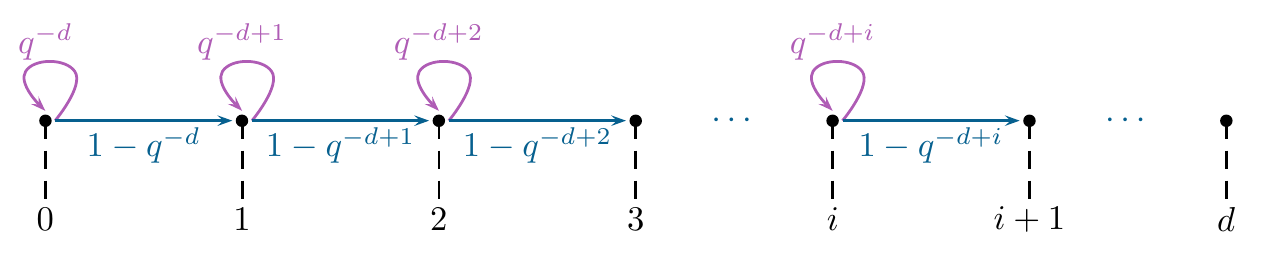}
{\normalsize\caption{The graph of the transitions of the Markov chain $(X_{k})_{k \geq 0}$ in a vector space of dimension $d$ over $\mathbb{F}_{q}$.}}
\end{center}
\end{figure}

We are looking at the probabilities $\proba_{d,q}[X_{a}=c | X_{b}=d]$, with $a \leq b$ and $c \leq d$. Let us first compute $\proba_{d,q}[X_{a}=c]$. To go from rank $0$ to rank $c$ in $a$ steps, one has to go through the $c$ transitions $0 \to 1$, $1 \to 2$, \emph{etc.}, $(c-1) \to c$, which gives a factor
$$\left(1-\frac{1}{q^{d}}\right)\left(1-\frac{1}{q^{d-1}}\right)\cdots\left(1-\frac{1}{q^{d-c+1}}\right)=\frac{(q^{-1})_{d}}{(q^{-1})_{d-c}};$$
and during the other $(a-c)$ transitions, one stays with the same rank. Therefore,
\begin{align*}\proba_{d,q}[X_{a}=c] &= \frac{(q^{-1})_{d}}{(q^{-1})_{d-c}}\,\sum_{\substack{A \,(a-c)\text{-multiset}\\ \text{with elements in } \lle 0,c\rre}}\prod_{x \in A} \frac{1}{q^{d-x}}\\
&=q^{d(c-a)} \frac{(q^{-1})_{d}}{(q^{-1})_{d-c}}\,h_{a-c}(1,q,\ldots,q^{c}),
\end{align*}
where the $h_{k}$'s are the homogeneous symmetric polynomials. We refer to \cite[\S1.2]{Mac95}, where the evaluations of these polynomials on geometric progressions are also computed. One obtains:
$$\proba_{d,q}[X_{a}=c]=q^{(d-c)(c-a)}\,\frac{(q^{-1})_{a}\,(q^{-1})_{d}}{(q^{-1})_{c}\,(q^{-1})_{a-c}\,(q^{-1})_{d-c}}.$$
In particular, 
$\proba_{d,q}[X_{a}=d]=\frac{(q^{-1})_{a}}{(q^{-1})_{a-d}}.$
Now, 
\begin{align*}
\proba_{d,q}[X_{a}=c|X_{b}=d]&=\frac{\proba_{d,q}[X_{b}=d|X_{a}=c] \, \proba_{d,q}[X_{a}=c]}{\proba_{d,q}[X_{b}=d]}\\
&=\frac{\proba_{d-c,q}[X_{b-a}=d-c] \, \proba_{d,q}[X_{a}=c]}{\proba_{d,q}[X_{b}=d]}\\
&= q^{(d-c)(c-a)}\,\frac{(q^{-1})_{a}\,(q^{-1})_{d}\,(q^{-1})_{b-a}\,(q^{-1})_{b-d}}{(q^{-1})_{b}\,(q^{-1})_{c}\,(q^{-1})_{a-c}\,(q^{-1})_{d-c}\,(q^{-1})_{b-a-d+c}} 
\end{align*}
by taking on the second line the quotient of $(\For_{q})^{d}$ by the vector subspace generated by $(v_{1},\ldots,v_{a})$ in order to transform the conditional probability. Finally, we set $a= n-k$, $b=n-j$, $c=p=m-k$, and $d=l-j$:
$$\proba[m]=q^{(k+l-j-m)(m-n)}\,\frac{(q^{-1})_{n-k}\,(q^{-1})_{n-l}\,(q^{-1})_{k-j}\,(q^{-1})_{l-j}}{(q^{-1})_{k+l-j-m}\,(q^{-1})_{n-m}\,(q^{-1})_{n-j}\,(q^{-1})_{m-k}\,\,(q^{-1})_{m-l}}. $$
This expression is non-zero if and only if 
$ \sup(k,l) \leq  m \leq \inf (n, k+l-j);$ moreover, it is symmetric in $k$ and $l$.
\end{proof}
\bigskip

\begin{corollary}\label{uniformproba}
Fix subspaces $U,V,W,X$ with $j=\dim U = \dim V$, $k=\dim (U+W)$ and $l=\dim (V+X)$. One chooses a subspace $U^{+}$ randomly among those containing $U$ and of dimension $l$, and then set $Y=U^{+}+W$ and $Z=(V+X)^{+}$, where $Z$ is again chosen randomly among the subspaces  containing $V+X$ and of dimension $m=\dim Y$. The law of $m$ is given by Lemma \ref{complicateproba}, and conditionally to $m$, the law of $(Y,Z)$ is the uniform law on pairs of spaces of same dimension $m$ and with $U+W \subset Y$ and $V+X \subset Z$. 
\end{corollary}

\begin{proof}
One only needs to verify that conditionally to $m$, the law of $(Y,Z)$ is the uniform law on pairs of spaces containing $U+W$ and $V+X$ (in particular, $Y$ and $Z$ are independent conditionally to their dimension). By construction, $Z=(V+X)^{+}$ is independent of $Y$ conditionally to $m$, and uniformly distributed. Consider then $Y=U^{+}+W$, conditioned by its dimension $m$. We construct $Y$ in the same way as in Lemma \ref{complicateproba}:
\begin{align*}
U&=\Vect(e_{1},\ldots,e_{j})\qquad;\qquad U+W=\Vect(e_{1},\ldots,e_{k})\\
U^{+}&=\Vect(e_{1},\ldots,e_{j},f_{j+1},\ldots,f_{l})
\end{align*}
with $(f_{j+1},\ldots,f_{l})$ uniformly chosen among completions of $(e_{1},\ldots,e_{j})$ into a free family of length $l$. Denote $f_{j+1}^{U},\ldots,f_{l}^{U}$ the images of the random vectors $f_{j+1},\ldots,f_{l}$ by the quotient map
$\pi^{U} : (\For_{q})^{n}  \to (\For_{q})^{n}/U,$
and similarly with $f_{j+1}^{U+W},\ldots,f_{l}^{U+W}$.\bigskip

The family $f_{j+1}^{U},\ldots,f_{l}^{U}$ is a uniform family of $(l-j)$ linearly independent vectors in $(\For_{q})^{n}/U$: indeed, each such family has $q^{(l-j)j}$ preimages by $\pi^{U}$,  and all with the same weight. Denote then $\pi^{U+W}_{U}$ the quotient map $(\For_{q})^{n}/U \to (\For_{q})^{n}/(U+W)$: one has $f^{U+W}_{i}=\pi^{U+W}_{U}(f_{i})$, and on the other hand, to say that $Y|m$ is the uniform random subspace containing $U+W$ and with dimension $m$ is the same as saying that $(f^{U+W}_{j+1},\ldots,f^{U+W}_{l})|m$ is a uniform family of random vectors in $(\For_{q})^{n}/(U+W)$ among those of rank $p=m-k$. Therefore, it suffices to prove the following general statement: if $(e_{1},\ldots ,e_{d})$ is a uniform random family of $d$ linearly independent vectors in a $\For_{q}$-vector space $K$, then knowing the rank $p$ of $(\pi(e_{1}),\ldots,\pi(e_{d}))$ in a quotient $K/H=\pi(K)$, the family $(f_{1},\ldots,f_{d})=(\pi(e_{1}),\ldots,\pi(e_{d}))$ is uniformly distributed among those of rank $p$. Again, this is because each family $(f_{1},\ldots,f_{d})$ with given rank $p$ in $K/H$ has the same number of preimages and all with same weight, namely,
$$q^{nm}\,\proba_{n,q}[X_{n}=d|X_{n-m}=p],$$
where $n=\dim K$, $m=\dim H$, and $(X_{k})_{k \geq 0}$ is the same Markov chain as before.\end{proof}
\bigskip

An important consequence of the previous discussion is that the law of $(Y,Z)$ is the same law on pairs of subspaces of same dimension                                                           $m$ containing $U+W$ and $V+X$ \vspace{2mm}
\begin{enumerate}
\item if one first chooses $U^{+}$ random extension of $U$ of dimension $l$, and then set $Y=U^{+}+W$ and $Z=(V+X)^{+}$\vspace{2mm}
\item or if one first chooses $V^{+}$ random extension of $V$ of dimension $k$, and then set $Z=V^{+}+X$ and $Y=(U+W)^{+}$.\vspace{2mm}
\end{enumerate}
Indeed, this law is $\proba[(Y,Z)]=\proba[m]\,\uni_{m,U+W}[Y] \,\uni_{m,V+X}[Z],$ where $\proba[m]$ is given by Lemma \ref{complicateproba} (remember the symmetry of the roles played by $k$ and $l$), and the $\uni_{m}$'s are the uniform laws on subspaces described by Corollary \ref{uniformproba}. \bigskip

This symmetry plays an essential role in the proof of Theorem \ref{mainstepan}. Fix a partial isomorphism $I=\partiso{U}{V}{g_{1}}{g_{2}}$. One has
\begin{align*}
\mathrm{R}^{X} (I) &= \sum_{\substack{U \subset U^{+},\,\, \dim U^{+} = l}} \uni_{l,U}[U^{+}] \,\,\mathrm{LR}_{(U,V)}^{(U^{+},V+X)}(I)\\
\mathrm{L}^{W}\circ \mathrm{R}^{X}(I) &= \!\sum_{\substack{U \subset U^{+},\,\, \dim U^{+} = l \\ V+X \subset Z,\,\,\dim Z=m}}\!\! \uni_{l,U}[U^{+}] \,\uni_{m,V+X}[Z]\,\,\mathrm{LR}^{(Y,Z)}_{(U^{+},V+X)}\mathrm{LR}_{(U,V)}^{(U^{+},V+X)}(I)
\end{align*}
with $Y=U^{+}+W$ and $m=\dim Y$ on the second line. According to the previous discussion, this is also:
$$\mathrm{L}^{W}\circ\mathrm{R}^{X}(I)=\!\!\!\!\!\!\!\!\!\!\sum_{\substack{\sup(k,l)\leq m \leq \inf(n,k+l-j) \\ \dim Y = \dim Z = m,\,\,U+W \subset Y,\,\,V+X \subset Z\\\dim U^{+}=l,\,\,U \subset U^{+},\,\, Y = U^{+}+W}}\!\!\!\!\!\!\!\!\! \proba[(Y,Z)]\,\proba[U^{+}|(Y,Z)]\,\,\mathrm{LR}^{(Y,Z)}_{(U^{+},V+X)}\mathrm{LR}_{(U,V)}^{(U^{+},V+X)}(I),$$
where $\proba[(Y,Z)]$ is the law on pair on subspaces described by Corollary \ref{uniformproba}. It remains then to understand what is
\begin{equation} \Psi_{(U,W;V,X)}^{(Y;Z)}(I)=\sum_{\dim U^{+}=l,\,\,U \subset U^{+},\,\, Y = U^{+}+W} \proba[U^{+}|(Y,Z)]\,\,\mathrm{LR}^{(Y,Z)}_{(U^{+},V+X)}\mathrm{LR}_{(U,V)}^{(U^{+},V+X)}(I) \label{strangelr} \end{equation}
when $Y$ and $Z$ are fixed.\bigskip

\begin{remark} Notice that $\Psi_{(U,W;V,X)}^{(Y;Z)}(I)$ cannot be the same as $\mathrm{LR}^{(Y,Z)}_{(U,V)}(I)$. A partial isomorphism appearing in \eqref{strangelr} satisfies
$$g_{2}^{++}(V+X)+W = g_{2}^{+}(V+X)+W = U^{+} + W = Y,$$
and in general this is not the case of a trivial extension of $\partiso{U}{V}{g_{1}}{g_{2}}$ to $Y$ and $Z$. Similarly,
\begin{align*}g_{1}^{++}(U+W)+X&=g_{1}(U)+g_{1}^{++}(W)+X=(V+X)+g_{1}^{++}(W)\\
&=g_{1}^{+}(U^{+})+g_{1}^{++}(W)=g_{1}^{++}(Y)=Z,
\end{align*}
which is not in general the case of an extension in $\mathscr{E}\partiso{Y \uparrow U}{V \uparrow Z}{g_{1}}{g_{2}}$. We are actually going to prove that the expression \eqref{strangelr} is the mean of all trivial extensions of $I$ to spaces $Y$ and $Z$ such that $g_{1}^{++}(U+W)+X=Z$ and $g_{2}^{++}(V+X)+W=Y$; see Theorem \ref{lrdecomposition}.
\end{remark}\bigskip

First, we need to understand the distribution of $U^{+}$ knowing $Y$ and $Z$. Obviously, $U^{+}$ is in fact independent from $Z$, so we only need to condition by $Y$. 
\begin{lemma}\label{lemma3co1}
In the previous probabilistic scheme, the law of $U^{+}$ knowing $Y$ is the uniform law on vector subspaces $U^{+}$ satisfying the three conditions
\begin{equation}
\dim U^{+}=l\quad;\quad U \subset U^{+} \subset Y \quad;\quad U^{+}+W=Y.\label{threeconditions}
\end{equation}
The number of such spaces (hence, the inverse of the probability $\proba[U^{+}|Y]$) is
$$
q^{(m-l)(l-j)}\,\frac{(q^{-1})_{k-j}}{(q^{-1})_{m-l}\,(q^{-1})_{k+l-j-m}}.
$$
\end{lemma}
\begin{proof}
We start by counting subspaces $U^{+}$ satisfying the three conditions. It is useful to introduce the quotient map $\pi^{U} : Y \to Y/U$; then, the three conditions \eqref{threeconditions} rewrite as
\begin{equation}
\dim \pi^{U}(U^{+})=l-j \quad;\quad \pi^{U}(U^{+})+\pi^{U}(W)=Y/U.\label{quotientconditions}
\end{equation}
So, given two nested spaces $H \subset K$ of respective dimensions $a $ and $b$, the problem is now to count subspaces $G \subset K$ with fixed dimension $c$ and with $G+H=K$. One will then take 
$$a=k-j \quad;\quad b = m-j \quad;\quad c = l-j \quad;\quad G=\pi^{U}(U^{+}) \quad;\quad H=\pi^{U}(W)\quad;\quad K=Y/U.$$
This problem falls again in the setting of the Markov chain $(X_{k})_{k \geq 0}$. Consider vectors $e_{1},\ldots,e_{c}$, and write their matrix in a basis of $K$ whose first vectors form a basis of $H$. Then, setting $G=\Vect(e_{1},\ldots,e_{c})$, one has $G+H=K$ if and only if the $(b-a)$ last rows of the matrix are of rank $(b-a)$, so $G$ satisfies the previous assumptions in
\begin{align*}q^{bc}\,\proba_{c,q}[X_{b-a}=b-a\,\wedge\,X_{b}=c ]&=q^{bc}\,\,\proba_{c,q}[X_{b-a}=b-a]\,\,\proba_{a+c-b,q}[X_{a}=a+c-b]\\
&=q^{bc}\,\frac{(q^{-1})_{a}\,(q^{-1})_{c}}{(q^{-1})_{b-c}\,(q^{-1})_{a+c-b}}
\end{align*}
cases. Dividing this expression by the number $q^{c^{2}}\,(q^{-1})_{c}$ of bases of $G$, and replacing $a,b,c$, one obtains the number
$$
q^{(m-l)(l-j)}\,\frac{(q^{-1})_{k-j}}{(q^{-1})_{m-l}\,(q^{-1})_{k+l-j-m}}
$$
of subspaces satisfying the two ``quotient'' conditions \eqref{quotientconditions}. Finally, $\pi^{U}$ etablishes a bijection between subspaces of $Y$ containing $U$ and subspaces of $Y/U$, whence the formula in the statement of the lemma.\bigskip

Now, we check that $\proba[U^{+}|Y]$ is indeed the inverse of the previous quantity. Write $\mathrm{Comp}(U^{+},Y)$ for the characteristic function of the three compatibility conditions \eqref{threeconditions}. One has
\begin{align*}
\proba[U^{+}|Y]&=\frac{\proba[(U^{+},Y)]}{\proba[Y]} = \frac{\uni_{l,U}[U^{+}]\,\mathrm{Comp}(U^{+},Y)}{\proba[m]\,\uni_{m,U+W}[Y]}\\
&=q^{(l-m)(l-j)}\,\frac{(q^{-1})_{m-l}\,(q^{-1})_{k+l-j-m}}{(q^{-1})_{k-j}}\,\mathrm{Comp}(U^{+},Y)
\end{align*}
by using Lemma \ref{complicateproba} for $\proba[m]$, and twice the formula $\uni_{a,b,c}=\frac{(q^{-1})_{c-b}\,(q^{-1})_{b-a}}{q^{(c-b)(b-a)}\,(q^{-1})_{c-a}}$ for the uniform probability of a random space of dimension $b$ contained in one of dimension $c$, and containing a fixed subspace of dimension $a$.
\end{proof}
\bigskip

To reintroduce some symmetry in Equation \eqref{strangelr}, let us fix $U^{+}$, $Y$ and $Z$, and for $\partext{Y}{Z}{g_{1}^{++}}{g_{2}^{++}}$ appearing in $\mathrm{LR}_{(U^{+},V+X)}^{(Y,Z)}\,\mathrm{LR}_{(U,V)}^{(U^{+},V+X)}(I)$,  consider the space 
$$V^{+}=g_{1}^{++}(U+W).$$ Since $\mathrm{LR}_{(U^{+},V+X)}^{(Y,Z)}\,\mathrm{LR}_{(U,V)}^{(U^{+},V+X)}(I)$ is an average of partial isomorphisms, one can then consider $V^{+}$ as a random subspace of $Z$, with probability conditioned by $U^{+}$, $Y$ and $Z$.
\begin{lemma}\label{lemma3co2}
The law of $V^{+}$ knowing $U^{+}$, $Y$ and $Z$ is the uniform law on subspaces of $Z$ satisfying the three conditions
$$
\dim V^{+}=k\quad;\quad V \subset V^{+} \subset Z \quad;\quad V^{+}+X=Z.
$$
\end{lemma}
\begin{proof}
As in the previous Lemma, we have to show that $\proba[V^{+}|(U^{+},Y,Z)]$ is equal to the inverse of 
$$
q^{(m-k)(k-j)}\,\frac{(q^{-1})_{l-j}}{(q^{-1})_{m-k}\,(q^{-1})_{k+l-j-m}},
$$
and on the other hand,
$$
\proba[V^{+}|(U^{+},Y,Z)]=\frac{\proba[(U^{+},V^{+},Y,Z)]}{\proba[(U^{+},Y,Z)]} = \frac{\proba[(U^{+},V^{+})]\,\mathrm{Comp}(U^{+},Y)\,\mathrm{Comp}(V^{+},Z)}{\proba[m]\,\uni_{m,U+W}[Y]\,\uni_{m,V+X}[Z]\,\proba[U^{+}|Y]}
$$
since $Y$ and $Z$ are entirely determined by $U^{+}$ and $V^{+}$ and by the compatibility conditions. The denominator can be calculated by using the previous results:
$$ \proba[(U^{+},Y,Z)]=q^{(m-n)(m-j)+(l-m)(l-j)}\,\frac{(q^{-1})_{n-m}\,(q^{-1})_{m-l}\,(q^{-1})_{l-j}}{(q^{-1})_{n-j}}.$$
Therefore, the claim of the lemma is equivalent to
\begin{equation}\proba[(U^{+},V^{+})]=q^{(m-n)(m-j)+(l-m)(l-j)+(k-m)(k-j)}\,\frac{(q^{-1})_{n-m}\,(q^{-1})_{m-l}\,(q^{-1})_{m-k}\,(q^{-1})_{k+l-j-m}}{(q^{-1})_{n-j}},\label{probuplusvplus}
\end{equation}
where the probability refers to the random choices of spaces and trivial extensions that are done during the computation of $\mathrm{L}^{W}\circ \mathrm{R}^{X}(I)$.\bigskip

We decompose this succession of random choices as follows, keeping track at each step of the probability to obtain a trivial extension $\partext{Y}{Z}{g_{1}^{++}}{g_{2}^{++}}$ with given subspaces $U^{+}$ and $V^{+}$.\vspace{2mm}
\begin{enumerate}
\item When applying $\mathrm{R}^{X}$, there is a probability $$\frac{(q^{-1})_{n-l}\,(q^{-1})_{l-j}}{q^{(n-l)(l-j)}\,(q^{-1})_{n-j}}$$ to obtain a given space $U^{+}$. Denote then $W'$ a subspace of $W$ such that $U+W=U \oplus W'$; it is of dimension $(k-j)$, and
$$\dim (U^{+}\cap W') = \dim U^{+}+\dim W' - \dim(U^{+}+W)=k+l-j-m.$$
Fix a basis $(e_{1},\ldots,e_{l})$ of $U^{+}$, such that 
\begin{align*}
&(e_{1},\ldots,e_{j}) \text{ is a basis of }U;\\
&(e_{j+1},\ldots,e_{k+l-m}) \text{ is a basis of }U^{+}\cap W'.
\end{align*} The (random) trivial extension $\partext{U^{+}}{V+X}{g_{1}^{+}}{g_{2}^{+}}$ of $\partiso{U}{V}{g_{1}}{g_{2}}$ is obtained by first choosing a completion $(f_{j+1},\ldots,f_{l})=(g_{1}^{+}(e_{j+1}),\ldots,g_{1}^{+}(e_{l}))$ of the basis $(f_{1},\ldots,f_{j})=(g_{1}(e_{1}),\ldots,g_{1}(e_{l}))$ of $V$ into a basis of $V+X$, and then by choosing the matrix of $g_{2}^{+}$. For the computation of $V^{+}$, one only needs to know what is $g_{1}^{+}(U^{+}\cap W')$. This is a subspace of $V+X$ of dimension $(k+l-j-m)$, linearly independent from $g_{1}(U)=V$, and arbitrary among those such subspaces, because its basis $(f_{j+1},\ldots,f_{k+l-m})$ is itself arbitrary and uniformly distributed. There are 
$$\frac{(q^{l}-q^{j})\cdots(q^{l}-q^{k+l-m-1})}{(q^{k+l-j-m}-1)\cdots(q^{k+l-j-m}-q^{k+l-j-m-1})}=\frac{q^{(m+j-k)(k+l-m-j)} \,(q^{-1})_{l-j}}{(q^{-1})_{m-k}\,(q^{-1})_{k+l-j-m}}$$
such subspaces, each with the same probability. So, the extension operator $\mathrm{R}^{X}$ yields the subspace $U^{+}$ and an arbitrary subspace $g_{1}^{+}(U^{+}\cap W')$ of $V+X$ of dimension $(k+l-j-m)$ and linearly independent from $V$, with probability
\begin{equation}
\frac{(q^{-1})_{n-l}\,(q^{-1})_{m-k}\,(q^{-1})_{k+l-j-m}}{q^{(n-l)(l-j)+(m+j-k)(k+l-j-m)}\,(q^{-1})_{n-j}}.\label{intermediarystep}
\end{equation}\vspace{2mm}
\item Let us now apply $\mathrm{L}^{W}$. Fix $W''\subset W'$ a complement of $U^{+} \cap W'$ inside $W'$, so that $\dim W''=m-l$ and
\begin{align*}V^{+}&=g_{1}^{++}(U+W)=g_{1}^{++}(U \oplus W')=V\oplus g_{1}^{++}(W')\\
&=V\oplus g_{1}^{+} (U^{+}\cap W')  \oplus g_{1}^{++}(W'').
\end{align*}
By similar arguments as before, $g_{1}^{++}(W'')$ is an arbitrary subspace of $(\For_{q})^{n}$ that is supplementary to $V+X$ and of dimension $(m-l)$, so it has probability
$$q^{(n+l-m)(l-m)} \frac{(q^{-1})_{n-m}\,(q^{-1})_{m-l}}{(q^{-1})_{n-l}}.$$
From this and Equation \eqref{intermediarystep}, we see that a given choice of $U^{+}$, $g_{1}^{+}(U^{+}\cap W')$ and $g_{1}^{++}(W'')$ has probability
\begin{equation}
q^{(n+l-m)(l-m)+(l-n)(l-j)+(m+j-k)(m+j-k-l)}\,\frac{(q^{-1})_{n-m}\,(q^{-1})_{m-l}\,(q^{-1})_{m-k}\,(q^{-1})_{k+l-j-m}}{(q^{-1})_{n-j}}.\label{intermediarystep2}
\end{equation}
It remains to see how many choices of $g_{1}^{+}(U^{+}\cap W')$ and $g_{1}^{++}(W'')$ yield the same $V^{+}$. The first part 
$$V\oplus g_{1}^{+}(U^{+}\cap W')=g_{1}^{+}(U^{+}\cap(U+W))$$
is a subspace of $V+X$ of dimension $(k+l-m)$, and it does not change if one takes for $g_{1}^{+}(U^{+}\cap W')$ another complement subspace of $V$ inside $g_{1}^{+}(U^{+}\cap(U+W))$. So, one will have to multiply Equation \eqref{intermediarystep2} by the factor $q^{j(k+l-m-j)}$. One will also have to multiply this probability by the factor $q^{(m-l)(k+l-m)}$, which is the number of complement subspaces $g_{1}^{++}(W')$ of $g_{1}^{+}(U^{+}\cap(U+W))$ inside $V^{+}$. So in the end,
$$\proba[(U^{+},V^{+})] = q^{(n+l-m)(l-m)+(l-n)(l-j)+(m+j-k)(m+j-k-l)+j(k+l-m-j)+(m-l)(k+l-m)}\,R_{q},$$
where $R_{q}$ is the quotient of symbols $(q^{-1})_{a}$ appearing in Equation \eqref{intermediarystep2} --- notice that it is also the expected quotient from Formula \eqref{probuplusvplus}. It is now easy to see that the power of $q$ simplifies as $(m-n)(m-j)+(l-m)(l-j)+(k-m)(k-j)$.\qedhere
\end{enumerate}
\end{proof}
\bigskip

Denote $\mathbb{C}_{l,U,W,Y}$ the uniform probability on subspaces $U^{+}$ satisfying the three conditions of compatibility \eqref{threeconditions}. Lemmas \ref{lemma3co1} and \ref{lemma3co2} show that $\Psi_{(U,W;V,X)}^{(Y;Z)}(I)$ is equal to
$$\sum_{\substack{U^{+},V^{+}\\g_{1}^{++},g_{2}^{++}}} \mathbb{C}_{l,U,W,Y}[U^{+}]\,\mathbb{C}_{k,V,X,Z}[V^{+}]\,\proba[(g_{1}^{++},g_{2}^{++})|(Y,Z,U^{+},V^{+})] \,\partext{Y}{Z}{g_{1}^{++}}{g_{2}^{++}}.$$
It remains to see that the last conditional expectation is in fact the uniform probability on trivial extensions with fixed left and right subspaces $Y$ and $Z$, and sending $U+W$ to $V^{+}$ and $V+X$ to $U^{+}$ --- we shall say that these extensions are compatible with $Y$, $Z$, $U^{+}$ and $V^{+}$.\bigskip

We shall use the following counting argument. In Equation \eqref{strangelr}, each partial isomorphism that appears has a subspace $U^{+}$ which is entirely determined, because it is equal to $g_{2}^{++}(V+X)$. Therefore, $\Psi_{(U,W;V,X)}^{(Y,Z)}(I)$ is a linear combination of
$$\frac{q^{(m-l)(l-j)}\,(q^{-1})_{k-j}}{(q^{-1})_{m-l}\,(q^{-1})_{k+l-j-m}}\,F_{q}(m,l,l-j+j_{1})\,F_{q}(l,j,j_{1})=
q^{(m+j-j_{1})(m-j)}\,\frac{(q^{-1})_{k-j}\,(q^{-1})_{l-j}}{(q^{-1})_{k+l-j-m}}
$$
distinct partial isomorphisms, each with the same weight (the inverse of this number of terms). So, $\proba[(g_{1}^{++},g_{2}^{++})|(Y,Z,U^{+},V^{+})]$, which is a probability on trivial extensions compatible with $Y$, $Z$, $U^{+}$ and $V^{+}$, is equal either to $0$ or to the constant
$$P  = \frac{ q^{(m-k)(k-j)+(m-l)(l-j)-(m+j-j_{1})(m-j)}}{(q^{-1})_{m-k}\,(q^{-1})_{m-l}\,(q^{-1})_{k+l-j-m}}.$$
As a consequence, it suffices to show:
\begin{lemma}
The number of trivial extensions compatible with $Y$, $Z$, $U^{+}$ and $V^{+}$ is $$\frac{1}{P}=q^{(m+j-j_{1})(m-j)-(m-k)(k-j)-(m-l)(l-j)}\,(q^{-1})_{m-k}\,(q^{-1})_{m-l}\,(q^{-1})_{k+l-j-m}.$$ 
So, $\proba[(g_{1}^{++},g_{2}^{++})|(Y,Z,U^{+},V^{+})]$ is never equal to $0$ for a compatible extension, and it is the uniform probability on these compatible trivial extensions.
\end{lemma}
\begin{proof}
With the same notations as in Lemma \ref{lemma3co1}, one chooses a basis $(e_{1},\ldots,e_{m})$ of $Y$ such that:
\begin{align*}
&(e_{1},\ldots,e_{j}) \text{ is a basis of }U;\\
&(e_{1},\ldots,e_{l}) \text{ is a basis of }U^{+};\\
&(e_{l+1},\ldots,e_{m})\text{ is a basis of }W'';\\
&(e_{m+j-k+1},\ldots,e_{l})\text{ is a basis of }U^{+}\cap W'; \\
&(e_{m+j-k+1},\ldots,e_{m})\text{ is a basis of }W'. 
\end{align*}
Fix a compatible trivial extension, and denote $f_{i}=g_{1}^{++}(e_{i})$.\vspace{2mm}
\begin{enumerate}
\item The first $j$ vectors $(f_{1},\ldots,f_{j})$ are fixed since $(g_{1}^{++})_{|U}=g_{1}$, so there is no choice there.\vspace{2mm}
\item The $(k+l-j-m)$ vectors $(f_{m+j-k+1},\ldots,f_{l})$ form a basis of a supplementary of $V$ in $V^{+}\cap(V+X)$; here one has 
$$
(q^{k+l-m}-q^{j})\cdots(q^{k+l-m}-q^{k+l-m-1})=q^{(k+l-m)(k+l-j-m)}\,(q^{-1})_{k+l-j-m}
$$
possibilities.\vspace{2mm}
\item The $(m-k)$ vectors $(f_{j+1},\ldots,f_{m+j-k})$ form a basis of a complement subspace of $V^{+}\cap(V+X)$ inside $V+X$; this gives $q^{l(m-k)}\,(q^{-1})_{m-k}$ possibilities.\vspace{2mm}
\item Finally, the $(m-l)$ vectors $(f_{l+1},\ldots,f_{m})$ form a basis of a complement subspace  of $V^{+}\cap(V+X)$ inside $V^{+}$; this gives $q^{k(m-l)}\,(q^{-1})_{m-l}$ possibilities.\vspace{2mm}
\end{enumerate}
So, there are $$q^{k(m-l)+l(m-k)+(k+l-m)(k+l-j-m)}\,(q^{-1})_{m-k}\,(q^{-1})_{m-l}\,(q^{-1})_{k+l-j-m}$$ possibilities for $g_{1}^{++}$. Then, with respect to the two bases $\Ec$ and $\Fc$ previously chosen, Condition $(4)$ in Definition \ref{partisodef} ensures that a trivial extension of $\partiso{U}{V}{g_{1}}{g_{2}}$ to spaces $Y$ and $Z$ is given by a matrix
$$\matr_{\Fc,\Ec}(g_{2}^{++})=\begin{pmatrix}
G & (G-I_{m-j})\,R \\
0 &I_{m-j}
\end{pmatrix},$$
with $R$ rectangular matrix. This rectangular matrix $R$ can really be chosen arbitrary, because the conditions of compatibility have all been ensured by the choice of $g_{1}^{++}$. So, there are $q^{(m-j)(j-j_{1})}$ possibilities for $(G-I_{m-j})\,R$. This last factor multiplied by the previous quantity gives indeed
$$q^{(m+j-j_{1})(m-j)-(m-k)(k-j)-(m-l)(l-j)}\,(q^{-1})_{m-k}\,(q^{-1})_{m-l}\,(q^{-1})_{k+l-j-m}$$
distinct trivial extensions that are compatible with $Y$, $Z$, $U^{+}$ and $V^{+}$.
\end{proof}
\bigskip

Thus, we have proved the following decomposition theorem for $\mathrm{L}^{W}\circ \mathrm{R}^{X}$:
\begin{theorem}\label{lrdecomposition}
If $I=\partiso{U}{V}{g_{1}}{g_{2}}$, then 
\begin{align*}
\mathrm{L}^{W}\circ \mathrm{R}^{X}(I)&= \sum \proba[m]\,\uni_{m,U+W}[Y]\,\uni_{m,V+X}[Z]\,\Psi_{(U,W;V,X)}^{(Y,Z)}(I)\\
\Psi_{(U,W;V,X)}^{(Y,Z)}(I) &= \sum \mathbb{C}_{l,U,W,Y}[U^{+}]\,\mathbb{C}_{k,V,X,Z}[V^{+}]\,\mathbb{D}_{\substack{U,U^{+},W\\V,V^{+},X}}(g_{1}^{++},g_{2}^{++})\,\partext{Y}{Z}{g_{1}^{++}}{g_{2}^{++}},
\end{align*}
where:\vspace{2mm}
\begin{itemize}
\item $m$ runs over $\lle \sup(k,l),\inf(n,k+l-j)\rre$, and it has probability $\proba[m]$ given by Lemma \ref{complicateproba};\vspace{2mm}
\item $(Y,Z)$ runs over pair of subspaces of dimension $m$ containing respectively $U+W$ and $V+X$, and it has uniform probability $\uni_{m,U+W}[Y]\,\uni_{m,V+X}[Z]$ among these pairs of subspaces;\vspace{2mm}
\item $U^{+}$ (respectively, $V^{+}$) runs over subspaces satisfying the three conditions of Lemma \ref{lemma3co1} (respectively, Lemma \ref{lemma3co2}), and it has uniform probability $\mathbb{C}_{l,U,W,Y}[U^{+}]$ (resp., $\mathbb{C}_{k,V,X,Z}[V^{+}]$) among them;\vspace{2mm}
\item and $(g_{1}^{++},g_{2}^{++})$ runs over trivial extensions of $(g_{1},g_{2})$ to $Y$ and $Z$ with the additional constraints $$g_{1}^{++}(U+W)=V^{+} \qquad;\qquad g_{2}^{++}(V+X)=U^{+}\qquad;$$ and it has uniform probability $\mathbb{D}_{\substack{U,U^{+},W\\V,V^{+},X}}(g_{1}^{++},g_{2}^{++})$ among these trivial extensions.\vspace{2mm}
\end{itemize}
\end{theorem}
\noindent Since this description is symmetric in the entries of $I$ on the left and on the right, one obtains the same expansion for $\mathrm{R}^{X}\circ \mathrm{L}^{W} (I)$, so Theorem \ref{mainstepan} is true.
\bigskip

\subsection{Proof of the associativity}
Fix three partial isomorphisms $G=\partiso{S}{T}{g_{1}}{g_{2}}$, $H=\partiso{U}{V}{h_{1}}{h_{2}}$ and $I=\partiso{W}{X}{i_{1}}{i_{2}}$, of respective dimensions $j_{G}$, $j_{H}$ and $j_{L}$. The last step before the proof of Theorem \ref{mainan} is a more concrete description of
$\mathrm{R}^{W}(\mathrm{R}^{U}(G)\cdot \mathrm{L}^{T}(H))$, where $\cdot$ indicates that one takes the product of two partial isomorphisms such that the right subspace of the left-hand side corresponds to the left subspace of the right-hand side (this restriction $\cdot$ of the product $*$ is clearly associative). Recall the description 
$$\mathrm{R}^{W}\mathrm{L}^{T}(H)=\sum \proba[m]\,\uni_{m,T+U}[Y]\,\uni_{m,V+W}[Z]\,\Psi_{(U,T;V,W)}^{(Y;Z)}(H),$$
where $\proba[m]$ is an explicit probability on integers between $\max(\dim(T+U),\dim(V+W))$ and $\min(n,\dim(T+U)+\dim(V+W)-j_{H})$, see Lemma \ref{complicateproba}.
\begin{proposition}\label{RRL}
With the same notations,
$$\mathrm{R}^{W}(\mathrm{R}^{U}(G)\cdot \mathrm{L}^{T}(H))=\sum \proba[m]\,\uni_{m,T+U}[Y]\,\uni_{m,V+W}[Z]\,\left(\mathrm{R}^{Y}_{T}(G)\cdot\Psi_{(U,T;V,W)}^{(Y;Z)}(H)\right).$$
\end{proposition}
\medskip

\begin{lemma}\label{lastlemmaan}
Consider two partial isomorphisms $G$ defined over spaces $A$ and $B$, and $H$ defined over spaces $D$ and $E$. If $A\subset A^{+} \subset A^{++}$, $E \subset E^{+} \subset E^{++}$, and $B+D \subset C^{+}$, then
\begin{align*}&\mathrm{LR}_{(A^{+},E^{+})}^{(A^{++},E^{++})}\left(\mathrm{LR}_{(A,B)}^{(A^{+},C^{+})}(G)\cdot \mathrm{LR}_{(D,E)}^{(C^{+},E^{+})}(H)\right) \\
&= \sum\, \uni_{m,C^{+}}[C^{++}]\,\left(\Phi^{(A^{++};C^{++})}_{(A,A^{+};B,C^{+})}(G)\cdot \Phi_{(D,C^{+};E,E^{+})}^{(C^{++};E^{++})}(H)\right),
\end{align*}
where $m=\dim A^{++}=\dim E^{++}$; and $$\Phi^{(A^{++};C^{++})}_{(A,A^{+};B,C^{+})}(G)=\mathrm{LR}_{(A^{+},C^{+})}^{(A^{++},C^{++})}\mathrm{LR}_{(A,B)}^{(A^{+},C^{+})}(G)$$ is the mean of all trivial extensions of $G$ to spaces $A^{++}$ and $C^{++}$ that send $A^{+}$ to $C^{+}$ (and similarly for $\Phi_{(D,C^{+};E,E^{+})}^{(C^{++};E^{++})}(H)$).
\end{lemma}

\begin{proof}
Let $G=\partiso{U}{V}{g_{1}}{g_{2}}$ and $H=\partiso{V}{W}{h_{1}}{h_{2}}$ be two partial isomorphisms of same dimension $k$, and $m=\dim U^{+}=\dim W^{+}$ with $U\subset U^{+}$ and $W\subset W^{+}$. We claim that:
\begin{equation}
\mathrm{LR}_{(U,W)}^{(U^{+},W^{+})}(G \cdot H)=\sum\, \uni_{m,V}[V^{+}]\, \left(\mathrm{LR}_{(U,V)}^{(U^{+},V^{+})}(G)\cdot \mathrm{LR}_{(V,W)}^{(V^{+},W^{+})}(H)\right).\label{entersrandomness}
\end{equation}
 To begin with, notice that both sides of \eqref{entersrandomness} are averages of partial isomorphisms contained in the set $\mathscr{E}\partiso{U^{+}\uparrow U}{W \uparrow W^{+}}{g_{1}h_{1}}{h_{2}g_{2}}$. This is obvious for the left-hand side, and for the right-hand side, if $\partext{U^{+}}{V^{+}}{g_{1}^{+}}{g_{2}^{+}}$ and $\partext{V^{+}}{W^{+}}{h_{1}^{+}}{h_{2}^{+}}$ are two trivial extensions of $G$ and $H$, then with the notations of Definition \ref{partisodef}, $\widetilde{g_{1}}\widetilde{g_{2}}=\id_{U^{+}/U}$ and $\widetilde{h_{1}}\widetilde{h_{2}}=\id_{V^{+}/V}$, so
$$\widetilde{g_{1}h_{1}}\,\widetilde{h_{2}g_{2}} = \widetilde{g_{1}}\left(\widetilde{h_{1}}\widetilde{h_{2}}\right)\widetilde{g_{2}}=\widetilde{g_{1}}\widetilde{g_{2}}=\id_{U^{+}/U},$$
which means that the product of trivial extensions is a trivial extension. Therefore, it suffices to show that every element of $\mathscr{E}\partiso{U^{+}\uparrow U}{W \uparrow W^{+}}{g_{1}h_{1}}{h_{2}g_{2}}$ appears in the right-hand side of \eqref{entersrandomness}, and with the same weight. Fix two bases $\Ec^{+}$ and $\Gc^{+}$ of $U^{+}$ and $W^{+}$ such that 
$$\matr_{\Ec,\Gc}(g_{1}h_{1}) = I_{k} \qquad;\qquad \matr_{\Gc,\Ec}(h_{2}g_{2})=GH.$$
By Lemma \ref{importantbijection}, a trivial extension of $(G\cdot H)$ is given by two matrices
\begin{align*}\matr_{\Ec^{+},\Gc^{+}}((g_{1}h_{1})^{+})=\begin{pmatrix} I_{k} & R_{1} \\
0& K\end{pmatrix} \quad\text{with }K \text{ invertible and }R_{1}\text{ rectangular};\\
\matr_{\Gc^{+},\Ec^{+}}((h_{2}g_{2})^{+}) =  \begin{pmatrix} GH & (GH-I_{k}) \,R_{2} \\ 0 & I_{m-k}\end{pmatrix} \begin{pmatrix} I & -R_{1}K^{-1} \\
0& K^{-1}\end{pmatrix} \quad\text{with }R_{2}\text{ rectangular}.
\end{align*}
Thus, introducing the two subgroups of $\GL(m,\For_{q})$
\begin{align*}\mathrm{P}(m,k)&=\left\{\begin{pmatrix} I_{k} & R \\ 0 & K\end{pmatrix},\,\,\text{ with }K \text{ invertible and }R\text{ rectangular}\right\};\\
\mathrm{N}(m,GH)&=\left\{\begin{pmatrix} I_{k} & (GH-I_{k})\,R \\ 0 & I_{m-k}\end{pmatrix},\,\,\text{ with }R\text{ rectangular}\right\},
\end{align*}
the group $\mathrm{P}(m,k) \times \mathrm{N}(m,GH)$ acts on $\Ec\partiso{U^{+}\uparrow U}{W \uparrow W^{+}}{g_{1}h_{1}}{h_{2}g_{2}}$ in a free transitive way. Consequently, it suffices to verify that the right-hand side of Equation \eqref{entersrandomness} is invariant under this action. In the following we denote RHS this right-hand side.\bigskip

Set $f_{i}=g_{1}(e_{i})$ for $i \in \lle 1,k\rre$, and
\begin{align*}
\matr_{\Ec,\Fc}(g_{1})&=I_{k}\qquad;\qquad \matr_{\Fc,\Ec}(g_{2})=G\\
\matr_{\Fc,\Gc}(h_{1})&=I_{k}\qquad;\qquad \matr_{\Gc,\Fc}(h_{2})=H.
\end{align*}
To choose randomly $V^{+}$ in RHS, one can choose a random completion $(f_{k+1},\ldots,f_{m})$ of $(f_{1},\ldots,f_{k})$ into a family of $m$ independent vectors of $(\For_{q})^{n}$, and then set $V^{+}=\Vect(f_{1},\ldots,f_{m})$. This allows to choose simultaneously the extension $g_{1}^{+}$ by requiring that $\matr_{\Ec^{+},\Fc^{+}}(g_{1}^{+})=I_{m}$. Then, in RHS, conditionally to the previous choice of $\Fc^{+}$,
\begin{align*}
\matr_{\Fc^{+},\Gc^{+}}(h_{1}^{+})&=\begin{pmatrix} I_{k} & R_{1} \\ 0 & K \end{pmatrix} ;\\
\matr_{\Gc^{+},\Fc^{+}}(h_{2}^{+})&= \begin{pmatrix} H & (H-I_{k})\,R_{2} \\ 0 & I_{m-k} \end{pmatrix} \begin{pmatrix} I_{k} & -R_{1}K^{-1} \\ 0 & K^{-1} \end{pmatrix};\\
\matr_{\Fc^{+},\Ec^{+}}(g_{2}^{+}) &= \begin{pmatrix} G & (G-I_{k})\,R_{3} \\ 0 & I_{m-k}\end{pmatrix}
\end{align*}
where $K$ has uniform law on $\GL(m-k,\For_{q})$, and $R_{1}$, $R_{2}$ and $R_{3}$ are independent and uniformly distributed rectangular matrices in $\mathrm{M}(k \times (m-k), \For_{q})$. So,
\begin{align*}
\text{RHS}&=\sum \uni[\Fc^{+},K,R_{1},R_{2},R_{3}]\,\left( \left(\begin{smallmatrix} I_{k} & R_{1} \\ 0 & K \end{smallmatrix}\right) \rightleftarrows \left(\begin{smallmatrix} G & (G-I_{k})\,R_{3}\\ 0 & I_{m-k}  \end{smallmatrix}\right)\left(\begin{smallmatrix}  H & (H-I_{k})\,R_{2}\\ 0 & I_{m-k} \end{smallmatrix}\right)\left(\begin{smallmatrix}  I_{k} & -R_{1}K^{-1} \\ 0 & K^{-1}  \end{smallmatrix}\right) \right)\\
&=\sum\uni[\Fc^{+},K,R_{1},R_{2},R_{3}]\,\left( \left(\begin{smallmatrix} I_{k} & R_{1} \\ 0 & K \end{smallmatrix}\right) \rightleftarrows\left(\begin{smallmatrix}  GH\, & (GH-G)\,R_{2}+(G-I_{k})\,R_{3}\\ 0 & I_{m-k} \end{smallmatrix}\right)\left(\begin{smallmatrix}  I_{k} & -R_{1}K^{-1} \\ 0 & K^{-1}  \end{smallmatrix}\right) \right)
\end{align*}
where the $\uni$'s denote the previously described uniform laws. Now, note that the random free family $\Fc^{+}$ does not appear anymore in the description of the partial isomorphisms from $U^{+}$ to $W^{+}$. Thus, 
$$\text{RHS}=\sum\uni[K,R_{1},R_{2},R_{3}]\,\left( \left(\begin{smallmatrix} I_{k} & R_{1} \\ 0 & K \end{smallmatrix}\right) \rightleftarrows\left(\begin{smallmatrix}  GH\, & (GH-G)\,R_{2}+(G-I_{k})\,R_{3}\\ 0 & I_{m-k} \end{smallmatrix}\right)\left(\begin{smallmatrix}  I_{k} & -R_{1}K^{-1} \\ 0 & K^{-1}  \end{smallmatrix}\right) \right),$$
with the matrices written w.r.t. the two bases $\Ec^{+}$ and $\Gc^{+}$. In this representation of RHS, the invariance by action of $\mathrm{P}(m,k)$ is obvious. On the other hand, if one makes act an element $\left(\begin{smallmatrix} I_{k}& (GH-I_{k})\,R\\ 0 & I_{m-k} \end{smallmatrix}\right)$ of $\mathrm{N}(m,GH)$, then the random matrix
$$\left(\begin{smallmatrix}  GH\, & (GH-G)\,R_{2}+(G-I_{k})\,R_{3}\\ 0 & I_{m-k} \end{smallmatrix}\right)\quad\text{
becomes}\quad
\left(\begin{smallmatrix}  GH\, & (GH-G)\,(R_{2}+R)+(G-I_{k})\,(R_{3}+R)\\ 0 & I_{m-k} \end{smallmatrix}\right).$$
Since $R_{2}$ and $R_{3}$ are independent uniformly distributed rectangular matrices, the law of $(R_{2}+R,R_{3}+R)$ is the same as the law of $(R_{2},R_{3})$, so the invariance by action of $\mathrm{N}(m,GH)$ is also shown. Finally,  Lemma \ref{lastlemmaan} follows from the simpler and more general identity \eqref{entersrandomness} by expanding linearly $\mathrm{LR}_{(A,B)}^{(A^{+},C^{+})}(G)$ and $\mathrm{LR}_{(D,E)}^{(C^{+},E^{+})}(H)$.
\end{proof}
\bigskip

\begin{proof}[Proof of Proposition \ref{RRL}]
Set $k=\dim (T+U)$ and $l=\dim (V+W)$. One has by definition of the extension operators:
$$\mathrm{R}^{U}(G)\cdot \mathrm{L}^{T}(H) = \sum \uni_{k,S}[S^{+}] \,\uni_{k,V}[V^{+}]\,\left(\mathrm{LR}_{(S,T)}^{(S^{+},T+U)}(G)\cdot \mathrm{LR}_{(U,V)}^{(T+U,V^{+})}(H)\right).$$
By Lemma \ref{complicateproba}, the law of $m=\dim (V^{+}+W)$ is given by $\proba[m]$, and then by Corollary \ref{uniformproba} and Lemma \ref{lemma3co1}, one obtains for $\mathrm{R}^{W}(\mathrm{R}^{U}(G)\cdot \mathrm{L}^{T}(H))$
$$\sum \proba[m]\,\uni_{k,S}[S^{+}]\,\uni_{m,S^{+}}[S^{++}] \,\uni_{m,V+W}[Z]\,\mathbb{C}_{k,V,W,Z}[V^{+}]\,\Theta_{(m,S^{+},S^{++},Z,V^{+})}(G,H)$$
with
\begin{align*}
 \Theta_{(m,S^{+},S^{++},Z,V^{+})}(G,H) &= \mathrm{LR}_{(S^{+},V^{+})}^{(S^{++},Z)}\left(\mathrm{LR}_{(S,T)}^{(S^{+},T+U)}(G)\cdot \mathrm{LR}_{(U,V)}^{(T+U,V^{+})}(H)\right) \\ 
&=\sum \uni_{m,T+U}[Y]\,\left(\Phi^{(S^{++},Y)}_{(S,S^{+};T,T+U)}(G)\cdot \Phi_{(U,T+U;V,V^{+})}^{(Y;Z)}(H)\right)
\end{align*}
by Lemma \ref{lastlemmaan}. Therefore, by changing the order of summation, 
$$
\mathrm{R}^{W}(\mathrm{R}^{U}(G)\cdot \mathrm{L}^{T}(H))=\sum \proba[m]\,\uni_{m,T+U}[Y]\,\uni_{m,V+W}[Z]\,\mathbb{C}_{k,V,W,Z}[V^{+}]\, \Xi_{(m,Y,Z,V^{+})}(G,H),
$$
where
\begin{align*}
\Xi_{(m,Y,Z,V^{+})}(G,H)&=\sum \uni_{k,S}[S^{+}]\,\uni_{m,S^{+}}[S^{++}] \,\left(\Phi^{(S^{++};Y)}_{(S,S^{+};T,T+U)}(G)\cdot \Phi_{(U,T+U;V,V^{+})}^{(Y;Z)}(H)\right)\\
&=\sum\uni_{m,S}[S^{++}] \,\left(\mathrm{LR}^{(S^{++},Y)}_{(S,T)}(G)\cdot \Phi_{(U,T+U;V,V^{+})}^{(Y;Z)}(H)\right)\\
&=\mathrm{R}_{T}^{Y}(G)\cdot \Phi_{(U,T+U;V,V^{+})}^{(Y;Z)}(H).
\end{align*}
Finally, by using the remark at the beginning of the proof of Lemma \ref{lastlemmaan}, the definition of the operators $\Psi$, and Lemma \ref{lemma3co1},
$$\sum \mathbb{C}_{k,V,W,Z}[V^{+}] \,\Phi_{(U,T+U;V,V^{+})}^{(Y;Z)}(H)=\Psi_{(U,T;V,W)}^{(Y;Z)}(H),$$
which leads to the expansion stated in the Proposition.
\end{proof}\bigskip

\begin{proof}[Proof of Theorem \ref{mainan}]
Proposition \ref{RRL} shows that
\begin{align*}
(G*H)*I &= (\mathrm{R}^{W}(\mathrm{R}^{U}(G)\cdot \mathrm{L}^{T}(H)))*I\\
&=\sum \proba[m]\,\uni_{m,T+U}[Y]\,\uni_{m,V+W}[Z]\,\left(\left(\mathrm{R}^{Y}_{T}(G)\cdot\Psi_{(U,T;V,W)}^{(Y;Z)}(H)\right)*I\right)\\
&=\sum \proba[m]\,\uni_{m,T+U}[Y]\,\uni_{m,V+W}[Z]\,\left(\mathrm{R}^{Y}_{T}(G)\cdot\Psi_{(U,T;V,W)}^{(Y;Z)}(H)\cdot\mathrm{L}^{Z}_{W}(I)\right).
\end{align*}
The symmetry of this form ensures that one obtains the same with $G*(H*I)$. Thus, we have finally shown that $\mathscr{A}(n,\For_{q})$ is an algebra.
\end{proof}\medskip

\begin{remark}\label{naive}
The whole proof of the associativity of the product $*$ justifies Definition \ref{defpair} of partial isomorphisms as pairs of isomorphisms between two arbitrary subspaces of same dimension. One could have tried the following simpler construction of partial isomorphisms: a partial isomorphism is one linear automorphism $g$ of a subspace $V \subset (\For_{q})^{n}$, and the product of these simpler partial isomorphisms is defined by taking means of trivial extensions, a trivial extension of $(g,V)$ to a space $V^{+}$ containing $V$ being an automorphism $g^{+}$ of $V^{+}$ such $(g^{+})_{|V}=g$ and $t(g^{+})=t(g)\sqcup (X-1 : 1^{k^{+}-k})$. This seems easier and more natural, but it does not work: the underlying ``algebra'' is not associative as soon as $n\geq 2$. Indeed, given two ``naive'' partial isomorphisms $G,H$ defined on the same space $L$ of dimension $1$, and $I$ the partial isomorphism which is the identity on a $2$-dimensional space $P$ containing $L$, it is easy to see that for $G$ and $H$ non trivial, $G*(H*I)$ is not the same as $(G*H)*I$: the last product contains only diagonalizable partial isomorphisms of $P$, and it is not the case of the first product.
\end{remark}
\medskip\bigskip

\section{Constructions around $\mathscr{A}(n,\For_{q})$ and $\mathscr{Z}(n,\For_{q})$}\label{constructions}
In this section, we analyze the relations between $\C[\GL(n,\For_{q}) \times (\GL(n,\For_{q}))\opp]$ and $\mathscr{A}(n,\For_{q})$; between $\mathrm{Z}(n,\For_{q})$ and a subalgebra of invariants $\mathscr{Z}(n,\For_{q})\subset \mathscr{A}(n,\For_{q})$; and between the algebras $\mathscr{Z}(n,\For_{q})$ for various values of $n$. Our main goal is to prove Theorem \ref{generalFH}, which generalizes an old result of Farahat and Higman for products of conjugacy classes in the symmetric group algebras (\cite[Theorem 2.2]{FH59}), see also \cite[Proposition 7.3]{IK99}, \cite[Theorem 1]{Mel10} and \cite[Theorem 2.1]{Tout12}.
\medskip

\subsection{The map from $\mathscr{A}(n,\For_{q})$ to $\C[\GL(n,\For_{q}) \times (\GL(n,\For_{q}))\opp]$}
Consider the extension operator $$\pi_{n}=\mathrm{L}^{(\For_{q})^{n}}=\mathrm{R}^{(\For_{q})^{n}};$$ 
it is idempotent, and it sends $\mathscr{A}(n,\For_{q})$ to a subalgebra of it which is isomorphic to $\C[\GL(n,\For_{q}) \times \GL(n,\For_{q})\opp]$. Indeed, a partial isomorphism with underlying spaces $(\For_{q})^{n}$ and $(\For_{q})^{n}$ is just a pair of isomorphisms $(g_{1},g_{2})$, the product being 
$$(g_{1},g_{2})(h_{1},h_{2})=(g_{1}h_{1},h_{2}g_{2}),$$
that is to say the product of $\GL(n,\For_{q}) \times (\GL(n,\For_{q}))\opp$.
\begin{proposition}\label{projection}
The map $\pi_{n} :\mathscr{A}(n,\For_{q}) \to \C[\GL(n,\For_{q}) \times (\GL(n,\For_{q}))\opp]$ is a morphism of algebras. Moreover, it is compatible with the action of $\GL(n,\For_{q})$ on the left and on the right.
\end{proposition}
\begin{proof}
Take two partial isomorphisms $G=\partiso{S}{T}{g_{1}}{g_{2}}$ and $H=\partiso{U}{V}{h_{1}}{h_{2}}$. By Proposition \ref{RRL}, one has:
\begin{align*}
\pi_{n}(G * H)&=\mathrm{R}^{(\For_{q})^{n}}(\mathrm{R}^{U}(G)\cdot\mathrm{L}^{T}(H))\\
&=\sum \proba[m]\,\uni_{m,T+U}[Y]\,\uni_{m,V+(\For_{q})^{n}}[Z]\,\left(\mathrm{R}^{Y}_{T}(G)\cdot \Psi_{(U,T;V,(\For_{q})^{n})}^{(Y;Z)}(H)\right)\\
&= \mathrm{R}_{T}^{(\For_{q})^{n}}(G) \cdot \Psi_{(U,T;V,(\For_{q})^{n})}^{((\For_{q})^{n};(\For_{q})^{n})}(H) = \pi_{n}(G)\cdot \pi_{n}(H).
\end{align*}
Consider the inclusion
\begin{align*}
\GL(n,\For_{q}) &\hookrightarrow \GL(n,\For_{q}) \times (\GL(n,\For_{q}))\opp\\
g &\mapsto (g,g^{-1}).
\end{align*}
The compatibility of $\pi_{n}=\mathrm{R}^{(\For_{q})^{n}}$ with the action of $\GL(n,\For_{q})$ on the left follows from the general relations
\begin{align} &g \cdot (\mathrm{R}_{V}^{V^{+}}\partiso{U}{V}{g_{1}}{g_{2}})=\mathrm{R}_{V}^{V^{+}}(g\cdot \partiso{U}{V}{g_{1}}{g_{2}});\label{invariance1}\\
 &(\mathrm{L}_{U}^{U^{+}}\partiso{U}{V}{g_{1}}{g_{2}})\cdot k=\mathrm{L}_{U}^{U^{+}}(\partiso{U}{V}{g_{1}}{g_{2}}\cdot k).\label{invariance2}
\end{align}
They are themselves easily seen for instance from characterization $(2)$ of trivial extensions in Definition \ref{partisodef}. Since $\pi_{n}=\mathrm{L}^{(\For_{q})^{n}}$, one obtains for the same reasons the compatibility of $\pi_{n}$ with the action of $(\GL(n,\For_{q}))\opp$ on the right.
\end{proof}\bigskip

At this point one might want to construct an algebra $\mathscr{A}(\infty,\For_{q})$ which would be a projective limit of the $\mathscr{A}(n,\For_{q})$'s, in order to do generic computations directly for all the algebras $\C[\GL(n,\For_{q})\times (\GL(n,\For_{q}))\opp]$. However, the product of two partial isomorphisms $G=\partiso{S}{T}{g_{1}}{g_{2}}$ and $H=\partiso{U}{V}{h_{1}}{h_{2}}$ in an algebra $\mathscr{A}(n',\For_{q})$ involves arbitrary subspaces $S^{+}\supset S$ and $V^{+}\supset V$ inside $(\For_{q})^{n'}$. So, even if $U,V,S,T$ are included into $(\For_{q})^{n}$ with $n< n'$, there is no natural way to view the terms of $G*H$ as elements of $\mathscr{A}(n,\For_{q})$, even after applying a morphism to ``bring back'' $S^{+}$ and $V^{+}$ into $(\For_{q})^{n}$. In other terms, 
\begin{align*}
p_{n' \to n} : \mathscr{A}(n',\For_{q}) & \to\mathscr{A}(n,\For_{q})\\
\partiso{U}{V}{g_{1}}{g_{2}} &\mapsto \begin{cases} \partiso{U}{V}{g_{1}}{g_{2}} & \text{if } U,V \subset (\For_{q})^{n},\\
0&\text{otherwise}
\end{cases}
\end{align*}
is not a morphism of algebras, in opposition to what happens for partial permutations (\emph{cf.} \cite{IK99}). This difficulty seems to be a common feature for products defined by taking means, see \cite[Proposition 3.11]{Tout12}. It disappears by looking at subalgebras of invariants: indeed, one  knows then the proportion of configurations that stay in $(\For_{q})^{n}$, and this will enable us to construct projections $\mathscr{Z}(n',\For_{q})  \to\mathscr{Z}(n,\For_{q})$ that are now morphisms of (commutative) algebras.
\medskip

\subsection{The graded algebras of invariants $\mathscr{Z}(n,\For_{q})$}
We define $\mathscr{Z}(n,\For_{q})$ as the subspace of $\mathscr{A}(n,\For_{q})$ that consists in elements invariants under the actions of $\GL(n,\For_{q})$ on the left and of $(\GL(n,\For_{q}))\opp$ on the right. Equations \eqref{invariance1} and \eqref{invariance2} show that for general elements of $\mathscr{I}(n,\For_{q})$, \begin{align*}g\cdot(x*y)&=(g \cdot x)*y,\\ (x*y)\cdot h &= x*(y\cdot h);\end{align*} therefore, $\mathscr{Z}(n,\For_{q})$ is a subalgebra of $\mathscr{A}(n,\For_{q})$. Then, Proposition \ref{projection} ensures that $\pi_{n}(\mathscr{Z}(n,\For_{q}))$ is included into the center of the group algebra $\mathrm{Z}(n,\For_{q})$, and it is in fact equal to this center. Indeed, the injective map
\begin{align*}i_{n}: \C[\GL(n,\For_{q})\times(\GL(n,\For_{q}))\opp]) &\to \mathscr{A}(n,\For_{q})\\
(g,h) &\mapsto \partiso{(\For_{q})^{n}}{(\For_{q})^{n}}{g}{h}\end{align*}
satisfies $\pi_{n} \circ i_{n}=\id_{\C[G \times G\opp]}$, and it sends $\mathrm{Z}(n,\For_{q})$ into $\mathscr{Z}(n,\For_{q})$. So, for every $n$, we obtain an algebra that sends onto $\mathrm{Z}(n,\For_{q})$ in a natural way. In this paragraph, we detail the properties of $\mathscr{Z}(n,\For_{q})$.\bigskip

To start with, we endow $\mathscr{A}(n,\For_{q})$ with the grading $\deg \partiso{V}{W}{g_{1}}{g_{2}}=\dim V=\dim W.$ This is a grading of algebras: indeed, for two partial isomorphisms $\partiso{U}{V}{g_{1}}{g_{2}}$ and $\partiso{W}{X}{h_{1}}{h_{2}}$,
\begin{align*}\deg\left(\partiso{U}{V}{g_{1}}{g_{2}}*\partiso{W}{X}{h_{1}}{h_{2}}\right)&= \deg\left(\left( \mathrm{R}_{V}^{V+W}\partiso{U}{V}{g_{1}}{g_{2}} \right)\cdot \left(\mathrm{L}_{W}^{V+W}\partiso{W}{X}{h_{1}}{h_{2}}\right)\right) \\
&= \dim (V+W)\\
& \leq \dim V+\dim W \\
&\leq  \deg \partiso{U}{V}{g_{1}}{g_{2}}+\deg \partiso{W}{X}{h_{1}}{h_{2}}.
\end{align*}
For a polypartition $\bbmu$ of size $k \leq n$, recall that $\bbmu\!\! \uparrow^{n}$ is the polypartition of size $n$ obtained by adding parts $1$ to the partition $\mu(X-1)$. We also set $k=|\bbmu|$, $k_1=\ell(\mu(X-1))$ and $k_{11}=m_1(\mu(X-1))$. These quantities, and Pochhammer symbols of them, appear in the quotients of cardinalities $$\frac{\card\,C_{\bbmu \uparrow^n}}{\card\,C_\bbmu},$$ as follows from the formula given on page \pageref{cardinalconjugacy}.
\begin{proposition}
The space of invariants $\mathscr{Z}(n,\For_{q})$ is a graded algebra, with homogeneous basis given by
$$A_{\bbmu,n}=\sum_{\substack{\Ec=(e_{1},\ldots,e_{k})\\ \Fc=(f_{1},\ldots,f_{k})}} \partiso{\Vect(\Ec)}{\Vect(\Fc)}{I_{k}}{J(\bbmu)},$$
where $\bbmu$ runs over the set $\bigsqcup_{k=0}^{n} \ym(k,\For_{q})$ of polypartitions of size $k \in \lle 0,n\rre$; the sum is over free families $\Ec$ and $\Fc$ of size $k=|\bbmu|$; and the matrices $I_{k}$ and $J(\bbmu)$ determine isomorphisms with respect to the chosen bases $\Ec$ and $\Fc$. One has then $\deg A_{\bbmu,n}=|\bbmu|=k$, and 
$$\pi_{n}(A_{\bbmu,n})=q^{2(n-k)k_{1}}\,\frac{q^{2k^{2}}\,(q^{-1})_{k}\,(q^{-1})_{n-k+k_{11}}\,(q^{-1})_{n}}{(q^{-1})_{k_{11}}\,((q^{-1})_{n-k})^{2}\,(\card \,C_{\bbmu})} \,\,C_{\bbmu \uparrow^{n}}',$$
where as in the introduction $C_{\bbnu}'=\frac{1}{\card \,\GL(n,\For_{q})}\sum_{t(g_{1}g_{2})=\bbnu}\,(g_{1},g_{2})$ for a polypartition $\bbnu$ of size $n$.
\end{proposition}
\begin{proof}
The grading on $\mathscr{A}(n,\For_{q})$ restricts to $\mathscr{Z}(n,\For_{q})$, with obviously $\deg A_{\bbmu,n}=|\bbmu|=k$. Let us check that $(A_{\bbmu,n})_{|\bbmu| \leq n}$ is indeed a basis of $\mathscr{Z}(n,\For_{q})$. By a previous remark, the orbit of a partial isomorphism $\partiso{V}{W}{g_{1}}{g_{2}}$ is uniquely determined by the conjugacy type of $g_{1}g_{2} \in \GL(V)$, and an invariant element of $\mathscr{Z}(n,\For_{q})$ writes uniquely as a linear combination of such orbits. Since $A_{\bbmu,n}$ consists obviously in elements of type $\bbmu$, it suffices to verify that it is an invariant element; it will then be a multiple of the orbit of type $\bbmu$. Let $g \in \GL(n,\For_{q})$; one has
\begin{align*}
g \cdot A_{\bbmu,n} &= \sum_{\substack{\Ec=(e_{1},\ldots,e_{k})\\ \Fc=(f_{1},\ldots,f_{k})}} g\cdot \partiso{\Vect(\Ec)}{\Vect(\Fc)}{I_{k}}{J(\bbmu)} \\
&=\sum_{\substack{\Ec'=(g^{-1}(e_{1}),\ldots,g^{-1}(e_{k}))\\ \Fc=(f_{1},\ldots,f_{k})}}  \partiso{\Vect(\Ec')}{\Vect(\Fc)}{I_{k}}{J(\bbmu)}=A_{\bbmu,n}
\end{align*}
since $g^{-1}$ yields a permutation of all free families of size $k$ in $(\For_{q})^{n}$. This shows the invariance on the left, and the invariance on the right is proven similarly.\bigskip

Thus, $(A_{\bbmu,n})_{|\bbmu| \leq n}$ is a basis of $\mathscr{Z}(n,\For_{q})$, and $\pi_{n}(A_{\bbmu,n})$ is an element of $\mathrm{Z}(n,\For_{q})$ that consists in elements all of type $\bbmu\!\!\uparrow^{n}$ (they are trivial extensions of elements of type $\bbmu$). So, $\pi_{n}(A_{\bbmu,n})=\alpha_{\bbmu}\,C_{\bbmu\uparrow^{n}}'$, and the coefficient $\alpha_{\bbmu}$ is given by the quotient of the cardinalities of $A_{\bbmu,n}$ and of $C_{\bbmu\uparrow^{{n}}}$, the conjugacy class of elements of type $\bbmu\!\!\uparrow^{n}$ in $\GL(n,\For_{q})$. Denote $k_{1}=\ell(\mu(X-1))$, and $k_{11}=m_1(\mu(X-1))$.
\begin{align*}
\card\,A_{\bbmu,n}&=\left(q^{nk}\,\frac{(q^{-1})_{n}}{(q^{-1})_{n-k}}\right)^{2}\\
\card \,C_{\bbmu\uparrow^n}&=\left(\frac{q^{2(n-k)(k-k_{1})}\,(q^{-1})_{n}\,(q^{-1})_{k_{11}}}{(q^{-1})_{k}\,(q^{-1})_{n-k+k_{11}}}\right)\card \,C_{\bbmu}\\
\alpha_{\bbmu}&=q^{2(n-k)k_{1}}\,\frac{q^{2k^{2}}\,(q^{-1})_{k}\,(q^{-1})_{n-k+k_{11}}\,(q^{-1})_{n}}{(q^{-1})_{k_{11}}\,((q^{-1})_{n-k})^{2}\,(\card \,C_{\bbmu})}.
\end{align*}\end{proof}\bigskip

In the following, we shall deal with various scalar multiples of the (completed) conjugacy classes $C_{\bbmu\uparrow^n}$ and of the classes of partial isomorphisms $A_{\bbmu,n}$. Thus, we denote
\def\arraystretch{1.3}
$$\begin{tabular}{|l|l|}
\hline classes & number of elements\\ 
\hline & \vspace{-4mm}\\
 $\displaystyle C_{\bbmu\uparrow^n} = \sum_{\substack{g \in \GL(n,\For_q)\\t(g)=\bbmu \uparrow^n}}g$ & $\left(\frac{q^{2(n-k)(k-k_{1})}\,(q^{-1})_{n}\,(q^{-1})_{k_{11}}}{(q^{-1})_{k}\,(q^{-1})_{n-k+k_{11}}}\right) \card\, C_{\bbmu}$ \\
\hline & \vspace{-4mm}\\
$\displaystyle \widetilde{C}_{\bbmu\uparrow^n}=\frac{C_{\bbmu\uparrow^n}}{\card \,C_{\bbmu\uparrow^n}}$ & $1$ \\
\hline\hline & \vspace{-4mm}\\
 $\displaystyle A_{\bbmu,n} = \sum_{\substack{\Ec=(e_{1},\ldots,e_{k})\\ \Fc=(f_{1},\ldots,f_{k})}} \partiso{\Vect(\Ec)}{\Vect(\Fc)}{I_{k}}{J(\bbmu)} $ & $\left(q^{nk}\,\frac{(q^{-1})_{n}}{(q^{-1})_{n-k}}\right)^{2}$ \\
\hline & \vspace{-4mm}\\
$\displaystyle \widehat{A}_{\bbmu,n}=\frac{A_{\bbmu,n}}{\sqrt{\card\, A_{\bbmu,n}}}$ & $q^{nk}\,\frac{(q^{-1})_{n}}{(q^{-1})_{n-k}}$ \\
\hline & \vspace{-4mm }\\
$\displaystyle \widetilde{A}_{\bbmu\uparrow n}=\frac{A_{\bbmu,n}}{\card\, A_{\bbmu,n}}$ & $1$ \\
\hline
\end{tabular}$$
where by number of elements we mean the image by the linear map $g \in \C\GL(n,\For_q) \mapsto 1 \in \C$, or by the linear map
$\partiso{V}{W}{g_1}{g_2} \in \mathscr{A}(n,\For_q) \mapsto 1 \in \C$. Most manipulations are eased by dealing with the normalized versions of classes (those with a $\sim$).

\begin{proposition}\label{ziscommutative}
The algebra $\mathscr{Z}(n,\For_{q})$ is commutative.
\end{proposition}
\begin{proof}
Ultimately the result comes from the commutativity of all the algebras $\mathrm{Z}(\C\GL(V))$ with $V$ subspace of $(\For_{q})^{n}$, but the verifications are not obvious at all. On the other hand, we really need to do them in order to obtain a formula for the product of two classes $A_{\bbmu,n}$ and $A_{\bbnu,n}$; so we cannot skip this proof of commutativity. Take two random independent subspaces $V$ and $W$ of dimension $k$ and $l$ inside $(\For_{q})^{n}$. The law of $m=\dim(V+W)$ is the case $j=0$ of Lemma \ref{complicateproba}:
$$\proba[\dim (V+W)=m] = q^{(k+l-m)(m-n)}\,\frac{(q^{-1})_{n-k}\,(q^{-1})_{n-l}\,(q^{-1})_{k}\,(q^{-1})_{l}}{(q^{-1})_{k+l-m}\,(q^{-1})_{n-m}\,(q^{-1})_{n}\,(q^{-1})_{m-k}\,(q^{-1})_{m-l}}.$$
Then, since the law of the pair $(V,W)$ is invariant by action by $\GL(n,\For_{q})$, $Z=V+W$ is a random subspace with a law invariant by $\GL(n,\For_{q})$, so its law is a mixture of the uniform laws on random subspaces of fixed dimension $m$. Hence,
$$\proba[Z]=\proba[m]\,\uni_{m}[Z]\quad\text{with }m=\dim Z.$$
Corollary \ref{uniformproba} then tells us that knowing $Z$, $V$ is uniformly distributed among spaces of dimension $k$ included into $Z$; and knowing $Z$, $W$ is uniformly distributed among spaces of dimension $l$ included into $Z$. Notice however that the joint law of $(V,W)$ is not the product of these uniform laws; in other words, $V$ and $W$ are not independent conditionally to $Z$.  One can rewrite $A_{\bbmu,n}$ as a multiple of the normalized sum
$$
\widetilde{A}_{\bbmu,n}=\sum_{\Ec,\Fc\text{ free families}} \buni_{k}[\Ec]\,\buni_{k}[\Fc] \,\partiso{\Vect(\Ec)}{\Vect(\Fc)}{I_{k}}{J(\bbmu)} ,$$
where $\buni_{k}[\Ec]$ and $\buni_{k}[\Fc]$ are the uniform probabilities on free families of size $k=|\bbmu|$. Notice then that
\begin{align*}
&\mathrm{R}_{V}^{V^{+}}\!\left(\sum_{\substack{\Ec,\Fc\text{ free families}\\ \Vect(\Fc)=V}} \buni_{k}[\Ec]\,\buni_{V}[\Fc] \,\partiso{\Vect(\Ec)}{\Vect(\Fc)}{I_{k}}{J(\bbmu)} \right)\\
&=\sum_{\substack{\Ec^{+},\Fc^{+}\text{ free families}\\ \Vect(\Fc^{+})=V^{+},\,\,\Vect(\Fc)=V}} \buni_{k^{+}}[\Ec^{+}]\,\buni_{V^{+},V}[\Fc^{+}] \,\partiso{\Vect(\Ec^{+})}{\Vect(\Fc^{+})}{I_{k^{+}}}{J(\bbmu\uparrow k^{+})}
\end{align*}
if $V \subset V^{+}$ have dimension $k$ and $k^{+}$; here $\buni_{V}[\Fc]$ (resp., $\buni_{V^{+},V}[\Fc^{+}]$) is the uniform law on free families of $(\For_{q})^{n}$ such that $\Vect(\Fc)=V$ (resp., such that $\Vect(\Fc^{+})=V^{+}$ and $\Vect(\Fc)=V$, with $\Fc$ the family of the $k$ first vectors of $\Fc^{+}$).\bigskip

A product of two normalized sums $\widetilde{A}_{\bbmu,n}$ and $\widetilde{A}_{\bbnu,n}$ writes now as
\begin{align*}
&\sum_{\Ec,\Fc,\Gc,\Hc}\, \buni[\Ec,\Fc,\Gc,\Hc]\, \partiso{\Vect(\Ec)}{\Vect(\Fc)}{I_{k}}{J(\bbmu)} * \partiso{\Vect(\Gc)}{\Vect(\Hc)}{I_{l}}{J(\bbnu)}\\
&=\sum_{\substack{\dim V=k \\ \dim W =l \\ \Ec,\Fc,\Gc,\Hc\\ \Vect(\Fc)=V\\\Vect(\Gc)=W}}\binom{ \uni_{k}[V]\,\uni_{l}[W]\,\buni[\Ec,\Hc]\,\buni_{V}[\Fc]\,\buni_{W}[\Gc]\,\, \times \qquad \qquad \qquad\qquad}{\partiso{\Vect(\Ec)}{\Vect(\Fc)}{I_{k}}{J(\bbmu)} * \partiso{\Vect(\Gc)}{\Vect(\Hc)}{I_{l}}{J(\bbnu)} }\\
&=\sum_{\substack{Z,V,W \\ \Ec,\Fc,\Gc,\Hc\\ \Vect(\Fc)=V\\\Vect(\Gc)=W}}\binom{ \proba[m]\, \uni_{m}[Z]\,\proba[(V,W)|Z]\,\buni[\Ec,\Hc]\,\buni_{V}[\Fc]\,\buni_{W}[\Gc]\,\, \times \qquad\qquad \qquad\qquad}{\mathrm{R}_{V}^{V+W}\partiso{\Vect(\Ec)}{\Vect(\Fc)}{I_{k}}{J(\bbmu)} \cdot \mathrm{L}_{W}^{V+W}\partiso{\Vect(\Gc)}{\Vect(\Hc)}{I_{l}}{J(\bbnu)} }\\
&=\!\!\!\!\!\!\!\sum_{\substack{Z,V,W \\ \Ec^{+},\Fc^{+},\Gc^{+},\Hc^{+}\\ \Vect(\Fc^{+})=Z,\,\,\Vect(\Fc)=V \\ \Vect(\Gc^{+})=Z,\,\,\Vect(\Gc)=W}}\!\!\!\!\!\!\!\!\binom{ \proba[m]\, \uni_{m}[Z]\,\buni_{m}[\Ec^{+},\Hc^{+}]\,\proba[V|Z]\,\buni_{Z,V}[\Fc^{+}]\,\proba[W|(Z,V)]\,\buni_{Z,W}[\Gc^{+}]\,\, \times}{\partiso{\Vect(\Ec^{+})}{\Vect(\Fc^{+})}{I_{m}}{J(\bbmu \uparrow^{m})} \cdot \partiso{\Vect(\Gc^{+})}{\Vect(\Hc^{+})}{I_{m}}{J(\bbnu \uparrow^{m})} }\end{align*}
\begin{align*}
&=\!\!\!\!\!\!\!\sum_{\substack{Z,V,W \\ \Ec^{+},\Fc^{+},\Gc^{+},\Hc^{+}\\ \Vect(\Fc^{+})=Z,\,\,\Vect(\Fc)=V \\ \Vect(\Gc^{+})=Z,\,\,\Vect(\Gc)=W}}\!\!\!\!\!\!\!\!\binom{ \proba[m]\, \uni_{m}[Z]\,\buni_{m}[\Ec^{+},\Hc^{+}]\,\buni_{Z}[\Fc^{+}]\,\proba[W|(Z,V)]\,\buni_{Z,W}[\Gc^{+}]\,\, \times}{\partiso{\Vect(\Ec^{+})}{\Vect(\Hc^{+})}{P^{-1}}{J(\bbnu\uparrow^{m})PJ(\bbmu \uparrow^{m})} }\\
&=\!\!\!\!\!\!\!\sum_{\substack{Z,V,W \\ \Ec^{+},\Fc^{+},\Gc^{+},\Hc^{+}\\ \Vect(\Fc^{+})=Z,\,\,\Vect(\Fc)=V \\ \Vect(\Gc^{+})=Z,\,\,\Vect(\Gc)=W}}\!\!\!\!\!\!\!\!\binom{ \proba[m]\, \buni_{m}[Z]\,\buni_{m}[\Ec^{+},\Hc^{+}]\,\buni_{Z}[\Fc^{+}]\,\proba[W|(Z,V)]\,\buni_{Z,W}[\Gc^{+}]\,\, \times}{\partiso{\Vect(\Ec^{+})}{\Vect(\Hc^{+})}{I_{m}}{P^{-1}J(\bbnu\uparrow^{m})PJ(\bbmu \uparrow^{m})} }
\end{align*}
where $P$ stands for the matrix $\matr_{\Fc^{+}}(\Gc^{+})$. Notice that this matrix $P$ is the only part depending on $V,W,\Fc^{+}$ and $\Gc^{+}$ in the last average of partial isomorphisms. Therefore, the sum will be simplified if we are able to identify the law of $P$ under the probability $\buni_{Z}[\Fc^{+}]\,\proba[W|(Z,V)]\,\buni_{Z,W}[\Gc^{+}]$, which is a conditional probability depending on the random variable $Z$. We claim that this law is just the uniform law on matrices in $\GL(m,\For_{q})$. Remember that  $\Fc^{+}$ is chosen uniformly, that the law of $W$ knowing $Z$ and $V$ is the uniform law on spaces with dimension $l$ and with $V+W=Z$, and that $\Gc^{+}$ is then the uniform law on bases of $Z$ with $\Vect(\Gc)=W$. This implies that:\vspace{2mm}
\begin{enumerate} 
\item The probability of a given matrix $P$ is invariant by conjugation by elements of $\GL(m,\For_{q})$:
$$\proba[\matr_{\Fc^{+}}(\Gc^{+})=P]=\proba[\matr_{g(\Fc^{+})}(g(\Gc^{+}))=P]=\proba[\matr_{\Fc^{+}}(\Gc^{+})=gPg^{-1}].$$
\item For $V$ fixed, the probability conditionally to $Z$ and $V$ of $\Gc^{+}$ is invariant by action of the parabolic subgroup $\mathrm{T}(k,m-k,\For_{q})$ of block-triangular matrices 
$$\begin{pmatrix}
K_{1}&R \\ 0 & K_{2}
\end{pmatrix}\quad \text{with }K_{1} \in \GL(k,\For_{q}),\,\,K_{2} \in \GL(m-k,\For_{q}),\,\,R \in \mathrm{M}(k\times(m-k),\For_{q}),$$
since this is the stabilizer of $V$ inside $\GL(m,\For_{q})$.\vspace{2mm}
\end{enumerate}
If $k=0$ or $m$ the second point proves the claim. Otherwise, the group that leaves invariant the probability $\proba[P]$ is a subgroup of $\GL(m,\For_{q})$ that contains the maximal parabolic subgroup $\mathrm{T}(k,m-k,\For_{q})$, but also all its conjugates. Therefore it contains strictly $\mathrm{T}(k,m-k,\For_{q})$, so it can only be $\GL(m,\For_{q})$ itself, and the claim is shown in full generality.\bigskip

We can finally write:
\begin{align*}\widetilde{A}_{\bbmu,n}*\widetilde{A}_{\bbnu,n}&= \sum_{\substack{\Ec^{+},\Hc^{+}\\ P \in \GL(m,\For_{q})}} \frac{\proba[m]\,\buni_{m}[\Ec^{+},\Hc^{+}]}{\card \,\GL(m,\For_{q})}\,\partiso{\Vect(\Ec^{+})}{\Vect(\Hc^{+})}{I_{m}}{P^{-1}J(\bbnu\uparrow^{m})PJ(\bbmu \uparrow^{m})} \\
&=\sum_{\substack{\Ec^{+},\Hc^{+}\\ P \in \GL(m,\For_{q})}} \frac{\proba[m]\,\buni_{m}[\Ec^{+},\Hc^{+}]}{\card \,\GL(m,\For_{q})}\,\partiso{\Vect(\Ec^{+})}{\Vect(\Hc^{+})}{I_{m}}{J(\bbmu \uparrow^{m})P^{-1}J(\bbnu\uparrow^{m})P}\\
&=\sum_{\substack{\Fc^{+},\Hc^{+}\\ P \in \GL(m,\For_{q})}} \frac{\proba[m]\,\buni_{m}[\Fc^{+},\Hc^{+}]}{\card \,\GL(m,\For_{q})}\,\partiso{\Vect(\Fc^{+})}{\Vect(\Hc^{+})}{P}{J(\bbmu \uparrow^{m})P^{-1}J(\bbnu\uparrow^{m})}\\
&=\sum_{\substack{\Fc^{+},\Gc^{+}\\ P \in \GL(m,\For_{q})}} \frac{\proba[m]\,\buni_{m}[\Fc^{+},\Gc^{+}]}{\card \,\GL(m,\For_{q})}\,\partiso{\Vect(\Fc^{+})}{\Vect(\Gc^{+})}{I_{m}}{PJ(\bbmu \uparrow^{m})P^{-1}J(\bbnu\uparrow^{m})}\\
&=\sum_{\substack{\Fc^{+},\Gc^{+}\\ P \in \GL(m,\For_{q})}} \frac{\proba[m]\,\buni_{m}[\Fc^{+},\Gc^{+}]}{\card \,\GL(m,\For_{q})}\,\partiso{\Vect(\Fc^{+})}{\Vect(\Gc^{+})}{I_{m}}{P^{-1}J(\bbmu \uparrow^{m})PJ(\bbnu\uparrow^{m})} \\
&= \widetilde{A}_{\bbnu,n}*\widetilde{A}_{\bbmu,n}
\end{align*}
by using on the second line the fact that $\sum_{P} P^{-1}J(\bbnu\!\!\uparrow ^{m})P$ commutes with everyone in $\C\GL(m,\For_{q})$, because it is a multiple of a conjugacy class. 
\end{proof}\bigskip

\begin{remark}
Starting from the last computation in the proof of commutativity, one can give a useful expression for the product of two normalized classes $\widetilde{A}_{\bbmu,n}$ and $\widetilde{A}_{\bbnu,n}$: this is
$$\widetilde{A}_{\bbmu,n}*\widetilde{A}_{\bbnu,n}=\sum_{\substack{\Ec^{+},\Hc^{+}\\t(G)=\bbmu\uparrow^{m},\,\,t(H)=\bbnu\uparrow^{m}}} \frac{\proba[m]\,\buni_{m}[\Ec^{+},\Hc^{+}]}{(\card \,C_{\bbmu\uparrow^{m}})\,(\card \,C_{\bbnu\uparrow^{m}})}\,\partiso{\Vect(\Ec^{+})}{\Vect(\Hc^{+})}{I_{m}}{HG}.$$
With this expansion, it becomes obvious that $\widetilde{A}_{\bbmu,n}*\widetilde{A}_{\bbnu,n}=\widetilde{A}_{\bbnu,n}*\widetilde{A}_{\bbmu,n}$, because in any group algebra $\C\GL(m,\For_{q})$, $\widetilde{C}_{\bbmu \uparrow^{m}}*\widetilde{C}_{\bbnu \uparrow^{m}}=\widetilde{C}_{\bbnu \uparrow^{m}}*\widetilde{C}_{\bbmu \uparrow^{m}}$.
\end{remark}
\bigskip

\subsection{The projective limit $\mathscr{Z}(\infty,\For_{q})$} We focus now on the construction of a projective limit of the subalgebras of invariants (in the category of graded algebras). The main tool is the following:
\begin{proposition}
Consider the linear maps
\begin{align*}
\phi^{n+p}_{n} : \mathscr{A}(n+p,\For_{q}) & \to \mathscr{A}(n,\For_{q})\\
\partiso{V}{W}{g_{1}}{g_{2}} &\mapsto \begin{cases}
q^{pk}\, \frac{(q^{-1})_{n+p}\,(q^{-1})_{n-k}}{(q^{-1})_{n+p-k}\,(q^{-1})_{n}}\,\partiso{V}{W}{g_{1}}{g_{2}}&\text{if }V+W \subset (\For_{q})^{n}, \\
0 & \text{otherwise},
\end{cases}
\end{align*}
$k$ denoting the degree of the partial isomorphism. Their restrictions $\phi^{n+p}_{n} : \mathscr{Z}(n+p,\For_{q})  \to \mathscr{Z}(n,\For_{q})$ are morphisms of commutative algebras and they satisfy the relations 
\begin{equation}
\phi_{n}^{n+p}\left(\widehat{A}_{\bbmu,n+p}\right)=\widehat{A}_{\bbmu,n},\label{magicbasis}
\end{equation}
where 
$$\widehat{A}_{\bbmu,n}=\frac{A_{\bbmu,n}}{\sqrt{\card\, A_{\bbmu,n}}}=\frac{(q^{-1})_{n-k}}{q^{nk}\,(q^{-1})_{n}}\,A_{\bbmu,n}.$$
\end{proposition}
\begin{proof}
To avoid any ambiguity, we shall precise by indices the notations of the proof of Proposition \ref{ziscommutative}. Thus, we shall denote $\proba_{k,l,n}[m]$ the probability of $m$ in the previous probabilistic scheme and in $(\For_{q})^{n}$, and $\buni_{k,n}[\Ec]$ the uniform probability over free family of size $k$ in $(\For_{q})^{n}$. Fix two polypartitions $\bbmu$ and $\bbnu$ of sizes $k$ and $l$. If $k$ or $l$ is strictly bigger than $n$, then $\phi_{n}^{n+p}(\widehat{A}_{\bbmu,n+p})$ or $\phi_{n}^{n+p}(\widehat{A}_{\bbnu,n+p})$ is equal to zero, and on the other hand, every partial isomorphism appearing in $\widehat{A}_{\bbmu,n+p}*\widehat{A}_{\bbnu,n+p}$ will have degree bigger than $\max(k,l)>n$, so one has
$$\phi_{n}^{n+p}\left(\widehat{A}_{\bbmu,n+p}*\widehat{A}_{\bbnu,n+p}\right)=0 =\phi_{n}^{n+p}\left(\widehat{A}_{\bbmu,n+p}\right)*\phi_{n}^{n+p}\left(\widehat{A}_{\bbnu,n+p}\right).$$
Suppose now $k\leq n$ and $l \leq n$. We compute $\phi_{n}^{n+p}(\widehat{A}_{\bbmu,n+p}*\widehat{A}_{\bbnu,n+p})$ as follows:
$$\phi_{n}^{n+p}\left(\widehat{A}_{\bbmu,n+p}*\widehat{A}_{\bbnu,n+p}\right)=\frac{q^{(n+p)(k+l)}\,((q^{-1})_{n+p})^{2}}{(q^{-1})_{n+p-k}\,(q^{-1})_{n+p-l}}\,\,\phi_{n}^{n+p}\left(\widetilde{A}_{\bbmu,n+p}*\widetilde{A}_{\bbnu,n+p}\right)$$
and
\begin{align*}
&\phi_{n}^{n+p}\left(\widetilde{A}_{\bbmu,n+p}*\widetilde{A}_{\bbnu,n+p}\right)\\
&=\sum_{\substack{\Ec^{+},\Hc^{+}\\t(G)=\bbmu\uparrow^{m}\\t(H)=\bbnu\uparrow^{m}}} \!\!\frac{\proba_{k,l,n+p}[m]\,\buni_{m,n+p}[\Ec^{+},\Hc^{+}]}{\card (C_{\bbmu\uparrow^{m}}\times C_{\bbnu\uparrow^{m}})}\,\phi_{n}^{n+p}\partiso{\Vect(\Ec^{+})}{\Vect(\Hc^{+})}{I_{m}}{HG},
\end{align*}
where the families $\Ec^{+}$ and $\Hc^{+}$ run \emph{a priori} over free families of size $m$ in $(\For_{q})^{n+p}$, but in fact over free families of size $m$ in $(\For_{q})^{n}$, since otherwise $\phi_{n}^{n+p}\partiso{\Vect(\Ec^{+})}{\Vect(\Hc^{+})}{I_{m}}{HG}$ vanishes. The quantity $\proba_{k,l,n+p}[m]\,\buni_{m,n+p}[\Ec^{+},\Hc^{+}]$ is equal to 
\begin{align*}&\frac{ q^{(k+l-m)(m-n-p)}\,(q^{-1})_{n+p-k}\,(q^{-1})_{n+p-l}\,(q^{-1})_{k}\,(q^{-1})_{l}}{(q^{-1})_{k+l-m}\,(q^{-1})_{n+p-m}\,(q^{-1})_{n+p}\,(q^{-1})_{m-k}\,(q^{-1})_{m-l}}\,\left(\frac{(q^{-1})_{n+p-m}}{q^{m(n+p)}\,(q^{-1})_{n+p}}\right)^{2}\\
&=\frac{ q^{(k+l-m)(m-n-p)}\,(q^{-1})_{n+p-k}\,(q^{-1})_{n+p-l}\,(q^{-1})_{n+p-m}\,(q^{-1})_{k}\,(q^{-1})_{l}}{((q^{-1})_{n+p})^{3}\,(q^{-1})_{k+l-m}\,(q^{-1})_{m-k}\,(q^{-1})_{m-l}}\,\left(\frac{1}{q^{m(n+p)}}\right)^{2},\end{align*}
which multiplied by $\sqrt{(\card\,A_{\bbmu,n+p})\times (\card \,A_{\bbnu,n+p})}$ gives
$$\frac{ q^{(k+l-m-n-p)m}\,(q^{-1})_{n+p-m}\,(q^{-1})_{k}\,(q^{-1})_{l}}{(q^{-1})_{n+p}\,(q^{-1})_{k+l-m}\,(q^{-1})_{m-k}\,(q^{-1})_{m-l}}.$$
Multiplying again by the coefficient in the statement of the proposition, we get:
\begin{align}&\phi_{n}^{n+p}\left(\widehat{A}_{\bbmu,n+p}*\widehat{A}_{\bbnu,n+p}\right)\nonumber\\
&=\sum_{\substack{\Ec^{+},\Hc^{+}\\t(G)=\bbmu\uparrow^{m}\\t(H)=\bbnu\uparrow^{m}}} \,\frac{ q^{(k+l-m-n)m}\,(q^{-1})_{k}\,(q^{-1})_{l}\,(q^{-1})_{n-m}}{(q^{-1})_{k+l-m}\,(q^{-1})_{m-k}\,(q^{-1})_{m-l}\,(q^{-1})_{n}}\, \frac{\partiso{\Vect(\Ec^{+})}{\Vect(\Hc^{+})}{I_{m}}{HG}}{\card (C_{\bbmu\uparrow^{m}}\times C_{\bbnu\uparrow^{m}})}.\label{systemaps}\end{align}
On the other hand,
\begin{align*}
\phi_{n}^{n+p}\left(\widehat{A}_{\bbmu,n+p}\right)&=\frac{q^{k(n+p)}\,(q^{-1})_{n+p}}{(q^{-1})_{n+p-k}}\,\,\phi_{n}^{n+p}\left(\widetilde{A}_{\bbmu,n+p}\right)\\
&=\frac{q^{k(n+p)}\,(q^{-1})_{n+p}}{(q^{-1})_{n+p-k}}\,\sum_{\substack{\Ec,\Fc\\t(G)=\bbmu}}\frac{\buni_{k,n+p}[\Ec,\Fc]}{\card\, C_{\bbmu}} \,\,\phi_{n}^{n+p}\partiso{\Vect(\Ec)}{\Vect(\Fc)}{I_{k}}{G}\\
&=\sum_{\substack{\Ec,\Fc\\t(G)=\bbmu}}\frac{q^{-nk}\,(q^{-1})_{n-k}}{(q^{-1})_{n}} \,\frac{\partiso{\Vect(\Ec)}{\Vect(\Fc)}{I_{k}}{G}}{\card\, C_{\bbmu}},
\end{align*}
where again the families $\Ec$ and $\Fc$ run \emph{a priori} over free families of size $k$ in $(\For_{q})^{n+p}$, but in fact over free families of size $k$ in $(\For_{q})^{n}$, since otherwise $\phi_{n}^{n+p}\partiso{\Vect(\Ec)}{\Vect(\Fc)}{I_{k}}{G}$ vanishes. Therefore, one can rewrite this as:
\begin{align*}\phi_{n}^{n+p}\left(\widehat{A}_{\bbmu,n+p}\right)&=\frac{q^{nk}\,(q^{-1})_{n}}{(q^{-1})_{n-k}}\,\sum_{\substack{\Ec,\Fc\\t(G)=\bbmu}} \frac{\buni_{k,n}[\Ec,\Fc]}{\card\, C_{\bbmu}}\,\,\partiso{\Vect(\Ec)}{\Vect(\Fc)}{I_{k}}{G}\\
&=\frac{q^{nk}\,(q^{-1})_{n}}{(q^{-1})_{n-k}}\,\widetilde{A}_{\bbmu,n}=\widehat{A}_{\bbmu,n}.\end{align*}
This proves in particular that $\phi_{n}^{n+p}(\mathscr{Z}(n+p,\For_{q}))=\mathscr{Z}(n,\For_{q})$. Finally,
$$\widetilde{A}_{\bbmu,n}*\widetilde{A}_{\bbnu,n}=\sum_{\substack{\Ec^{+},\Hc^{+}\\t(G)=\bbmu\uparrow^{m}\\
t(H)=\bbnu\uparrow^{m}}} \frac{\proba_{k,l,n}[m]\,\buni_{m,n}[\Ec^{+},\Hc^{+}]}{(\card \,C_{\bbmu\uparrow^{m}})\,(\card \,C_{\bbnu\uparrow^{m}})}\,\partiso{\Vect(\Ec^{+})}{\Vect(\Hc^{+})}{I_{m}}{HG};$$
the quantity $\proba_{k,l,n}[m]\,\buni_{m,n}[\Ec^{+},\Hc^{+}]$ is equal to 
$$\frac{ q^{(k+l-m)(m-n)}\,(q^{-1})_{n-k}\,(q^{-1})_{n-l}\,(q^{-1})_{n-m}\,(q^{-1})_{k}\,(q^{-1})_{l}}{((q^{-1})_{n})^{3}\,(q^{-1})_{k+l-m}\,(q^{-1})_{m-k}\,(q^{-1})_{m-l}}\,\left(\frac{1}{q^{mn}}\right)^{2},$$
which multiplied by $\sqrt{(\card\,A_{\bbmu,n})\times (\card \,A_{\bbnu,n})}$ gives
$$\frac{ q^{(k+l-m-n)m}\,(q^{-1})_{n-m}\,(q^{-1})_{k}\,(q^{-1})_{l}}{(q^{-1})_{n}\,(q^{-1})_{k+l-m}\,(q^{-1})_{m-k}\,(q^{-1})_{m-l}}.$$
This is exactly the coefficient of Equation \eqref{systemaps}, so the morphism property is shown --- notice that if $$\mathbb{1}_{n}=\partiso{\{0\}}{\{0\}}{\id}{\id}=A_{\emptyset}=\widehat{A}_{\emptyset}$$
is the unity of $\mathscr{A}(n,\For_{q})$, then one has indeed also $\phi_{n}^{n+p}(\mathbb{1}_{n+p})=\mathbb{1}_{n}$. The compatibility between the maps $\phi_{n}^{n+p}$ is obvious from the relations \eqref{magicbasis}, so one gets as announced an inverse system of graded commutative algebras
$$\begin{CD} \cdots @>>>\mathscr{Z}(n+2,\For_{q})  @>>>\mathscr{Z}(n+1,\For_{q})  @>>>\mathscr{Z}(n,\For_{q}) @>>> \cdots \end{CD}\,\,.$$
\end{proof}
\bigskip

From there everything gets easy. With $|\bblambda|=k$ and $|\bbmu|=l$, denote $S_{\bblambda\bbmu}^{\bbnu}$ the structure coefficients given by the product
$$\widehat{A}_{\bblambda,k+l} * \widehat{A}_{\bbmu,k+l}=\sum_{|\bbnu|\leq k+l}S_{\bblambda\bbmu}^{\bbnu}\,\widehat{A}_{\bbnu,k+l}$$
in $\mathscr{Z}(k+l,\For_{q})$. If $|\bbnu|>n$, we convene that $\widehat{A}_{\bbnu,n}=0$. Since the maps $\phi_{n}^{k+l}$ are morphisms of algebras, for every $n \leq k+l$, one also has
$$\widehat{A}_{\bblambda,n} * \widehat{A}_{\bbmu,n}=\sum_{|\bbnu|\leq k+l}S_{\bblambda\bbmu}^{\bbnu}\,\widehat{A}_{\bbnu,n}.$$
 We claim that this identity still holds for $n >k+l$. Indeed, denote $S_{\bbmu\bbnu}^{\bblambda}(n)$ the structure coefficients in $\mathscr{Z}(n,\For_{q})$, such that
$$\widehat{A}_{\bblambda,n} * \widehat{A}_{\bbmu,n}=\sum_{|\bbnu|\leq n}S_{\bblambda\bbmu}^{\bbnu}(n)\,\widehat{A}_{\bbnu,n}.$$
The grading ensures that the sum is indeed over polypartitions of size less than $k+l$. By applying the map $\phi_{k+l}^{n}$ to the previous identity, we get
$$\widehat{A}_{\bblambda,k+l} * \widehat{A}_{\bbmu,k+l}=\sum_{|\bbnu|\leq k+l}S_{\bblambda\bbmu}^{\bbnu}(n)\,\widehat{A}_{\bbnu,k+l}.$$
Since $(A_{\bbnu})_{|\bbnu|\leq k+l}$ is a basis of $\mathscr{Z}(k+l,\For_{q})$, this shows that $S_{\bblambda\bbmu}^{\bbnu}(n)=S_{\bblambda\bbmu}^{\bbnu}(k+l)=S_{\bblambda\bbmu}^{\bbnu}$. As a consequence:
\begin{theorem} 
The graded commutative algebra $\mathscr{Z}(\infty,\For_{q})$ with linear basis $\widehat{A}_{\bbmu}$, $\bbmu \in \bigsqcup_{n=0}^{\infty}\ym(n,\For_{q})$ and with product
$$\widehat{A}_{\bblambda} * \widehat{A}_{\bbmu}=\sum_{|\bbnu|\leq |\bblambda| + |\bbmu|}S_{\bblambda\bbmu}^{\bbnu}\,\widehat{A}_{\bbnu}$$
is well-defined and forms a projective limit of the $\mathscr{Z}(n,\For_{q})$'s in the category of graded algebras. 
\end{theorem}\bigskip

The structure coefficients $S_{\bblambda\bbmu}^{\bbnu}$ are rational numbers: they are obviously in $\Q(q)$ since every product in $\mathscr{A}(k+l,\For_{q})$ involves averages with coefficients in $\Q(q)$, and assuming $q$ fixed, $\Q(q)$ is identified with $\Q$. Denote $\Pi_{n}=\pi_{n} \circ \phi_{n}^{\infty}$, where $\phi_{n}^{\infty} : \mathscr{Z}(\infty,\For_{q}) \to \mathscr{Z}(n,\For_{q})$ is the canonical projection, and $\pi_{n}$ is the map from Proposition \ref{projection}. This is a morphism of algebras from $\mathscr{Z}(\infty,\For_{q})$ to $Z(\C\GL(n,\For_{q}))$, with 
$$\Pi_{n}\left(\widehat{A}_{\bbmu}\right) = \begin{cases} q^{n(2k_{1}-k)} \,\frac{q^{2k(k-k_{1})}\,(q^{-1})_{k}\,(q^{-1})_{n-k+k_{11}}}{(q^{-1})_{k_{11}}\,(q^{-1})_{n-k}}\frac{C_{\bbmu\uparrow^{n}}}{\card\,C_{\bbmu}}  &\text{if }|\bbmu|=k \leq n\\ 0 &  \text{otherwise}.\end{cases}$$
Fix two polypartitions $\bblambda$ and $\bbmu$ of sizes $k$ and $l$, and with $k_{11}=m_1(\lambda(X-1))=0$ and $l_{11}=m_1(\mu(X-1))=0$; hence, they correspond to classes of partial isomorphisms that are not trivial extensions of smaller partial isomorphisms. This is the analogue of the restriction ``without fixed points'' for permutations in our statement of Farahat-Higman's theorem on page \pageref{fhig}. Notice however that one does not require $k_1=\ell(\lambda(X-1))$ or $l_1=\ell(\mu(X-1))$ to vanish; so, the isomorphisms considered may have non-zero fixed vectors, but no non-zero fixed vector with a stable complement subspace. Under the previous assumption, most of the previous formula simplifies and we get
\begin{align*}&q^{n(2k_{1}+2l_{1}-k-l)}\,q^{2k(k-k_{1})+2l(l-l_{1})}\,(q^{-1})_{k}\,(q^{-1})_{l}\,\frac{C_{\bblambda\uparrow^{n}}\times C_{\bbmu\uparrow^{n}}}{\card(C_{\bblambda}\times C_{\bbmu})} \\
&= \sum_{|\bbnu|=m\leq k+l} S_{\bblambda\bbmu}^{\bbnu}\,\, q^{n(2m_{1}-m)} \,\frac{q^{2m(m-m_{1})}\,(q^{-1})_{m}\,(q^{-1})_{n-m+m_{11}}}{(q^{-1})_{m_{11}}\,(q^{-1})_{n-m}}\frac{C_{\bbnu\uparrow^{n}}}{\card\,C_{\bbnu}}.\end{align*}
By putting everything independent from $n$ in modified structure coefficients 
$$s_{\bblambda\bbmu}^{\bbnu}= S_{\bblambda\bbmu}^{\bbnu}\,\frac{\card (C_{\bblambda}\times C_{\bbmu})}{\card\,C_{\bbnu}}\, \frac{q^{2m(m-m_{1})-2k(k-k_{1})-2l(l-l_{1})}\,(q^{-1})_{m}}{(q^{-1})_{k}\,(q^{-1})_{l}\,(q^{-1})_{m_{11}}},$$
we obtain finally:
$$C_{\bblambda\uparrow^{n}}\times C_{\bbmu\uparrow^{n}} = \sum_{|\bbnu|=m \leq k+l} s_{\bblambda\bbmu}^{\bbnu}\,q^{n((k-2k_{1})+(l-2l_{1})-(m-2m_{1}))} \, \frac{(q^{-1})_{n-m+m_{11}}}{(q^{-1})_{n-m}}\,C_{\bbnu\uparrow^{n}}.$$
This leads to a $\GL(n,\For_{q})$ version of a theorem of Farahat and Higman, which to our knowledge was not known before.
\begin{theorem}\label{generalFH}
Fix two polypartitions $\bblambda$ and $\bbmu$ of sizes $k$ and $l$, with $k_{11}=m_1(\lambda(X-1))=0$ and $l_{11}=m_1(\mu(X-1))=0$. There exists polynomials $p_{\bblambda\bbmu}^{\bbnu}(X)$ with rational coefficients such that for every $n \geq k+l$, in the center of the group algebra $\C\GL(n,\For_{q})$,
$$C_{\bblambda\uparrow^{n}}\times C_{\bbmu\uparrow^{n}} = \sum p_{\bblambda\bbmu}^{\bbnu}(q^{n})\,C_{\bbnu\uparrow^{n}},$$
where the sum runs over polypartitions $\bbnu$ of size $m \leq k+l$, and again with $m_{11}=m_1(\nu(X-1))=0$.
\end{theorem}
\begin{proof}
The $n$-dependent part of the coefficient of $C_{\bbnu\uparrow^{n}}$ in the previous expansion writes as
\begin{align*}&q^{n((k-2k_{1})+(l-2l_{1})-(m-2m_{1}))} \, \frac{(q^{-1})_{n-m+m_{11}}}{(q^{-1})_{n-m}} \\
&= q^{n((k-2k_{1})+(l-2l_{1})-(m-2m_{1}))}\,(1-q^{-(n-m+1)})(1-q^{-(n-m+2)})\cdots(1-q^{-(n-m+m_{11})}).
\end{align*}
This is \emph{a priori} a Laurent polynomial in $q^{n}$ with rational coefficients. Gathering together the polypartitions $\bbnu$ which give the same completed polypartition $\bbnu\!\!\uparrow^{n}$ (they differ by the number of parts $1$ in $\nu(X-1)$), we conclude from the previous discussion that
$$C_{\bblambda\uparrow^{n}}\times C_{\bbmu\uparrow^{n}} = \sum p_{\bblambda\bbmu}^{\bbnu}(q^{n})\,C_{\bbnu\uparrow^{n}},$$
where the sum is over the finite set of polypartitions $\bbnu$ with $|\bbnu|\leq |\bblambda|+|\bbmu|$ and $\nu(X-1)$ without parts of size $1$; and the coefficients $p_{\bblambda\bbmu}^{\bbnu}(q^{n})$ are rational Laurent polynomials in $\Q[X,X^{-1}]$. However, the group algebra of a finite group is defined over $\Z$, so for every integer $n$, $p_{\bblambda\bbmu}^{\bbnu}(q^{n})$ is an integer. With $q=p^{e}$, by looking at the $p$-valuation in $\Q$, one sees that $p_{\bblambda\bbmu}^{\bbnu}$ cannot have negative powers, so it is in fact a polynomial in $\Q[X]$.
\end{proof}
\bigskip

\section{Explicit computations with degree $1$ terms}\label{secdegree1}
In this section we make the previous discussion concrete by computing the polynomials $p_{\bblambda\bbmu}^{\bbnu}(q^{n})$ when $\bblambda$ and $\bbmu$ have degree $1$, hence correspond to irreducible polynomials $X-a$ and $X-b$ with $a,b \in (\For_{q})^{\times}$. The generic invariant $\widehat{A}_{\{X-a:1\}}$, that we shall abbreviate as $\widehat{A}_{X-a}$, has for projections in the algebras $\mathscr{Z}(n,\For_{q})$:
$$\widehat{A}_{X-a,n}=\frac{1}{q^{n}-1}\sum_{u,v \in (\For_{q})^{n}\setminus \{0\}} \partiso{u}{v}{1}{a}.$$
Here $\partiso{u}{v}{1}{a}$ means that one sends the vector $u$ to $v$ by the first arrow, and $v$ to $au$ by the second arrow. To compute the generic product $\widehat{A}_{X-a}*\widehat{A}_{X-b}$ in $\mathscr{Z}(\infty,\For_{q})$, it suffices to do so in $\mathscr{Z}(2,\For_{q})$ by the discussion of the previous section. Take two partial isomorphisms $\partiso{u}{v}{1}{a}$ and $\partiso{w}{x}{1}{b}$ in $\mathscr{I}(2,\For_{q})$. Among the $(q^{2}-1)^{4}$ possibilities for $u,v,w,x$, the vectors $v$ and $w$ are colinear in $(q^{2}-1)^{3}(q-1)$ cases, the factor $(q-1)$ corresponding to the possibilities for the factor of proportionnality $\alpha$ such that $w=\alpha v$. In all these cases, 
$$\partiso{u}{v}{1}{a} * \partiso{w}{x}{1}{b}=\partiso{u}{\alpha^{-1} x}{1}{ab}$$ 
has type $\{X-ab:1\}$. So, this situation contributes to a term 
\begin{equation}
(q^{2}-1)(q-1)\,\widetilde{A}_{X-ab}=(q-1)\,\widehat{A}_{X-ab}\label{colinear}
\end{equation} in $\widehat{A}_{X-a}*\widehat{A}_{X-b}$. In every other situation, $(\For_{q})^{2}=\Vect(v,w)$ and we have to compute trivial extensions. It should be noticed here that the form of $\pi_{2}(\partiso{u}{v}{1}{a})$ is different when $a=1$ and when $a \neq 1$. Therefore, we have several cases to treat separately, the most interesting being when $a \neq 1$ and $b \neq 1$. Suppose then $v$ and $w$ not colinear. One has
\begin{align*}\mathrm{R}_{\Vect(v)}^{\Vect(v,w)}\partiso{u}{v}{1}{a}&=\frac{1}{q^{2}(q-1)}\sum_{\substack{t \neq \alpha u \\ r \in \For_{q} }} \left( \Vect(u,t) \,\,\bigg|\,\, I_{2} \rightleftarrows \left(\begin{smallmatrix} a & (a-1)r \\ 0 & 1 \end{smallmatrix}\right) \,\,\bigg|\,\, \Vect(v,w)\right) \\
&=\frac{1}{q^{2}(q-1)}\sum_{\substack{t \neq \alpha u \\ r \in \For_{q} }} \left( \Vect(u,t) \,\,\bigg|\,\, I_{2} \rightleftarrows \left(\begin{smallmatrix} a & r \\ 0 & 1 \end{smallmatrix}\right) \,\,\bigg|\,\, \Vect(v,w)\right); \\ 
\mathrm{L}_{\Vect(w)}^{\Vect(v,w)}\partiso{w}{x}{1}{b}&= \frac{1}{q^{2}(q-1)}\sum_{\substack{y \neq \alpha x \\ s \in \For_{q} }} \left( \Vect(v,w) \,\,\bigg|\,\, I_{2} \rightleftarrows \left(\begin{smallmatrix} 1 & 0 \\ s & b \end{smallmatrix}\right) \,\,\bigg|\,\, \Vect(y,x)\right).
\end{align*}
So, outside the terms of \eqref{colinear}, the remaining part of the product $\widehat{A}_{X-a}*\widehat{A}_{X-b}$ is equal to
\begin{equation} \frac{1}{q^{3}(q-1)(q^{2}-1)}\sum_{\substack{u,t,x,y\\r,s}} \left( \Vect(u,t) \,\,\bigg| I_{2} \rightleftarrows \left(\begin{smallmatrix} a+rs & rb \\ s & b \end{smallmatrix}\right) \bigg|\,\, \Vect(y,x)\right) .\label{noncolinear} \end{equation}
\bigskip

\subsection{Irreducible polynomials of degree $2$ over $\For_{q}$}\label{irrpoly}
We then need to recognize irreducible polynomials of degree $2$ over $\For_{q}$. The theory is a bit different when $q$ is even and when $q$ is odd. Suppose to begin with that $q$ is odd. Then,
$$X^{2}+aX+b = \left(X+\frac{a}{2}\right)^{2}+\frac{4b-a^{2}}{4}$$
is irreducible if and only if $a^{2}-4b$ is not a square in $\For_{q}$. In $\For_{q}^{\times}$, there are $\frac{q-1}{2}$ squares identified by the Jacobi symbol $\jacobi{a}{q}=1$, and $\frac{q-1}{2}$ non-squares identified by the Jacobi symbol $\jacobi{a}{q}=-1$, $\jacobi{\cdot}{q}$ being a morphism from $(\For_{q})^{\times}$ to $\{\pm 1\}$. So: \vspace{2mm}
\begin{itemize}
\item $X^{2}+aX+b$ is irreducible if $\jacobi{a^{2}-4b}{q}=-1$;\vspace{2mm}
\item otherwise, it is equal to $(X-\frac{-a+\delta}{2})(X-\frac{-a-\delta}{2})$, where $\delta^{2}=a^{2}-4b$.\vspace{2mm}
\end{itemize}
This criterion has the following consequence. Suppose $b\neq 0$ fixed and consider the family of polynomials $P_{b}=\{X^{2}+aX+b,\,\,a\in\For_{q}\}$. Notice that 
\begin{align*}
\card \{a \in \For_{q}\,\,&|\,\, X^{2}+aX+b \text{ is irreducible}\}  = \card \left\{ a \in \For_{q} \,\,| \left(\frac{a^{2}-4b}{q}\right)=-1\right\}\\
&=q - \left\{ a \in \For_{q} \,\,|\,\, a^{2}-4b \text{ is a square}\right\} \\
&= q- \mathbb{1}_{\jacobi{b}{q}=1}  - \left\{ a \in \For_{q} \,\,|\,\, a^{2}-4b \text{ is a non-zero square}\right\} \\
&=q- \mathbb{1}_{\jacobi{b}{q}=1}  - \frac{q-1}{2}. 
\end{align*}
Therefore, $I_{b}=\{a \in \For_{q}\,\,|\,\,X^{2}+aX+b \text{ is irreducible}\}$ is of cardinality $\frac{q+1}{2}-\mathbb{1}_{\jacobi{b}{q}=1}$ for $b \neq 0$. \bigskip

Suppose now $q$ even; the solution to the quadratic equation is then a particular case of the Artin-Schreier theory. Recall that on a field of characteristic $2$, $x \mapsto x^{2}$ is a linear map. The trace of an element of $\For_{2^{n}}$ is the linear map
$$\mathrm{Tr}(x) = x+x^{2}+\cdots + x^{2^{n-1}} \in \For_{2}=\{0,1\}.$$
The image of the Artin-Schreier map $x \mapsto x^{2}+x$ is the $\For_{2}$-subspace of $\For_{2^{n}}$ of elements with trace zero. Then, given an arbitrary monic polynomial $X^{2}+aX+b$ of degree $2$:\vspace{2mm}
\begin{itemize}
\item if $a=0$, then $X^{2}+b=(X+b)^{2}$ is not irreducible.\vspace{2mm}
\item if $a\neq 0$ and $\mathrm{Tr}(ba^{-2})=0$, then $X^{2}+aX+b=(X+a(u+1))(X+au)$, where $u$ and $u+1$ are the two elements of $\For_{2^{n}}$ such that $u^{2}+u=ba^{-2}$.  \vspace{2mm}
\item otherwise, if $a \neq 0$ and $\mathrm{Tr}(ba^{-2})=1$, then $X^{2}+aX+b$ is irreducible.\vspace{2mm}
\end{itemize}
This implies that for every $b \neq 0$, $I_{b}=\{a \in \For_{q}\,\,|\,\,X^{2}+aX+b \text{ is irreducible}\}$ is of cardinality $\frac{q}{2}$.
\bigskip

\subsection{The case $a,b \neq 1$, $q$ odd}
Suppose $q$ odd and $a\neq 1$, $b \neq 1$. Going back to Equation \eqref{noncolinear}, the characteristic polynomial of the matrix $\left(\begin{smallmatrix} a+rs & rb \\ s & b \end{smallmatrix}\right)$ is $X^{2}-(a+rs+b)X+ab$. Notice that the multiplication map $m: (r,s) \in  \For_{q} \times \For_{q} \mapsto rs \in \For_{q}$
satisfies $\card\,m^{-1}(\{0\})=2q-1$ and $\card\, m^{-1}(\{c\})=q-1$ for $c \neq 0$. Therefore, with $u,t,x,y$ fixed:\vspace{2mm}\begin{enumerate}
\item One obtains $q$ times the polynomial $X^{2}-(a+b)X+ab$, corresponding for instance to the case $r=0$ and $s \in \For_{q}$ arbitrary. If $a \neq b$ this gives a contribution 
\begin{equation}q\,\widetilde{A}_{\{X-a: 1, X-b :1\}}=\frac{1}{(q-1)(q^{2}-1)}\,\widehat{A}_{\{X-a: 1, X-b :1\}};\label{qodd1}
\end{equation}
 otherwise if $a=b$ one gets 
 \begin{align}
 \widetilde{A}_{\{X-a:1^{2}\}}&+(q-1)\,\widetilde{A}_{\{X-a:2\}}= \frac{1}{q(q-1)(q^{2}-1)}\,\widehat{A}_{\{X-a:1^{2}\}}+\frac{1}{q(q^{2}-1)}\,\widehat{A}_{\{X-a:2\}} \nonumber\\
 &= \frac{\mathbb{1}_{a=b}}{q(q^{2}-1)} \left(\widehat{A}_{\{X-a:2\}}- \widehat{A}_{\{X-a:1^{2}\}}\right)+ \frac{1}{(q-1)(q^{2}-1)}\,\widehat{A}_{\{X-a: 1, X-b :1\}}.
 \label{qodd2}
 \end{align}
One can also replace \eqref{qodd1} by \eqref{qodd2} thanks to the symbol $\mathbb{1}_{a=b}$.\vspace{2mm}
\item Then, for every $c \in \For_{q}$, including $0$, one obtains $(q-1)$ times the polynomial $X^{2}+cX+ab$. By the discussion of \S\ref{irrpoly}, $\frac{q+1}{2}-\mathbb{1}_{\jacobi{ab}{q}=1}$ values of $c$ give an irreducible polynomial, whence a contribution
\begin{equation}
(q-1)\sum_{c \in I_{ab}} \widetilde{A}_{\{X^{2}+cX+ab:1\}} = \frac{1}{q(q^{2}-1)} \, \sum_{c \in I_{ab}} \widehat{A}_{\{X^{2}+cX+ab:1\}}. \label{qodd5}
\end{equation}
The other values of $c$ correspond to decompositions $X^{2}+cX+ab=(X-\alpha)(X-\beta)$. If $ab$ is not a square, then $\alpha \neq \beta$ and one obtains a contribution
\begin{align}
(q-1)\sum_{c \notin I_{ab}} \widetilde{A}_{\{X-\alpha:1,X-\beta:1\}}&=\frac{q-1}{2}\,\sum_{d \in (\For_{q})^{\times}}  \widetilde{A}_{\{X-ad^{-1}:1,X-bd:1\}} \nonumber\\
&=\frac{1}{2q(q^{2}-1)}\, \sum_{d \in (\For_{q})^{\times}}  \widehat{A}_{\{X-ad^{-1}:1,X-bd:1\}}.\label{qodd3}
\end{align}
If $ab$ is a square then one can have $\alpha=\beta$; in this case the type of the matrix $M=\left(\begin{smallmatrix} a+rs & rb \\ s & b \end{smallmatrix}\right)$ can be either $\{X-\alpha:1^{2}\}$ or $\{X-\alpha:2\} $. The first case is excluded since $rb \neq 0$, and therefore $M-\alpha I \neq 0$. So, if $ab=\delta^{2}$ is a square, then one obtains a contribution
\begin{align}
&(q-1)\,\widetilde{A}_{\{X-\delta:2\}}+(q-1)\,\widetilde{A}_{\{X+\delta:2\}}+ \frac{q-1}{2}\!\!\!\sum_{d \in (\For_{q})^{\times} \setminus \{-b^{-1}\delta,b^{-1}\delta\}}  \!\!\!\widetilde{A}_{\{X-ad^{-1}:1,X-bd:1\}} \nonumber\\
&=\frac{1}{q(q^{2}-1)}\,\left(\widehat{A}_{\{X-\delta:2\}}+\widehat{A}_{\{X+\delta:2\}}\right)+ \frac{1}{2q(q^{2}-1)}\sum_{d \in (\For_{q})^{\times} \setminus \{-b^{-1}\delta,b^{-1}\delta\}}  \!\!\!\widehat{A}_{\{X-ad^{-1}:1,X-bd:1\}}\nonumber \\
&=  \frac{\mathbb{1}_{ab=\delta^{2}}}{2q(q^{2}-1)}\,\left(2\widehat{A}_{\{X-\delta:2\}}-\widehat{A}_{\{X-\delta:1^{2}\}}+2\widehat{A}_{\{X+\delta:2\}}-\widehat{A}_{\{X+\delta:1^{2}\}}\right)\nonumber\\
&\quad+ \frac{1}{2q(q^{2}-1)}\sum_{d \in (\For_{q})^{\times}}  \widehat{A}_{\{X-ad^{-1}:1,X-bd:1\}}.\label{qodd4}
\end{align}
Again, one can replace \eqref{qodd3} by \eqref{qodd4} thanks to the symbol $\mathbb{1}_{ab=\delta^{2}}$.
\vspace{2mm}
\end{enumerate}
So, the remaining part \eqref{noncolinear} of the product $\widehat{A}_{X-a}*\widehat{A}_{X-b}$ is equal to $\frac{(q-1)(q^{2}-1)}{q}$ times the sum of \eqref{qodd2}, \eqref{qodd5} and \eqref{qodd4}, that is to say 
\begin{align}
 \frac{1}{q}\,\widehat{A}_{\{X-a: 1, X-b :1\}}  + \frac{q-1}{q^{2}}&\left(\sum_{c \in I_{ab}} \widehat{A}_{\{X^{2}+cX+ab:1\}}  + \frac{1}{2}\sum_{d \in (\For_{q})^{\times}}  \widehat{A}_{\{X-ad^{-1}:1,X-bd:1\}}\right.\nonumber\\
 &\quad+\frac{\mathbb{1}_{ab=\delta^{2}}}{2}\,\left(2\widehat{A}_{\{X-\delta:2\}}-\widehat{A}_{\{X-\delta:1^{2}\}}+2\widehat{A}_{\{X+\delta:2\}}-\widehat{A}_{\{X+\delta:1^{2}\}}\right)\nonumber\\
&\quad+\left. \vphantom{\frac{1}{2}\sum_{d \in (\For_{q})^{\times}}}\mathbb{1}_{a=b} \left(\widehat{A}_{\{X-a:2\}}- \widehat{A}_{\{X-a:1^{2}\}}\right) \right).\label{noncolinearodd}
\end{align}
Hence, the product $\widehat{A}_{X-a}*\widehat{A}_{X-b}$ is the sum of the quantities \eqref{colinear} and \eqref{noncolinearodd} when $a,b\neq 1$ and $q$ is odd.
\bigskip

\subsection{Other cases and the general product formula}
The exact same discussion applies to the case $a \neq 1$, $b \neq 1$ and $q$ even. Formula \eqref{qodd2} can be kept, and it remains to add for every $c \in \For_{q}$ the $(q-1)$ terms with characteristic polynomial equal to $X^{2}+cX+ab$. Half of the cases yield the contribution \eqref{qodd5}, the only difference with the odd case being that the sum is over a fixed number of values of $c$ (independent of $ab$), namely, $\frac{q}{2}$. For the other values of $c$, remark that $ab$ is always a square in $\For_{q}$, namely, the square of $\delta=(ab)^{\frac{q}{2}}$. Therefore, these over values of $c \notin I_{ab}$ give a contribution
$$ \frac{1}{2q(q^{2}-1)}\,\left(2\widehat{A}_{\{X-\delta:2\}}-\widehat{A}_{\{X-\delta:1^{2}\}}\right) +\frac{1}{2q(q^{2}-1)}\sum_{d \in (\For_{q})^{\times}}  \widehat{A}_{\{X-ad^{-1}:1,X-bd:1\}}.$$
which replaces \eqref{qodd4}. Thus, in the even case, the remaining part of $\widehat{A}_{X-a}*\widehat{A}_{X-b}$ is equal to
\begin{align}
 \frac{1}{q}\,\widehat{A}_{\{X-a: 1, X-b :1\}}  + \frac{q-1}{q^{2}}&\left(\sum_{c \in I_{ab}} \widehat{A}_{\{X^{2}+cX+ab:1\}}  + \frac{1}{2}\sum_{d \in (\For_{q})^{\times}}  \widehat{A}_{\{X-ad^{-1}:1,X-bd:1\}}\right.\nonumber\\
 &\quad\left.+\,\frac{1}{2}\,\left(2\widehat{A}_{\{X-\delta:2\}}-\widehat{A}_{\{X-\delta:1^{2}\}}\right)+ \vphantom{\frac{1}{2}\sum_{d \in (\For_{q})^{\times}}}\mathbb{1}_{a=b} \left(\widehat{A}_{\{X-a:2\}}- \widehat{A}_{\{X-a:1^{2}\}}\right) \right).\label{noncolineareven}
\end{align}
The sum of \eqref{colinear} and \eqref{noncolineareven} gives $\widehat{A}_{X-a}*\widehat{A}_{X-b}$ when $a,b\neq 1$ and $q$ is even.
\bigskip

Finally, when one of the coefficient $a$ or $b$ is equal to $1$, the unique corresponding trivial extension on $(\For_{q})^{2}$ is the identity (partial) isomorphism $\partiso{(\For_{q})^{2}}{(\For_{q})^{2}}{\id}{\id}$. The computation is then trivial and one obtains for instance for $\widehat{A}_{X-a}*\widehat{A}_{X-1}$ with $a\neq 1$ the result
$$(q-1)\,\widehat{A}_{X-a}+\widehat{A}_{\{X-a:1;X-1:1\}}.$$
The first term corresponds to colinear vectors $u$ and $v$, and the second term to non-colinear vectors. When $a$ is also equal to $1$, the same reasoning gives
$$(q-1)\,\widehat{A}_{X-1}+\widehat{A}_{\{X-1:1^{2}\}}.$$
Thus, we have proven the following: 
\begin{theorem}\label{productdegree1}
The product of two generic degree $1$ classes $\widehat{A}_{X-a}$ and $\widehat{A}_{X-b}$ is given by the following formulas. Each time, the writing is decreasing in degree and without terms appearing more than once, but in the sums preceded by a $\frac{1}{2}$, where each term appears twice.
\vspace{2mm}
\begin{itemize}
\item$a \neq b \neq 1$, $q$ odd, $\jacobi{ab}{q}=1$, $ab=\delta^{2}$:
\begin{align*}
& \frac{q-1}{q^{2}}\left(\widehat{A}_{\{X-\delta:2\}}+\widehat{A}_{\{X+\delta:2\}} + \sum_{c \in I_{ab}} \widehat{A}_{\{X^{2}+cX+ab:1\}}  + \frac{1}{2}\sum_{d \in (\For_{q})^{\times}\setminus \{1,b^{-1}a,\pm b^{-1}\delta\}}\!\!\!\!  \widehat{A}_{\{X-ad^{-1}:1,X-bd:1\}} \right) \\
 &+\frac{2q-1}{q^{2}}\,\widehat{A}_{\{X-a: 1, X-b :1\}}+( q-1)\,\widehat{A}_{X-ab} . 
\end{align*}
 \vspace{1mm}
\item$a \neq b \neq 1$, $q$ odd, $\jacobi{ab}{q}=-1$:
\begin{align*}
&\frac{q-1}{q^{2}}\left(\sum_{c \in I_{ab}} \widehat{A}_{\{X^{2}+cX+ab:1\}}  + \frac{1}{2}\sum_{d \in (\For_{q})^{\times}\setminus\{1,b^{-1}a\}}  \widehat{A}_{\{X-ad^{-1}:1,X-bd:1\}} \right)\\
&+ \frac{2q-1}{q^{2}}\,\widehat{A}_{\{X-a: 1, X-b :1\}}  +( q-1)\,\widehat{A}_{X-ab}  .
\end{align*}\vspace{1mm}
\item $a \neq b \neq 1$, $q$ even, $\delta=(ab)^{\frac{q}{2}}$: 
\begin{align*}
& \frac{q-1}{q^{2}}\left(\widehat{A}_{\{X-\delta:2\}}+\sum_{c \in I_{ab}} \widehat{A}_{\{X^{2}+cX+ab:1\}}  + \frac{1}{2}\sum_{d \in (\For_{q})^{\times}\setminus \{1,b^{-1}a,b^{-1}\delta\}}  \widehat{A}_{\{X-ad^{-1}:1,X-bd:1\}} \right)\\
&+ \frac{2q-1}{q^{2}}\,\widehat{A}_{\{X-a: 1, X-b :1\}}  +( q-1)\,\widehat{A}_{X-ab}  .
\end{align*}\vspace{1mm}
\item $a=b \neq 1$, $q$ odd:
\begin{align*}  &\frac{q-1}{q^{2}}\left(2\widehat{A}_{\{X-a:2\}} +\widehat{A}_{\{X+a:2\}}+\sum_{c \in I_{a^{2}}} \widehat{A}_{\{X^{2}+cX+a^{2}:1\}}  + \frac{1}{2}\sum_{d \in (\For_{q})^{\times}\setminus\{\pm 1\}}  \widehat{A}_{\{X-ad^{-1}:1,X-ad:1\}}\right)\\
 &+ \frac{1}{q^{2}}\,\widehat{A}_{\{X-a: 1^{2}\}} +( q-1)\,\widehat{A}_{X-a^{2}} .
\end{align*}
 \vspace{1mm}
\item $a=b \neq 1$, $q$ even:
\begin{align*}
 &\frac{q-1}{q^{2}}\left(2\widehat{A}_{\{X-a:2\}} + \sum_{c \in I_{a^{2}}} \widehat{A}_{\{X^{2}+cX+a^{2}:1\}}  + \frac{1}{2}\sum_{d \in (\For_{q})^{\times}\setminus \{1\}}  \widehat{A}_{\{X-ad^{-1}:1,X-ad:1\}}\right)\\
 &+\frac{1}{q^{2}}\,\widehat{A}_{\{X-a: 1^{2}\}}  +(q-1)\,\widehat{A}_{X-a^{2}}.
\end{align*}
 \vspace{1mm}
\item $a \neq 1$, $b=1$: $\widehat{A}_{\{X-a:1;X-1:1\}}+(q-1)\,\widehat{A}_{X-a}.$\vspace{1mm}
\item $a = b=1$: $\widehat{A}_{\{X-1:1^{2}\}}+(q-1)\,\widehat{A}_{X-1}.$\vspace{1mm}
\end{itemize}
\end{theorem}
\bigskip

The theorem yields readily the product of two completed conjugacy classes of degree $1$ for any $n \in \N$, since 
$$\Pi_{n}\left(\widehat{A}_{\bbmu}\right)=q^{n(2k_{1}-k)} \,\frac{q^{2k(k-k_{1})}\,(q^{-1})_{k}\,(q^{-1})_{n-k+k_{11}}}{(q^{-1})_{k_{11}}\,(q^{-1})_{n-k}}\,\frac{C_{\bbmu\uparrow^{n}}}{\card\,C_{\bbmu}} $$
for any polypartition $\bbmu$. More specifically,
\begin{align}
\Pi_n\left(\widehat{A}_{\{X-1:1\}}\right)&=(q^n -1)\,C_{\emptyset\uparrow^n}\nonumber\\
\Pi_n\left(\widehat{A}_{\{X-1:1^2\}}\right)&=(q^n-1)(q^n-q)\,C_{\emptyset\uparrow^n}\nonumber\\
\Pi_n\left(\widehat{A}_{\{X-a:1\}}\right) &=q^{1-n} \,(q-1) \, C_{\{X-a:1\}\uparrow^n} \nonumber\\
\Pi_n\left(\widehat{A}_{\{X-a:2\}}\right)&=q^{5-2n}\,(q-1) \,C_{\{X-a:2\}\uparrow^n}\nonumber
\end{align}
\begin{align}
\Pi_n\left(\widehat{A}_{\{X-a:1^2\}}\right)&=q^{5-2n}\,(q-1)(q^2-1)\,C_{\{X-a:1^2\}\uparrow^n}\nonumber\\
\Pi_n\left(\widehat{A}_{\{X-a:1;X-1:1\}}\right)&=q^{2-n}\,(q-1)(q^{n-1}-1)\,C_{\{X-a:1\}\uparrow^n}\nonumber\\
\Pi_n\left(\widehat{A}_{\{X-a:1;X-b:1\}}\right)&=q^{4-2n}\,(q-1)^2\,C_{\{X-a:1;X-b:1\}\uparrow^n}\nonumber\\
\Pi_n\left(\widehat{A}_{\{X^2+aX+b:1\}}\right)&=q^{4-2n}\,(q^2-1)\,C_{\{X^2+aX+b:1\}\uparrow^n}.\label{last1}
\end{align}
Applying these formulas to the seven cases of Theorem \ref{productdegree1}, one obtains the expansion in completed conjugacy classes of 
\begin{equation} \Pi_n\left(\widehat{A}_{X-a}*\widehat{A}_{X-b}\right) = (q^n-1)^2\,\frac{C_{\{X-a\}\uparrow^n}*C_{\{X-b\}\uparrow^n}}{\card\, C_{\{X-a\}\uparrow^n}\times \card \,C_{\{X-b\}\uparrow^n}}.
\label{last2}
\end{equation}
 So for instance, if $q$ is odd and $a \neq b \neq 1$, $\jacobi{ab}{q}=1$, $ab=\delta^{2}$, then by projection by $\Pi_n$ of the first case of Theorem \ref{productdegree1}, the product of completed conjugacy classes $C_{\{X-a\}\uparrow^n}*C_{\{X-b\}\uparrow^n}$ in $\C\GL(n,\For_q)$ is given by:\vspace{1mm}
\begin{align*}
C_{\{X-a\}\uparrow^n}*C_{\{X-b\}\uparrow^n}&= q\,C_{\{X-\delta:2\}\uparrow^n}+q\,C_{\{X+\delta:2\}\uparrow^n} +(2q-1) \,C_{\{X-a: 1, X-b :1\}}+q^{n-1}\,C_{\{X-ab:1\}\uparrow^n}  \\
&\quad+\sum_{c \in I_{ab}} (q+1)\,C_{\{X^{2}+cX+ab:1\}\uparrow^n}  \\
&\quad+ \frac{1}{2}\sum_{d \in (\For_{q})^{\times}\setminus \{1,b^{-1}a,\pm b^{-1}\delta\}}\!\!\!\!  (q-1)\,C_{\{X-ad^{-1}:1,X-bd:1\}\uparrow^n}  . 
\end{align*}
Each structure coefficient is indeed a polynomial in $q^n$, here of degree $0$ or $1$. Notice that this formula \emph{cannot} be specialized to the case $a=1$ or $b=1$ (in this case, the sixth and seventh cases of Theorem \ref{productdegree1} give indeed the much simpler expansion $C_{\{X-a:1\}\uparrow^n} * C_{\{X-1:1\}\uparrow^n}=C_{\{X-a:1\}\uparrow^n}$, as can be expected). The other cases are similar and easy computations after Theorem \ref{productdegree1} and Formulas \eqref{last1} and \eqref{last2}.
\bigskip

From these computations, it becomes clear that the determination of the whole multiplication table of $\mathscr{Z}(\infty,\For_{q})$ is not possible. However, it might exist simple rules to determine certain structure polynomials $p_{\bbmu\bbnu}^{\bblambda}(q^{n})$, especially those with $|\bblambda|=|\bbmu|+|\bbnu|$. If it exists, a general rule for these coefficients ``of maximal degree'' probably involves deeply the arithmetics and Galois theory of polynomials over $\For_{q}[X]$.\bigskip

\section*{Acknowledgements}
The author would like to thank Valentin F\'eray for several discussions around the construction of ``generic'' center group algebras; and the anonymous referee for his valuable comments, that allowed to improve the quality of the paper. 
\bigskip


\end{document}